\documentclass[review]{elsarticle}
\usepackage{}

\usepackage[latin1]{inputenc} % arXiv±ØÐë¼ÓÕâ¸ö£¬·ñÔòÍ¶¸å±àÒë²»³É¹¦
\usepackage{listings, tcolorbox}% arXiv±ØÐë¼ÓÕâ¸ö£¬·ñÔòÍ¶¸å±àÒë²»³É¹¦
\usepackage{float}
\usepackage{lineno}
\usepackage{amssymb}
\usepackage{color}
\usepackage{amsmath}
\usepackage{tikz}
\usepackage{amsfonts}
\usepackage{mathrsfs}
\usepackage{amsthm}

\usepackage{bbding}
\usepackage{mathrsfs}
\DeclareMathOperator*{\re}{Re}
\DeclareMathOperator*{\im}{Im}
\DeclareMathOperator*{\tr}{Tr}
\DeclareMathOperator*{\res}{Res}
\usepackage{amsmath, amsfonts, amssymb, mathrsfs, txfonts}
\usepackage{graphicx, subfigure}
\usepackage{wasysym}
\usepackage{color}
\usepackage{bm}
\usepackage{enumerate}

\usepackage{pifont}
\usepackage{txfonts}
\usepackage{amsmath}
\usepackage{graphicx}
\usepackage{amsfonts}
\usepackage{amssymb}
\usepackage{mathrsfs,psfrag,eepic,epsfig}
\usepackage{epstopdf}
\usepackage{hyperref}

\modulolinenumbers[5]

\journal{Journal of \LaTeX\ Templates}

\makeatletter \@addtoreset{equation}{section}

\newtheorem{thm}{Theorem}[section]

\newtheorem{prop}[thm]{Proposition}
\theoremstyle{definition}

\newtheorem{rem}[thm]{Remark}

\newtheorem{RHP}[thm]{RH Problem}

\renewcommand{\baselinestretch}{1.25}
\begin{document}

\begin{frontmatter}

\title{On long-time asymptotics to the nonlocal Lakshmanan -Porsezian-Daniel equation with step-like initial data \tnoteref{mytitlenote}}
\tnotetext[mytitlenote]{%Project supported by the Fundamental Research Fund for the Central Universities under the grant No. 2017XKQY101.\\
%\hspace*{3ex}$^{*}$
Corresponding author.\\
\hspace*{3ex}\emph{E-mail addresses}: sftian@cumt.edu.cn,
shoufu2006@126.com (S. F. Tian) }

\author{Wen-yu zhou, Shou-Fu Tian$^{*}$ and Xiao-fan Zhang}
\address{School of Mathematics, China University of Mining and Technology,  Xuzhou 221116, People's Republic of China}

\begin{abstract}
In this work, the nonlinear steepest descent method is employed to study the long-time asymptotics of the integrable nonlocal Lakshmanan-Porsezian-Daniel (LPD) equation with a step-like initial data: $q_{0}(x)\rightarrow0$ as $x\rightarrow-\infty$ and $q_{0}(x)\rightarrow A$ as $x\rightarrow+\infty$, where $A$ is an arbitrary positive constant. Firstly, we develop a matrix Riemann-Hilbert (RH) problem to represent the Cauchy problem of LPD equation. To remove the influence of singularities in this RH problem, we introduce the Blaschke-Potapov (BP) factor, then the original RH problem can be transformed into a regular RH problem which can be solved by the parabolic cylinder functions. Besides, under the nonlocal condition with symmetries $x\rightarrow-x$ and $t\rightarrow t$, we give the asymptotic analyses at $x>0$ and $x<0$, respectively. Finally, we derive the long-time asymptotics of the solution $q(x,t)$ corresponding to the complex case of three stationary phase points generated by phase function.
\end{abstract}

\begin{keyword}
Nonlocal LPD equation, Step-like initial data,  Riemann-Hilbert Problem, Long-time asymptotics, Nonlinear steepest descent method.
\end{keyword}

\end{frontmatter}

%\linenumbers
\tableofcontents
%\newpage

\section{Introduction}
As we know, the research of nonlinear partial differential equation is a significant part of nonlinear science. Many physical phenomena can be characterized by nonlinear waves. Therefore, many experts and scholars have done a series of valuable works on these nonlinear systems. For example, some nonlinear equations such as the nonlinear Schr\"{o}dinger (NLS) equation \cite{Biondini-2014}, modified Korteweg-de Vries (mKdV) equation \cite{Ji-Zhu-2017} and Sasa-Satsuma equation \cite{Geng-2018} have many important applications. To further study the effects of higher-order perturbations, many modified and generalized NLS equations have drawn attention \cite{Clarkson-1990, Brizhik-2003, Tian-JDE-2017, Tian-PAMS-2018}. One of the integrable systems with higher-order dispersion and nonlinear terms is the Lakshmanan-Porsezian-Daniel (LPD) equation \cite{LPD-1988} takes the form as follows
\begin{align}\label{LPD}
q_t(x,t)+\frac{1}{2}iq_{xx}(x,t)-iq^2(x,t)q(x,t)-\gamma H[q(x,t)]=0,
\end{align}
with
\begin{align*}
H[q(x,t)]&=-iq_{xxxx}(x,t)+6iq(x,t)q_{x}^2(x,t)+4iq(x,t)q_{x}(x,t)q_{x}(x,t)\\
&+8iq^{2}(x,t)q_{xx}(x,t)+2iq^2(x,t)q_{xx}(x,t)-6iq^2(x,t)q^3(x,t),
\end{align*}
where $\gamma$ is an arbitrary positive real parameter and the subscripts represent partial differentiations.

Recently, integrable nonlocal equations have attracted extensive attention. Unlike classical local equations, the potential induced by the nonlinear term of the nonlocal equation is Parity-Time (PT) symmetry. It was first introduced by Ablowitz and Musslimani to study the nonlocal NLS equation
\begin{align}\label{NLS}
iq_t(x,t)=q_{xx}(x,t)\pm 2q(x,t)\overline{q}(-x,t)q(x,t)=0,
\end{align}
which possesses a Lax pair and a infinite number of conservation laws \cite{Ablowitz-Musslimani-2013}. $\overline{q}(-x,t)$ denotes the complex conjugation of $q(-x,t)$. The PT symmetry is a special reduction from the famous AKNS system as $r(x,t)=\overline{q}(-x,t)$. It is worth noting that the equation with PT symmetry is invariant under parity P and time reversal T. Its space reversal operation is defined by $x\rightarrow-x$, and its time reversal operation is defined by $t\rightarrow-t$ \cite{Bender-Boettcher-1998}. Compared with the classical local integrable equations, since the nonlocal nonlinear term $\overline{q}(-x,t)$ replaces $\overline{q}(x,t)$, there generate many important properties which have profound significances in mathematics, physics and classical optics. At present, many nonlocal equations have been proposed, such as nonlocal NLS equation \cite{Fokas-2016}, nonlocal mKdV equation \cite{Liang-Zhou-2017}, nonlocal Sine-Gordon equation \cite{Ablowitz-Feng-2018}, nonlocal Camassa-Holm equation \cite{Heredero-Reyes-2009} and so on.

In this work, we further study the integrable nonlocal LPD equation
\begin{align}\label{e1}
\left\{
     \begin{aligned}
&q_{t}(x,t)+\frac{1}{2}iq_{xx}(x,t)-iq^{2}(x,t)r(x,t)-\gamma H[q(x,t)]=0,& \quad &x\in\mathbb{R},t>0,\\
&q(x,0)=q_{0}(x),& \quad &x\in\mathbb{R},
     \end{aligned}
\right.
\end{align}
with
\begin{align*}
H[q(x,t)]&=-iq_{xxxx}(x,t)+6ir(x,t)q_{x}^2(x,t)+4iq(x,t)q_{x}(x,t)r_{x}(x,t)\\
&+8ir(x,t)q(x,t)q_{xx}(x,t)+2iq^2(x,t)r_{xx}(x,t)-6ir^2(x,t)q^3(x,t),
\end{align*}
where the $\gamma$ is defined same as above. In addition, we have $r(x,t)=\sigma \overline{q}(-x,t),\sigma=\pm1$. When $\sigma=1$, $r(x,t)=\overline{q}(-x,t)$ corresponds to the defocusing case. When $\sigma=-1$, $r(x,t)=-\overline{q}(-x,t)$ corresponds to the focusing case. We mainly pay close attention to the latter case.
Besides, we study the initial value problem for nonlocal LPD equation with a step-like initial data
\begin{align}\label{e2,3}
&q_{0}(x)\rightarrow 0, \quad x\rightarrow -\infty,\\
&q_{0}(x)\rightarrow A, \quad x\rightarrow +\infty,
\end{align}
where A is an arbitrary constant which is always positive and the solution $q(x,t)$ of equation \eqref{e1} also satisfies boundary conditions as follows, where $t$ is a positive value
\begin{subequations}\label{e4}
\begin{align}
&q(x,t)=o(1),  \quad\quad x\rightarrow -\infty, \label{e4a}\\
&q(x,t)=A+o(1),   \quad x\rightarrow +\infty. \label{e4b}
\end{align}
\end{subequations}

The LPD equation was first proposed by Lakshmanan, Porsezian and Daniel \cite{LPD-1988}. They considered the one-dimensional classical Heisenberg ferromagnetic spin system and applied a multiple-scaling method to find the perturbed soliton solution of the non-integrable case. Besides, it can be a model to describe the propagation and interaction of the ultrashort pulses in high-speed optical fiber transmission system and has many other extensive applications. Recently, many scolars promote it and discover a lot of interesting properties and phenomena. The Lax pair and an infinite number of conservation laws have been given \cite{Zhang-2009}. At the same time, the integrability was proved and the multi-soliton solutions were generated in a recursive manner by using the Darboux transformation. The breathers and rogue waves solutions for the LPD equation have been given \cite{Guo-2013,Wang-2012}. The dynamical behavior of the rational soliton solutions and self-potential function of integrable nonlocal LPD equation were obtained by using the degenerate Darboux \cite{Liu-2016}. Through using the Darboux transformations, the localized wave solutions of the nonlocal LPD equation have been studied \cite{Yang-2020}. the initial-boundary value problem of LPD equation on the half-line, which has the physical applications in Heisenberg ferromagnetic spin has been analyzed \cite{Wang-Liu-2020}. At the same time, the soliton solution obtained by inverse scattering transform (IST) has also been given \cite{Xun-2020}. The long-time asymptotic behavior of the LPD equation has been investigated through the nonlinear steepest descent analysis in the Schwartz space \cite{Wang-Liu-2022}. By the nonlinear steepest descent method of Deift and Zhou, the explicit long-time asymptotic formula of the nonlocal LPD equation has been derived \cite{Peng-Chen-2023}. The LPD equation has many applications in nonlinear optics and physics, so it has great research value and significance.

The study of asymptotic solutions to nonlinear dispersion equations is a hot topic. It was first brought into forcing with the IST method by Manakov \cite{Manakov-1974}. After that, Zakharov and Manakov gave the large-time asymptotic solutions of the NLS equation with decaying initial value by this method \cite{Zakharov-Manakov-1976}. Subsequently, Deift and Zhou, inspired by this work, developed a nonlinear steepest descent method to obtain the long-time asymptotic behaviors of the solution for the mKdV equation through simplified the original Riemann-Hilbert (RH) problem to a model that the solution can be calculated by parabolic cylinder functions \cite{Deift-Zhou-1993}. Through simplifying the original oscillation RH problem into the form which can be solved by some deformations, then using the nonlinear steepest descent method, we can analyze the long-time asymptotic behavior of the solution to the integrable equation. There are many properties about the large-time asymptotic behavior by this method were obtained \cite{Xu-Fan-2015,Rybalko-2019}.

As we know, the higher-order NLS equations with non-zero boundary conditions have been studied \cite{Chen-2019}. Since the LPD equation is a special case of the fifth-order NLS equation, its results can be obtained by simplification.
The long-time asymptotics for the nonlocal NLS equation with decaying boundary conditions and step-like initial data have been studied \cite{Rybalko-JMP-2019,Rybalko-JDE-2021}. Moreover, they also present the long-time asymptotics for a one-parameter family curved wedges and a family of nonlocal NLS equation with step-like initial data \cite{Rybalko-SAPM-2021,Rybalko-CMP-2021,Rybalko-JMPAG-2021}.
In addition, the long-time asymptotics for the nonlocal mKdV equation with step-like initial data have been investigated \cite{Xu-Fan-2023}.
The mKdV equation \cite{Minakov-2011,Grava-Minakov-2020,Liu-2021,Xu-Fan-2022}, Camassa-Holm equation \cite{Minakov-2015, Yang-Fan-2022, Minakov-2016} and DNLS equation with step-like initial data have been researched \cite{Xu-Fan-2013}.

It can be seen that there are many literatures study the asymptotic behaviors of integrable systems, especially for equations with step-like type initial values. However, there are relatively little researches about the asymptotic solutions of nonlocal equations with step-like initial data. We will conduct our research on the basis of literatures \cite{Wang-Liu-2022,Peng-Chen-2023}.
In our work, we employ the nonlinear steepest decent method to study the long-time behavior of the nonlocal LPD equation with step-like initial data.
It should be noted that the LPD equation with step-like initial value conditions which are not decaying as $x\rightarrow\infty$ and will produce singularity at the point $\xi=0$. Besides, for the nonlocal LPD equation, we need to consider its special symmetry and the impact on the construction of asymptotic solutions. Moreover, because the phase function $\theta(\xi,\mu)$ of nonlocal LPD equation has a high-order about $\xi$, which corresponding to the complex case of three stationary phase points generated by phase function. This will bring certain difficulties to the analysis process and make the construction of asymptotic solutions more complex.\\

The frame of the work is arranged as: In section 2, We have performed spectral analysis on the nonlocal LPD equation and give the analytic, symmetric and asymptotic properties of the eigenfunctions and scattering data. In addition to this, we mention the special case of the scattering matrix under pure-step initial data condition and construct the RH problem.
In section 3, Through the nonlinear steepest descent method, we construct the long-time asymptotics of $q(x,t)$. In section 3.1, we decompose the jump matrix $J(x,t,\xi)$ into the matrices consist by the upper triangle and lower triangle. In section 3.2, we perform the second RH deformation to transform the contour and make the jump matrices decline to identity $I$ for the large-$t$. In section 3.3, we introduce the BP factor to transform the RH problem into a regular RH problem. Then the rough estimate about $q(x,t)$ is obtained. In section 3.4, we give the local models near the saddle points and solve them by the parabolic cylinder functions. In section 3.5, by the Beals-coifman theory, we gain the error analysis of regular RH problem. Subsequently, the long-time asymptotics of the solutions of LPD equation at cases $x>0$ and $x<0$ are attained, respectively.

\section{Spectral analysis and the RH problem}
This section aims to conduct spectral analyses of eigenfunctions and scattering data, which involve in analytic, symmetric and asymptotic properties.

To derive the nonlocal LPD equation \eqref{e1}, we consider the following Lax pair
\begin{align}\label{q1}
\left\{
     \begin{aligned}
     \phi_{x}&=M\phi, \quad M=-i\xi\sigma_{3}+Q,\\
     \phi_{t}&=N\phi, \quad N=i\xi^2\sigma_{3}-\xi Q+\frac{1}{2}V+\gamma V_{P},
     \end{aligned}
\right.
\end{align}
with
\begin{equation}
\hspace{0.5cm}Q=\left(\begin{array}{cc}
0 & q(x,t)   \\
-\bar{q}(-x,t) & 0
\end{array}\right),\hspace{0.5cm}V=\left(\begin{array}{cc}
-iq(x,t)\bar{q}(-x,t) & -iq_{x}(x,t)  \\
-i\bar{q}_{x}(-x,t) &iq(x,t)\bar{q}(-x,t)
\end{array}\right),\nonumber
\end{equation}
\begin{equation}
\hspace{0.5cm}V_{P}=\left(\begin{array}{cc}
iA_{P}(x,t) & B_{P}(x,t)   \\
-C_{P}(x,t) & -iA_{P}(x,t)
\end{array}\right),\hspace{0.5cm}\sigma_{3}=\left(\begin{array}{cc}
1 & 0  \\
0 & -1
\end{array}\right),\nonumber
\end{equation}
\begin{equation}
\begin{split}
A_{P}(x,t)=&-8\xi^4+4\bar{q}(-x,t)q(x,t)\xi^2+2i\bar{q}(-x,t)q_{x}(x,t)\xi
+2iq(x,t)\bar{q}(-x,t)\xi\\
&-3q^2(x,t)\bar{q}^2(-x,t)+q_{x}(x,t)\bar{q}_{x}(-x,t)-q(x,t)\bar{q}(-x,t)
-\bar{q}(-x,t)q_{xx}(x,t),
\end{split}
\nonumber
\end{equation}
\begin{equation}
\begin{split}
B_{P}(x,t)= &8q(x,t)\xi^3+4iq_{x}(x,t)\xi^2-2q_{xx}(x,t)\xi-4\bar{q}(-x,t)q^2(x,t)
-iq_{xxx}(x,t)\\
&-6iq(x,t)\bar{q}(-x,t)q_{x}(x,t),
\end{split}
\nonumber
\end{equation}
\begin{equation}
\begin{split}
C_{P}(x,t)=&-8\bar{q}(-x,t)\xi^3-4i\bar{q}_{x}(-x,t)\xi^2-2\bar{q}_{xx}(-x,t)\xi
-4\bar{q}^2(-x,t)q(x,t)\xi\\
&+i\bar{q}_{xxx}(-x,t)+6i\bar{q}(-x,t)\bar{q}_{x}(-x,t)q(x,t),
\end{split}
\nonumber
\end{equation}
where $\phi(x,t,\xi)$ is a $2\times2$ matrix-valued function, the potential function $q(x,t)$ is a complex function and $\xi\in\mathbb{C}$ is a spectra parameter. According to the compatibility condition $M_{t}-N_{x}+[M,N]=0$, we can get the LPD equation \eqref{e1}.

Considering the asymptotic spectral problem as $x\rightarrow\pm\infty$ of Lax pair \eqref{q1}, we have
\begin{align}\label{q2}
\left\{
     \begin{aligned}
     \phi_{\pm x}&=M_{\pm}(\xi)\phi_{\pm}, \quad M_{\pm}(\xi)=\displaystyle\lim_{x\rightarrow\pm\infty}M(x,t,\xi)
     =-i\xi\sigma_{3}+Q_{\pm},\\
     \phi_{\pm t}&=N_{\pm}(\xi)\phi_{\pm}, \quad N_{\pm}(\xi)=\displaystyle\lim_{x\rightarrow\pm\infty}N(x,t,\xi)
     =(-\xi+8\xi^3\gamma)M_{\pm}(\xi),
     \end{aligned}
\right.
\end{align}
with
\begin{equation}\label{q3}
\hspace{0.5cm} Q_{+}=\left(\begin{array}{cc}
0 & A   \\
0 & 0
\end{array}\right),\hspace{0.5cm}\ Q_{-}=\left(\begin{array}{cc}
0 & 0  \\
-A & 0
\end{array}\right).
\end{equation}
Then the Jost eigenfunctions $\phi_{\pm}(x,t,\xi)$ are defined as follows
\begin{equation}\label{q4}
\phi_{\pm}(x,t,\xi)\rightarrow
L_{\pm}(\xi)e^{-i\xi\sigma_{3}x+i(\xi^2-8\xi^4\gamma)\sigma_{3}t},
\quad x\rightarrow\pm\infty,
\end{equation}
where $L_{\pm}$ are matrices to make $M_{\pm}$ satisfy the following diagonalization
\begin{equation}\label{q5}
\hspace{0.5cm} L_{+}(\xi)=\left(\begin{array}{cc}
1 & \frac{A}{2i\xi}   \\
0 & 1
\end{array}\right),\hspace{0.5cm}\ L_{-}(\xi)=\left(\begin{array}{cc}
1 & 0  \\
\frac{A}{2i\xi} & 1
\end{array}\right),
\end{equation}
\begin{equation}\label{q6}
M_{\pm}L_{\pm}=L_{\pm}(-i\xi\sigma_{3}),
\quad N_{\pm}L_{\pm}=L_{\pm}(i\xi^2\sigma_{3}-8i\xi^4\gamma\sigma_{3}).
\end{equation}

Next, we consider the new matrix spectral functions $\psi_{\pm}(x,t,\xi)$
\begin{equation}\label{q07}
\phi_{\pm}(x,t,\xi)=\psi_{\pm}(x,t,\xi)e^{-i\xi\sigma_{3}x+i(\xi^2-8\xi^4\gamma)\sigma_{3}t},
\end{equation}
and
\begin{equation}\label{q8}
\psi_{\pm}(x,t,\xi)\rightarrow L_{\pm}(\xi), \quad x\rightarrow\pm\infty.
\end{equation}

Then we introduce the equivalent lax pair to \eqref{q1}
\begin{align}\label{q9}
\left\{
     \begin{aligned}
     (&L_{\pm}^{-1}\psi_{\pm})_{x}-i\xi[L_{\pm}^{-1}\psi_{\pm},\sigma_{3}]
     =L_{\pm}^{-1}\Delta Q_{\pm}\psi_{\pm},\\
     (&L_{\pm}^{-1}\psi_{\pm})_{t}+(i\xi^2-8i\xi^4\gamma)[L_{\pm}^{-1}\psi_{\pm},\sigma_{3}]
     =L_{\pm}^{-1}\Delta U_{\pm}\psi_{\pm},
     \end{aligned}
\right.
\end{align}
where \ $U=N-(i\xi^2\sigma_{3}-8i\xi^{4}\gamma\sigma_{3})$, \ $\Delta X_{\pm}=X-X_{\pm}$, \ $[\wedge,\sigma_{3}]=\wedge\sigma_{3}-\sigma_{3}\wedge$.

By choosing particular paths, both $\psi_{-}(x,t,\xi)$ and $\psi_{+}(x,t,\xi)$ can be uniquely determined by the following Volterra integral equations
\begin{align}\label{q10}
\left\{
     \begin{aligned}
     \psi_{-}(x,t,\xi)&=L_{-}(\xi)+\int_{-\infty}^{x}G_{-}(x,y,t,\xi)
     (Q(y,t)-Q_{-}(y,t))\psi_{-}(x,t,\xi)e^{i\xi(x-y)\sigma_{3}}dy,\\
     \psi_{+}(x,t,\xi)&=L_{+}(\xi)+\int_{\infty}^{x}G_{+}(x,y,t,\xi)
     (Q(y,t)-Q_{+}(y,t))\psi_{+}(x,t,\xi)e^{i\xi(x-y)\sigma_{3}}dy,
     \end{aligned}
\right.
\end{align}
where $G_{\pm}(x,y,t,\xi)=\phi_{\pm}(x,t,\xi)[\phi_{\pm}(y,t,\xi)]^{-1}
=L_{\pm}(x,t,\xi)e^{-i\xi(x-y)\sigma_{3}}L_{\pm}^{-1}(y,t,\xi)$.

Since $\phi_{\pm}(x,t,\xi)$ are the solutions of lax pair \eqref{q1}, which are systems of first-order linear homogeneous equation. There exists a matrix $S(\xi)$ independent of variable $x$ and $t$
\begin{align}\label{q11}
\phi_{-}(x,t,\xi)=\phi_{+}(x,t,\xi)S(\xi),\quad \xi\in \mathbb{R}\backslash{\left\{0\right\}}.
\end{align}
Substitute \eqref{q07} into \eqref{q11}
\begin{align}\label{q12}
\psi_{-}(x,t,\xi)=\psi_{+}(x,t,\xi)e^{-i\theta\sigma_{3}}S(\xi)e^{i\theta\sigma_{3}},\quad \xi\in \mathbb{R}\backslash{\left\{0\right\}},
\end{align}
where $\theta=\xi x-(\xi^{2}-8\xi^{4}\gamma)t)$ and the matrix is defined as follows
\begin{align}\label{q13}
\hspace{0.5cm} S(\xi)=\left(\begin{array}{cc}
s_{11}(\xi) & s_{12}(\xi)   \\
s_{21}(\xi) & s_{22}(\xi)
\end{array}\right).
\end{align}

\begin{prop}\label{pr1}
Matrices $\psi_{\pm}(x,t,\xi)$ and scattering data of $S(\xi)$ satisfy the following symmetry relations
\begin{enumerate}[(i)]
\item With regard to $\psi_{\pm}(x,t,\xi)$,
\begin{align}\label{q14}
\sigma_{1}\overline{\psi_{-}(-x,t,-\overline{\xi})}\sigma_{1}=\psi_{+}(x,t,\xi),\quad \xi\in\mathbb{R}\backslash{\left\{0\right\}}.
\end{align}
\item With regard to $s_{ij}(\xi)$, $i,j=1,2$,
\begin{align}\label{q15}
\begin{split}
s_{ij}(\xi)&=\overline{s_{ij}(-\overline{\xi})},\quad i=j, \\
s_{ij}(\xi)&=-\overline{s_{ij}(-\overline{\xi})},\quad i\neq j,
\end{split}
\end{align}
\end{enumerate}
where
$\sigma_{1}=\left(
          \begin{matrix}
            0 & 1\\
            1 & 0\\
          \end{matrix}
          \right)$.
\end{prop}
\begin{proof}
 About $(i)$, using Lax pair \eqref{q9} and relation $\sigma_{1}\overline{Q(-x,t)}\sigma_{1}=-Q(x,t)$ we can easily verify that matrices $\psi_{\pm}(x,t,\xi)$ satisfy the symmetry relation $\sigma_{1}\overline{\psi_{-}(-x,t,-\overline{\xi})}\sigma_{1}=\psi_{+}(x,t,\xi)$.

About $(ii)$, using relation \eqref{q07}, the symmetry relation of $\phi_{\pm}(x,t,\xi)$ can be easily acquired
\begin{align}\label{q16}
\sigma_{1}\overline{\phi_{-}(-x,t,-\overline{\xi})}\sigma_{1}=\phi_{+}(x,t,\xi),\quad \xi\in\mathbb{R}\backslash{\left\{0\right\}}.
\end{align}
Then using \eqref{q11}, we get the symmetry relation of matrix $S(\xi)$
\begin{align}\label{q17}
\sigma_{1}\overline{{S(-\overline{\xi})}^{-1}}\sigma_{1}=S(\xi), \quad \xi\in\mathbb{R}\backslash{\left\{0\right\}}.
\end{align}

Based on \eqref{q17}, we have the symmetry relations $s_{11}(\xi)=\overline{s_{11}(-\overline{\xi})}$, $s_{22}(\xi)=\overline{s_{22}(-\overline{\xi})}$ and $s_{21}(\xi)=-\overline{s_{12}(-\overline{\xi})}$.
Then we redefine the matrix $S(\xi)$ as follows
\begin{align}\label{q18}
S(\xi)=\left(
          \begin{matrix}
            a_{1}(\xi) & b(\xi)\\
            -\overline{b(-\overline{\xi})} & a_{2}(\xi)\\
          \end{matrix}
          \right),\quad \xi\in\mathbb{R}\backslash{\left\{0\right\}}.
\end{align}
\end{proof}

\begin{prop}\label{pr2}
Matrices $\psi_{\pm}(x,t,\xi)$ and scattering data of $S(\xi)$ satisfy the following analytic relations
\begin{enumerate}[(i)]
\item $\psi_{-}^{(1)}$ and $\psi_{+}^{(2)}$ are analytic in $\xi\in\mathbb{C}^{+}$ and continuous in $\overline{\mathbb{C}^{+}}\backslash{\left\{0\right\}}$; $\psi_{-}^{(2)}$ and $\psi_{+}^{(1)}$ are analytic in $\xi\in\mathbb{C}^{-}$ and continuous in $\overline{\mathbb{C}^{-}}$.
\item $a_{1}(\xi)$ is analytic in $\xi\in\mathbb{C}^{+}$ and continuous in $\overline{\mathbb{C}^{+}}\backslash{\left\{0\right\}}$; $a_{2}(\xi)$ is analytic in $\xi\in\mathbb{C}^{-}$ and continuous in $\overline{\mathbb{C}^{-}}$; $b(\xi)$ is continuous in $\xi\in\mathbb{R}$.
\end{enumerate}
where $\mathbb{C}^{+}=\left\{{\xi\in\mathbb{C}\mid Im\xi>0}\right\}$ and $\mathbb{C}^{-}=\left\{{\xi\in\mathbb{C}\mid Im\xi<0}\right\}$ stand for the upper half plane and the lower half plane of the complex plane, respectively. $\psi_{\pm}^{(k)}$ denotes the $k$-th column of $\psi_{\pm}$.
\end{prop}
\begin{proof}
About $(i)$, using Volterra integral \eqref{q10}, the analytical relations of $\psi_{\pm}(x,t,\xi)$ can be easily obtained.

About $(ii)$, according to the Lax pair \eqref{q1} and Abel formula, we have $\tr(M)=\tr(N)=0$, then we can easily verify $(\det\psi_{\pm})_{x}=(\det\psi_{\pm})_{t}=0$. Therefore $\det(\psi_{\pm})$ have nothing to do with variables $x$ and $t$, which means $\det(\phi_{\pm})=\det(\psi_{\pm})=1$. From \eqref{q11}, we also have $\det S(\xi)=1$.

By \eqref{q12}, we have the wronskian representations of the scattering coefficients $a_{1}(\xi)$, $a_{2}(\xi)$ and $b(\xi)$
\begin{align}\label{q20}
\begin{split}
\left\{
     \begin{aligned}
     a_{1}(\xi)&=Wr(\psi_{-}^{(1)}(0,0,\xi),\psi_{+}^{(2)}(0,0,\xi)),\quad \xi\in \overline{\mathbb{C}^{+}}\backslash{\left\{0\right\}},\\
     a_{2}(\xi)&=Wr(\psi_{+}^{(1)}(0,0,\xi),\psi_{-}^{(2)}(0,0,\xi)),\quad \xi\in \overline{\mathbb{C}^{-}},\\
     b(\xi)&=Wr(\psi_{+}^{(1)}(0,0,\xi),\psi_{-}^{(1)}(0,0,\xi)),\quad \xi\in \mathbb{R},
     \end{aligned}
\right.
\end{split}
\end{align}
according to the analytical relations of $\psi_{\pm}(x,t,\xi)$, we get the item $(ii)$.
\end{proof}

\begin{prop}\label{pr3}
Matrices $\psi_{\pm}(x,t,\xi)$ and scattering data of $S(\xi)$ satisfy the following asymptotic properties
\begin{enumerate}[(i)]
\item As $\xi\rightarrow\infty$,
\begin{align}\label{q21}
\begin{split}
\left\{
     \begin{aligned}
     \psi_{-}^{(1)}(x,t,\xi)=
     \left(
     \begin{matrix}
     1\\
     0
     \end{matrix}
     \right)
     +O(\xi^{-1}),\quad
     \psi_{+}^{(2)}(x,t,\xi)=
     \left(
     \begin{matrix}
     0\\
     1
     \end{matrix}
     \right)
     +O(\xi^{-1}),\quad \xi\in\mathbb{C}^{+},\\
     \psi_{-}^{(2)}(x,t,\xi)=
     \left(
     \begin{matrix}
     0\\
     1
     \end{matrix}
     \right)
     +O(\xi^{-1}),\quad
     \psi_{+}^{(1)}(x,t,\xi)=
     \left(
     \begin{matrix}
     1\\
     0
     \end{matrix}
     \right)
     +O(\xi^{-1}),\quad \xi\in\mathbb{C}^{-},
     \end{aligned}
\right.
\end{split}
\end{align}
\vspace{-1mm}
\begin{align}\label{q22}
\begin{split}
\left\{
     \begin{aligned}
     a_{j}(\xi)&=1+O(\xi^{-1}),j=1,2,\quad \xi\in\overline{\mathbb{C}^{\pm}},\\
     b(\xi)&=O(\xi^{-1}),\quad \xi\in\mathbb{R}.
     \end{aligned}
\right.
\end{split}
\end{align}
\vspace{-1mm}
\item As $\xi\rightarrow0$,
\vspace{-1mm}
\begin{align}\label{q23}
\begin{split}
\left\{
     \begin{aligned}
     \psi_{-}^{(1)}(x,t,\xi)&=\frac{1}{\xi}
     \left(
     \begin{matrix}
     f_{1}(x,t)\\
     f_{2}(x,t)
     \end{matrix}
     \right)
     +O(1),\quad
     \psi_{-}^{(2)}(x,t,\xi)=\frac{2i}{A}
     \left(
     \begin{matrix}
     f_{1}(x,t)\\
     f_{2}(x,t)
     \end{matrix}
     \right)
     +O(\xi),\\
     \psi_{+}^{(1)}(x,t,\xi)&=-\frac{2i}{A}
     \left(
     \begin{matrix}
     \overline{f_{2}}(-x,t)\\
     \overline{f_{1}}(-x,t)
     \end{matrix}
     \right)
     +O(\xi),\quad
     \psi_{+}^{(2)}(x,t,\xi)=-\frac{1}{\xi}
     \left(
     \begin{matrix}
     \overline{f_{2}}(-x,t)\\
     \overline{f_{1}}(-x,t)
     \end{matrix}
     \right)
     +O(1),
     \end{aligned}
\right.
\end{split}
\end{align}
\vspace{-1mm}
\begin{align}\label{q24}
\begin{split}
\left\{
     \begin{aligned}
     a_{1}(\xi)&=\frac{A^{2}a_{2}(0)}{4\xi^{2}}+O(\xi^{-1}),\quad \xi\in\overline{\mathbb{C}^{+}},\\
     b(\xi)&=\frac{A a_{2}(0)}{2i\xi}+O(1),\quad \xi\in\mathbb{R},
     \end{aligned}
\right.
\end{split}
\end{align}
where $f_{1}(x,t)$, $f_{2}(x,t)$ can be solved by the following Volterra integral equations
\begin{align}\label{q25}
\begin{split}
\left\{
     \begin{aligned}
     f_{1}(x,t)&=\int_{-\infty}^{x}q(y,t)f_{2}(y,t)dy,\\
     f_{2}(x,t)&=\frac{A}{2i}+\int_{-\infty}^{x}(-\overline{q(-y,t)}+A)f_{1}(y,t)dy.
     \end{aligned}
\right.
\end{split}
\end{align}
\vspace{-2mm}
\end{enumerate}
\end{prop}
\begin{proof}
About $(i)$, because the determinant of $\psi_{\pm}$ is equal to one, we can get the columns of $\psi_{\pm}$ represented as \eqref{q21}. Then substituting \eqref{q21} into \eqref{q20}, we have \eqref{q22}.

About $(ii)$, tacking advantage of \eqref{q10}, we can assume
\begin{align}\label{q26}
\begin{split}
\left\{
     \begin{aligned}
     \psi_{-}^{(1)}(x,t,\xi)&=\frac{1}{\xi}
     \left(
     \begin{matrix}
     f_{1}(x,t)\\
     f_{2}(x,t)
     \end{matrix}
     \right)
     +O(1),\quad
     \psi_{-}^{(2)}(x,t,\xi)=
     \left(
     \begin{matrix}
     \tilde{f}_{1}(x,t)\\
     \tilde{f}_{2}(x,t)
     \end{matrix}
     \right)
     +O(\xi),\\
     \psi_{+}^{(1)}(x,t,\xi)&=
     \left(
     \begin{matrix}
     \tilde{g}_{1}(x,t)\\
     \tilde{g}_{2}(x,t)
     \end{matrix}
     \right)
     +O(\xi),\quad
     \psi_{+}^{(2)}(x,t,\xi)=\frac{1}{\xi}
     \left(
     \begin{matrix}
     g_{1}(x,t)\\
     g_{2}(x,t)
     \end{matrix}
     \right)
     +O(1),
     \end{aligned}
\right.
\end{split}
\end{align}
where some $f_{j}(x,t)$, $\tilde{f}_{j}(x,t)$, $g_{j}(x,t)$ and $\tilde{g}_{j}(x,t)$ $(j=1,2)$ are undetermined.

By the symmetry relation \eqref{q14}, we have
\begin{equation}\label{q27}
\left(\begin{array}{cc}
g_{1}(x,t)\\
g_{2}(x,t)
\end{array}\right)
=
\left(\begin{array}{cc}
-\overline{f_{2}}(-x,t)\\
-\overline{f_{1}}(-x,t)
\end{array}\right), \quad
\left(\begin{array}{cc}
\tilde{g}_{1}(x,t)\\
\tilde{g}_{2}(x,t)
\end{array}\right)
=
\left(\begin{array}{cc}
\overline{\tilde{f}_{2}}(-x,t)\\
\overline{\tilde{f}_{1}}(-x,t)
\end{array}\right).
\end{equation}

Submitting $\psi_{-}(x,t,\xi)=
\left(
\begin{matrix}
\frac{1}{\xi}f_{1}(x,t) & \tilde{f_{1}}(x,t)\\
\frac{1}{\xi}f_{2}(x,t) & \tilde{f_{2}}(x,t)
\end{matrix}
\right) $
into Volterra integral \eqref{q10} and letting $\xi\rightarrow0$, these undetermined equations take the forms
\begin{subequations}
\begin{align}
&f_{1}(x,t)=\int_{-\infty}^{x}q(y,t)f_{2}(y,t)dy,\label{q28a}\\
&\tilde{f_{1}}(x,t)=\int_{-\infty}^{x}q(y,t)\tilde{f_{2}}(y,t)dy,\label{q28b}\\
&f_{2}(x,t)=\frac{A}{2i}+
\int_{-\infty}^{x}(-\overline{q(-y,t)}+A)f_{1}(y,t)dy,\label{q28c}\\
&\tilde{f_{2}}(x,t)=1+
\int_{-\infty}^{x}(-\overline{q(-y,t)}+A)\tilde{f_{1}}(y,t)dy.\label{q28d}
\end{align}
\end{subequations}
Then we have
\begin{equation}\label{q29}
\left(\begin{array}{cc}
\tilde{f_{1}}(x,t)\\
\tilde{f_{2}}(x,t)
\end{array}\right)
=\frac{2i}{A}
\left(\begin{array}{cc}
f_{1}(x,t)\\
f_{2}(x,t)
\end{array}\right).
\end{equation}

Using the relations of \eqref{q29} and \eqref{q27}, the asymptotic properties of\eqref{q23} can be defined. Thus we can only use functions $f_{1}(x,t)$ and $f_{2}(x,t)$ to describe matrices $\psi_{\pm}(x,t,\xi)$.

It is going to be similar to $(i)$, where we submit \eqref{q23} into \eqref{q20}
\begin{subequations}
\begin{align}
a_{1}(\xi)&=\frac{1}{\xi^{2}}(\left|f_{2}(0,0)\right|^{2}-\left|f_{1}(0,0)\right|^{2})
+O(\xi^{-1}),\label{q30a}\\
a_{2}(\xi)&=\frac{4}{A^{2}}(\left|f_{2}(0,0)\right|^{2}-\left|f_{1}(0,0)\right|^{2})
+O(\xi),\label{q30b}\\
b(\xi)&=-\frac{2i}{\xi A}(\left|f_{2}(0,0)\right|^{2}-\left|f_{1}(0,0)\right|^{2})
+O(1).\label{q30c}
\end{align}
\end{subequations}
\end{proof}
\begin{rem}\label{re1}
In the case of pure-step initial data, that is, when
\begin{align}\label{q31}
q_{0}(x)=q_{0A}(x):=
\left\{
       \begin{aligned}
       0, \quad x<0,\\
       A, \quad x>0,
       \end{aligned}
\right.
\end{align}
the scattering matrix $S(\xi)$ can be expressed as follows
\begin{equation}\label{q32}
S(\xi)=[\phi_{+}(0,0,\xi)]^{-1}\phi_{-}(0,0,\xi)=
\left(\begin{array}{cc}
1+\frac{A^{2}}{4\xi^{2}} & -\frac{A}{2i\xi}\\
\frac{A}{2i\xi} & 1
\end{array}\right).
\end{equation}

It can be seen that in this case $a_{1}(\xi)$ has a single, simple zero $\xi=\frac{A}{2}i$ in the upper half-plane and $a_{2}(\xi)$ has no zeros in the lower half-plane.
\end{rem}

According to scattering relation \eqref{q12} and Proposition \ref{pr2}, the piece-wise meromorphic matrices can be defined as follows
\begin{align}\label{u1}
M(x,t;\xi)=\left\{ \begin{array}{ll}
M_{+}(x,t,\xi)=\left(  \frac{\psi_{-}^{(1)}(x,t;\xi)}{a_{1}(\xi)}, \psi_{+}^{(2)}(x,t;\xi)\right),   &\text{as } \xi\in \mathbb{C}^{+},\\[12pt]
M_{-}(x,t,\xi)=\left( \psi_{+}^{(1)}(x,t;\xi),\frac{\psi_{-}^{(2)}(x,t;\xi)}{a_{2}(\xi)}\right)  , &\text{as }\xi\in \mathbb{C}^{-},\\
\end{array}\right.
\end{align}
with
\begin{equation}\label{u2}
M_{+}(x,t,\xi)=M_{-}(x,t,\xi)J(x,t,\xi), \quad \xi\in\mathbb{R}\backslash{\left\{0\right\}},
\end{equation}
where the jump matrix
\begin{equation}\label{u3}
J(x,t,\xi)=
\left(\begin{array}{cc}
1+\frac{b(\xi)\overline{b(-\bar{\xi})}}{a_{1}(\xi)a_{2}(\xi)} & -\frac{b(\xi)}{a_{2}(\xi)}e^{-2i(\xi x-\xi^{2}t+8\xi^{4}\gamma t)}\\
-\frac{\overline{b(-\bar{\xi})}}{a_{1}(\xi)}e^{2i(\xi x-\xi^{2}t+8\xi^{4}\gamma t)} & 1
\end{array}\right).
\end{equation}
Now we define
\begin{equation}\label{r1}
r_{1}(\xi)=\frac{\overline{b(-\bar{\xi})}}{a_{1}(\xi)},\quad
r_{2}(\xi)=\frac{b(\xi)}{a_{2}(\xi)},
\end{equation}
then we have $1+r_{1}(\xi)r_{2}(\xi)=\frac{1}{a_{1}(\xi)a_{2}(\xi)}$ for $\xi\in\mathbb{R}\backslash\left\{0\right\}$. From the symmetry relations of $a_{j}(\xi)$, $(j=1,2)$ and $b(\xi)$, we have $\overline{r_{1}(-\bar{\xi})}=r_{1}(\xi)$ and $\overline{r_{2}(-\bar{\xi})}=r_{2}(\xi)$.

Looking back at \eqref{q24}, when $\xi\rightarrow0$, the different behaviors of the two cases $a_{2}(0)=0$ and $a_{2}(0)\neq 0$ make $P(x,t,\xi)$ qualitatively different. The case $a_{2}(0)\neq 0$ contains pure-step initial data in remark \ref{re1}, where $a_{1}(\xi)$ has a single, simple zero located on the imaginary axis in $\mathbb{C}^{+}$, and $a_{2}(\xi)$ has no zero in $\mathbb{C}^{-}$. Since small (in the $L^{1}$ norm) perturbations of the pure-step initial data preserve these properties, we will concentrate on the following two cases

Case1: $a_{1}(\xi)$ has a simple and pure imaginary zero in $\overline{\mathbb{C}^{+}}\backslash{\left\{0\right\}}$ at $\xi=i\xi_{1}$ with $\xi_{1}>0$ and $a_{2}(\xi)$ has no zero in $\overline{\mathbb{C}^{-}}$.

Case2: $a_{1}(\xi)$ has a simple and pure imaginary zero in $\overline{\mathbb{C}^{+}}\backslash{\left\{0\right\}}$ at $\xi=i\xi_{1}$ with $\xi_{1}>0$ and $a_{2}(\xi)$ has a zero in $\overline{\mathbb{C}^{-}}$ at $\xi=0$. Thus we assume that $\dot{a}_{2}(0)\neq 0$ and $a_{11}:=\displaystyle\lim_{\xi \to 0}\xi a_{1}(\xi)\neq 0$.
\begin{prop}\label{pro1}
In order to get the zero $\xi$ of $a_{1}(\xi)$, we need to calculate $\xi_{1}$ for both cases
\begin{enumerate}[(i)]
\item In case1,
\begin{equation}\label{u4}
\xi_{1}=\frac{A}{2}\exp\left\{-\frac{1}{2\pi i}v.p.\int_{-\infty}^{+\infty}
\frac{\ln\frac{\vartheta^{2}}{\vartheta^{2}+1}
(1-b(\vartheta)\overline{b(-\vartheta)})}{\vartheta}d\vartheta\right\}.
\end{equation}
\item In case2,
\begin{equation}\label{u5}
\xi_{1}=A\frac{\sqrt{(\re b(0))^{2}+F_{2}^{2}}-\re b(0)}{2F_{1}F_{2}},
\end{equation}
where
\begin{equation}\label{u6}
F_{1}=\exp\left\{\frac{1}{2\pi i}v.p.\int_{-\infty}^{+\infty}\frac{
\ln(1-b(\vartheta)\overline{b(-\vartheta)})}{\vartheta}d\vartheta\right\},\quad
F_{2}=\exp\left\{\frac{1}{2}\ln(1-\left|b(0)\right|^{2})\right\}.
\end{equation}
\end{enumerate}
\end{prop}
\begin{proof}
About $(i)$, to construst a scalar RH problem, which satisfies analytic and has no zeros in $\overline{\mathbb{C}^{+}}$ and $\overline{\mathbb{C}^{-}}$, we make transformations on $a_{1}(\xi)$ and $a_{2}(\xi)$ as follows
\begin{equation}\label{u7}
\hat{a}_{1}(\xi)=a_{1}(\xi)\frac{\xi^{2}}{(\xi-i\xi_{1})(\xi+i)},\quad
\hat{a}_{2}(\xi)=a_{2}(\xi)\frac{\xi-i\xi_{1}}{\xi-i},
\end{equation}
with
\begin{equation}\label{u9}
\hat{a}_{1}(\xi)\rightarrow1,\quad \hat{a}_{2}(\xi)\rightarrow1,\quad \xi\rightarrow\infty.
\end{equation}

Since $\hat{a}_{1}(\xi)$ and $\hat{a}_{2}(\xi)$ have no zeros in $\overline{\mathbb{C}^{+}}$ and $\overline{\mathbb{C}^{-}}$ and this RH problem is regular, therefore by using Sokhotski-Plemelj formula, we has the unique solution
\begin{equation}\label{u10}
\hat{a_{1}}(\xi)=e^{T(\xi)},\quad \hat{a_{2}}(\xi)=e^{-T(\xi)},
\end{equation}
with
\begin{equation}\label{u11}
T(\xi)=\frac{1}{2\pi i}\int_{-\infty}^{+\infty}\frac{\ln\frac{\vartheta^{2}}{\vartheta^{2}+1}
(1-b(\vartheta)\overline{b(-\vartheta)})}{\vartheta-\xi}d\vartheta,
\end{equation}
and
\begin{equation}\label{u12}
T(+i0)+T(-i0)=\frac{1}{\pi i}v.p.\int_{-\infty}^{+\infty}\frac{\ln\frac{\vartheta^{2}}{\vartheta^{2}+1}
(1-b(\vartheta)\overline{b(-\vartheta)})}{\vartheta}d\vartheta.
\end{equation}
When $\xi\rightarrow0$, we have
\begin{equation}\label{u13}
a_{1}(\xi)=\frac{\xi_{1}e^{T(+i0)}}{\xi^{2}}(1+o(\xi)),\quad
a_{2}(\xi)=\frac{1}{\xi_{1}}e^{-T(-i0)}.
\end{equation}
On the other hand,
\begin{equation}\label{u14}
a_{1}(\xi)=\frac{A^{2}a_{2}(0)}{4\xi^{2}}(1+o(\xi))
=\frac{A^{2}}{4\xi^{2}\cdot\xi_{1}}e^{-T(-i0)}(1+o(\xi)).
\end{equation}
Combining equations \eqref{u13} and \eqref{u14}, we have
\begin{equation}\label{u15}
\frac{4\xi_{1}^{2}}{A^{2}}=e^{-[T(+i0)+T(-i0)]}.
\end{equation}
Then $\xi_{1}$ can be solved as follows
\begin{equation}\label{u16}
\xi_{1}=\frac{A}{2}\exp\left\{-\frac{1}{2\pi i}v.p.\int_{-\infty}^{+\infty}\frac{\ln\frac{\vartheta^{2}}{\vartheta^{2}+1}
(1-b(\vartheta)\overline{b(-\vartheta)})}{\vartheta}d\vartheta\right\}.
\end{equation}

About $(ii)$, rewriting \eqref{q21} and using the symmetry relation \eqref{u14}, we have
\begin{subequations}
\begin{align}
\psi_{-}^{(1)}(x,t,\xi)&=\frac{1}{\xi}
     \left(
     \begin{matrix}
     f_{1}(x,t)\\
     f_{2}(x,t)
     \end{matrix}
     \right)
     +\left(
     \begin{matrix}
     m_{1}(x,t)\\
     m_{2}(x,t)
     \end{matrix}
     \right)+O(\xi),\label{u17a}\\
\psi_{-}^{(2)}(x,t,\xi)&=\frac{2i}{A}
     \left(
     \begin{matrix}
     f_{1}(x,t)\\
     f_{2}(x,t)
     \end{matrix}
     \right)+\xi
     \left(
     \begin{matrix}
     n_{1}(x,t)\\
     n_{2}(x,t)
     \end{matrix}
     \right)
     +O(\xi^{2}),\label{u17b}\\
\psi_{+}^{(1)}(x,t,\xi)&=-\frac{2i}{A}
     \left(
     \begin{matrix}
     \overline{f_{2}}(-x,t)\\
     \overline{f_{1}}(-x,t)
     \end{matrix}
     \right)-\xi
     \left(
     \begin{matrix}
     \overline{n_{2}}(-x,t)\\
     \overline{n_{1}}(-x,t)
     \end{matrix}
     \right)
     +O(\xi^{2}),\label{u17c}\\
\psi_{+}^{(2)}(x,t,\xi)&=-\frac{1}{\xi}
     \left(
     \begin{matrix}
     \overline{f_{2}}(-x,t)\\
     \overline{f_{1}}(-x,t)
     \end{matrix}
     \right)+
     \left(
     \begin{matrix}
     \overline{m_{2}}(-x,t)\\
     \overline{m_{1}}(-x,t)
     \end{matrix}
     \right)
     +O(\xi).\label{u17d}
\end{align}
\end{subequations}
Then, according to \eqref{q20}, the scattering data has the forms
\begin{subequations}
\begin{align}
a_{1}(\xi)&=\frac{1}{\xi}(f_{1}\overline{m_{1}}
-\overline{f_{1}}m_{1}-f_{2}\overline{m_{2}}+
\overline{f_{2}}m_{2})\big|_{x,t=0}+O(1),\label{u18a}\\
a_{2}(\xi)&=\frac{2i\xi}{A}(f_{1}\overline{n_{1}}+\overline{f_{1}}n_{1}
-f_{2}\overline{n_{2}}-\overline{f_{2}}n_{2})\big|_{x,t=0}+O(\xi^{2}),\label{u18b}\\
b(\xi)&=[\frac{2i}{A}(\overline{f_{1}}m_{1}-\overline{f_{2}}m_{2})
+(f_{1}\overline{n_{1}}-f_{2}\overline{n_{2}})]\big|_{x,t=0}+O(\xi),\label{u18c}
\end{align}
\end{subequations}
there $\overline{n_{1}}(0,0)m_{1}(0,0)-\overline{n_{2}}(0,0)m_{2}(0,0)=0$ and
the functions $f$, $m$ and $n$ satisfy the relations
$\left|f_{2}(0,0)\right|^{2}-\left|f_{1}(0,0)\right|^{2}=0$, $\left|m_{2}(0,0)\right|^{2}-\left|m_{1}(0,0)\right|^{2}=0$ and $\left|n_{2}(0,0)\right|^{2}-\left|n_{1}(0,0)\right|^{2}=0$. Besides, we also have
\begin{equation}\label{u19}
a_{11}=\displaystyle\lim_{\xi \to 0}(\xi a_{1}(\xi))=f_{1}\overline{m_{1}}
-\overline{f_{1}}m_{1}-f_{2}\overline{m_{2}}+\overline{f_{2}}m_{2}.
\end{equation}
In view of \eqref{u18a}, \eqref{u18b} and \eqref{u18c},\eqref{u19} can be rewritten as
\begin{equation}\label{u20}
a_{11}=-\frac{A^{2}}{4}\dot{a_{2}}(0)+iA\cdot \re b(0).
\end{equation}

For another, similar to $(i)$, we also make transformations on $a_{1}(\xi)$ and $a_{2}(\xi)$
\begin{equation}\label{u21}
\tilde{a}_{1}(\xi)=a_{1}(\xi)\frac{\xi}{\xi-i\xi_{1}},\quad \tilde{a}_{2}(\xi)=a_{2}(\xi)\frac{\xi-i\xi_{1}}{\xi}.
\end{equation}
Then we have the relations
\begin{equation}\label{u22}
a_{1}(\xi)=\frac{\xi-i\xi_{1}}{\xi}\exp\left\{\frac{1}{2\pi i}\int_{-\infty}^{+\infty}\frac{\ln
(1-b(\vartheta)\overline{b(-\vartheta)})}{\vartheta-\xi}d\vartheta\right\},
\end{equation}
\begin{equation}\label{u23}
a_{2}(\xi)=\frac{\xi}{\xi-i\xi_{1}}\exp\left\{-\frac{1}{2\pi i}\int_{-\infty}^{+\infty}\frac{\ln
(1-b(\vartheta)\overline{b(-\vartheta)})}{\vartheta-\xi}d\vartheta\right\},
\end{equation}
and
\begin{equation}\label{u24}
a_{11}=-i\xi_{1}F_{2}F_{1},
\end{equation}
there $F_{1}$, $F_{2}$ are given by \eqref{u6}. Since the following relations exist
\begin{equation}\label{u25}
1-\left|b(0)\right|^{2}=\displaystyle\lim_{\xi \to 0}a_{1}(\xi)a_{2}(\xi)=a_{11}\dot{a_{2}}(0),
\end{equation}
then
\begin{equation}\label{u26}
\dot{a}_{2}(0)=a_{11}^{-1}(1-\left|b(0)\right|^{2})=\frac{i}{\xi_{1}}F_{1}^{-1}F_{2}.
\end{equation}

Combining with \eqref{u20}, \eqref{u24} and \eqref{u26}, we have
\begin{equation}\label{u27}
(4iF_{2}F_{1})\xi_{1}^{2}+(4iA\cdot Reb(0))\xi_{1}-iA^{2}F_{1}^{-1}F_{2}=0,
\end{equation}
with the solution
\begin{equation}\label{u28}
\xi_{1}=A\frac{\sqrt{(Reb(0))^{2}+F_{2}^{2}}-Reb(0)}{2F_{1}F_{2}}.
\end{equation}
\end{proof}

Considering the singularity condition of $a_{1}(\xi)$ and $\psi_{j}(x,t,\xi)$, $j=1,2$ at $\xi=0$, we define the asymptotic behaviors of $P(x,t,\xi)$ at $\xi=0$ for both cases
\begin{enumerate}[(i)]
\item In case 1,
\begin{subequations}
\begin{align}
M_{+}&=
     \left(
     \begin{matrix}
     \frac{4}{A^{2}a_{2}(0)}f_{1}(x,t) & -\overline{f_{2}}(-x,t)\\
     \frac{4}{A^{2}a_{2}(0)}f_{2}(x,t) & -\overline{f_{1}}(-x,t)
     \end{matrix}
     \right)
     (I+O(\xi))
     \left(
     \begin{matrix}
     \xi & 0\\
     0 & \frac{1}{\xi}
     \end{matrix}
     \right),\quad \xi\rightarrow+i0,\label{u29a}\\
M_{-}&=\frac{2i}{A}
     \left(
     \begin{matrix}
     -\overline{f_{2}}(-x,t) & \frac{f_{1}(x,t)}{a_{2}(0)}\\
     -\overline{f_{1}}(-x,t) & \frac{f_{2}(x,t)}{a_{2}(0)}
     \end{matrix}
     \right)+O(\xi),\quad \xi\rightarrow-i0.\label{u29b}
\end{align}
\end{subequations}
\item In case 2,
\begin{subequations}\label{u30}
\begin{align}
M_{+}&=
     \left(
     \begin{matrix}
     \frac{f_{1}(x,t)}{a_{11}(\xi)} & -\overline{f_{2}}(-x,t)\\
     \frac{f_{2}(x,t)}{a_{11}(\xi)} & -\overline{f_{1}}(-x,t)
     \end{matrix}
     \right)
     (I+O(\xi))
     \left(
     \begin{matrix}
     1 & 0\\
     0 & \frac{1}{\xi}
     \end{matrix}
     \right),\quad \xi\rightarrow+i0,\label{u30a}\\
M_{-}&=\frac{2i}{A}
     \left(
     \begin{matrix}
     -\overline{f_{2}}(-x,t) & \frac{f_{1}(x,t)}{\dot{a}_{2}(0)}\\
     -\overline{f_{1}}(-x,t) & \frac{f_{2}(x,t)}{\dot{a}_{2}(0)}
     \end{matrix}
     \right)
     (I+O(\xi))
     \left(
     \begin{matrix}
     1 & 0\\
     0 & \frac{1}{\xi}
     \end{matrix}
     \right),\quad \xi\rightarrow-i0.\label{u30b}
\end{align}
\end{subequations}
\end{enumerate}

We consider the residue condition for $M(x,t,\xi)$ at zero $\xi=i\xi_{1}$
\begin{equation}\label{u31}
\res_{\xi=i\xi_{1}}\left[M(x,t,\xi)\right]_{1}=\frac{r_{1}}{\dot{a}_{1}(i\xi_{1})}
e^{-2\xi_{1}x+2i\xi_{1}^{2}t+16i\xi_{1}^{4}\gamma t}\left[M(x,t,i\xi_{1})\right]_{2},
\end{equation}
with
\begin{equation}\label{u32}
\psi_{-}^{(1)}(x,t,i\xi_{1})=r_{1}\psi_{+}^{(1)}(x,t,i\xi_{1}), \quad \left|r_{1}\right|=1,
\end{equation}
and $r_{1}$ is a constant.
\begin{RHP}\label{rhp}
Find a piece-wise meromorphic matrix $M(x,t,\xi)$ such that
\begin{enumerate}[(i)]
\item Jump conditions:
The non-tangential limits $M_{\pm}(x,t,\xi)=M(x,t,\xi\pm i0)$ exist a.e. for $\xi\in\mathbb{R}$ such that $M(x,t,\cdot)-I\in L^{2}(\mathbb{R}\backslash \left[-\varepsilon,\varepsilon\right])$ for any $\varepsilon>0$ and $M_{\pm}(x,t,\xi)$ satisfy the condition
\begin{equation}\label{u33}
M_{+}(x,t,\xi)=M_{-}(x,t,\xi)J(x,t,\xi),\quad \xi\in\mathbb{R}\backslash{\left\{0\right\}},
\end{equation}
where the jump matrix $J(x,t,\xi)$ is given by (\ref{u3}) and the jump contour is shown in Figure 1, with $a_{j}(\xi)$, $j=1,2$, $b(\xi)$ are given in case 1 or case 2;
\item Normalization condition at $\xi=\infty$:
\begin{align}\label{u34}
M(x,t,\xi)=I+O(\frac{1}{\xi}),\quad \xi\rightarrow\infty;
\end{align}
\item Residue condition (\ref{u31}) with $\xi_{1}$ given in terms of $b(\xi)$ using (\ref{u4}) or (\ref{u5});
\item Singularity conditions at $\xi=0$: $M(x,t,\xi)$ satisfies (\ref{u29a}), (\ref{u29b}) or (\ref{u30a}), (\ref{u30b}).
\end{enumerate}
\end{RHP}

\begin{figure}[H]
      \centering
\begin{tikzpicture}[scale=1.2]
\draw[->,thick](-4,0)--(4,0)node[right]{ \textcolor{black}{$\mathbb{R}$}};
\draw[-,dashed](0,0)--(0,2);
\draw[fill] (-2,0) circle [radius=0.035]node[below]{$\lambda_{3}$};
\draw[fill] (0.6,0) circle [radius=0.035]node[below]{$\lambda_{2}$};
\draw[fill] (2,0) circle [radius=0.035]node[below]{$\lambda_{1}$};
\draw[fill,blue] (0,1) circle [radius=0.035]node[right]{$i\xi_{1}$};
\draw[fill,blue] (0,0) circle [radius=0.035];
\draw[fill] (0,0) node[below]{$0$};
\draw[->,thick](1,0)--(1.5,0);
\draw[->,thick](-1.5,0)--(-1,0);
\end{tikzpicture}
          \caption{ \footnotesize The jump contour $\mathbb{R}$ and singular points of RH problem for $M(x,t,\xi)$.}
\end{figure}
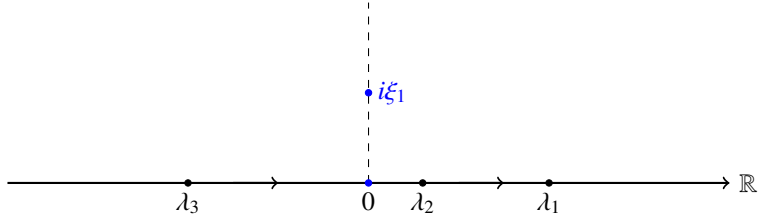

Assume that the RH problem (i)-(iv) has a solution $M(x,t,\xi)$. Then the solution of the initial value problem is given in terms of $M_{12}$ and $M_{21}$ as follows
\begin{equation}\label{u35}
q(x,t)=2i\displaystyle\lim_{\xi \to \infty}\xi M_{12}(x,t,\xi),
\end{equation}
\begin{equation}\label{u36}
q(-x,t)=-2i\displaystyle\lim_{\xi \to \infty}\xi \overline{M_{21}(x,t,\xi)}.
\end{equation}

\begin{prop}\label{prop1}
Suppose that $a_{1}(\xi)$, $a_{2}(\xi)$ and $b(\xi)$ satisfy the following conditions
\begin{enumerate}[(i)]
\item $a_{1}(\xi)$ and $a_{2}(\xi)$ are given by case 2;
\item $b(\xi)=0$ for all the $\xi\in\mathbb{R}$.
\end{enumerate}
Then $\xi_{1}$ is uniquely determined by $\xi_{1}=\frac{A}{2}$. The exact solution $q(x,t)$ of problem \eqref{e1} and \eqref{e4} is given by
\begin{equation}\label{a1}
q(x,t)=\frac{A}{1-e^{-Ax+\frac{i}{2}A^{2}t+iA^{4}\gamma t+i\alpha}},
\end{equation}
there $r_{1}=e^{i\alpha}$ with $\alpha\in\mathbb{R}$.
\end{prop}
\begin{proof}
Since $b(\xi)=0$, then we can obtain $b(0)=0$. From case 2 of Proposition \ref{pro1}, it can be received that $\xi_{1}=\frac{A}{2}$, moreover,
\begin{align}\label{u37}
a_{1}(\xi)=\frac{\xi-\frac{A}{2}i}{\xi},\quad
a_{2}(\xi)=\frac{\xi}{\xi-\frac{A}{2}i},
\end{align}
and
\begin{align}\label{u38}
a_{11}=\frac{2i}{A},\quad
\dot{a}_{2}(0)=\frac{A}{2i}.
\end{align}

Based on the condition $b(\xi)=0$, it can be seen that $M(x,t,\xi)$ is a meromorphic function with the only pole at point $\xi=i\xi_{1}$. Then, comparing \eqref{u30a} and \eqref{u30b}, we can conclude that $f_{1}(x,t)=-\overline{f_{2}}(-x,t)$. Therefore, the singularity conditions \eqref{u30} convert into a residue condition
\begin{align}\label{u39}
\res_{\xi=0}\left[M(x,t,\xi)\right]_{2}=\frac{A}{2i}
\left[M(x,t,0)\right]_{1}.
\end{align}

Then considering the normalization condition (the item (ii) of RH problem \ref{rhp}) at $\xi=\infty$, we derive the following representation of $M(x,t,\xi)$
\begin{align}\label{u40}
M(x,t,\xi)=\left(\begin{array}{cc}
\frac{\xi+f_{1}(x,t)}{\xi-\frac{iA}{2}} & \frac{f_{1}(x,t)}{\xi}   \\
\frac{-\overline{f_{1}}(-x,t)}{\xi-\frac{iA}{2}} & \frac{\xi-\overline{f_{1}}(-x,t)}{\xi}
\end{array}\right).
\end{align}
Using the residue condition \eqref{u31} at $\xi=\frac{A}{2}i$, we obtain
\begin{align}\label{u41}
f_{1}(x,t)=\frac{A}{2i}\frac{1}{1-e^{-Ax+\frac{i}{2}A^{2}t+iA^{4}\gamma t+i\alpha}}.
\end{align}
Finally, by \eqref{u35}, we can get the one-soliton solution \eqref{a1}.
\end{proof}

\section{The long-time asymptotics}\quad

In this section, we consider the long-time asymptotic behaviors of the solution $q(x,t)$ to the nonlocal LPD equation. By utilizing the nonlinear steepest descent method, the original RH problem \ref{rhp} can be solved into an explicit problem by transformation. Owing to the equations \eqref{u35} and \eqref{u36}, it is enough that we only study the RH problem for case $x>0$. Let $\mu=\frac{x}{t}$ the phase function can be expressed as follows
\begin{align}\label{4-1}
\theta(\xi,\mu)=\xi\mu-\xi^{2}+8\xi^{4}\gamma.
\end{align}
Then the exponentials of jump matrices have the form $e^{it\theta(\xi)}=e^{t\varphi(\xi)}$. By calculation, it can be concluded that there are several cases to the roots for $\varphi^{\prime}(\xi)$:
\begin{enumerate}[(i)]
\item when $\mu^{2}>\frac{1}{27\gamma}$, there exists one real stationary point;
\item when $\mu^{2}=\frac{1}{27\gamma}$, there exist three real stationary points and two of them are equal;
\item when $\mu^{2}<\frac{1}{27\gamma}$, there exist three different real stationary points.
\end{enumerate}

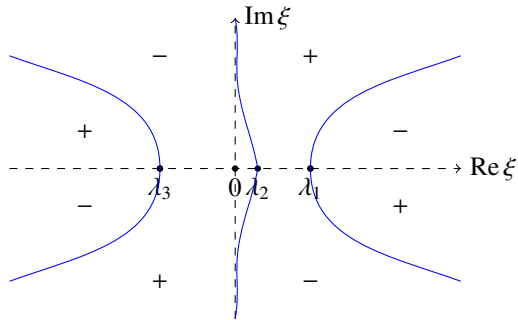
\begin{figure}[H]
      \centering
\begin{tikzpicture}[scale=1]
\draw[->,dashed](-3,0)--(3,0)node[right]{ \textcolor{black}{$\re \xi$}};
\draw[->,dashed](0,-2)--(0,2)node[right]{\textcolor{black}{$\im \xi$}};
\draw[fill] (-1,0) circle [radius=0.035]node[below]{$\lambda_{3}$};
\draw[fill] (1,0) circle [radius=0.035]node[below]{$\lambda_{1}$};
\draw[fill] (1,0) circle [radius=0.035];
\draw[fill] (-1,0) circle [radius=0.035];
\draw[fill] (0,0) circle [radius=0.035]node[below]{$0$};
\draw[fill] (0.3,0) circle [radius=0.035]node[below]{$\lambda_{2}$};
\draw [-, blue] (1,0) to [out=90,in=-160] (3,1.5);
\draw [-, blue] (1,0) to [out=-90,in=160] (3,-1.5);
\draw [-, blue] (-1,0) to [out=90,in=-20] (-3,1.5);
\draw [-, blue] (-1,0) to [out=-90,in=20] (-3,-1.5);
\draw [-, blue] (0.3,0) to [out=100,in=-80] (0.05,1);
\draw [-, blue] (0.3,0) to [out=-100,in=80] (0.05,-1);
\draw [-, blue] (0.05,1) to [out=100,in=-80] (0,2);
\draw [-, blue] (0.05,-1) to [out=-100,in=80] (0,-2);
\draw[fill] (-1,1.5) node{$-$};
\draw[fill] (-1,-1.5) node{$+$};
\draw[fill] (-2,0.5) node{$+$};
\draw[fill] (2.2,0.5) node{$-$};
\draw[fill] (2.2,-0.5) node{$+$};
\draw[fill] (-2,-0.5) node{$-$};
\draw[fill] (1,1.5) node{$+$};
\draw[fill] (1,-1.5) node{$-$};
\end{tikzpicture}
          \caption{ \footnotesize The signature table for $\re\varphi(\xi)$ in the complex $\xi$-plane.}
\end{figure}

In the following work, we consider the case $\mu^{2}<\frac{1}{27\gamma}$ and the three different real roots are as follows
\begin{align}\label{4-2}
\begin{split}
&\lambda_{1}=\frac{w^{2}}{4}\sqrt[3]{-\frac{\mu}{\gamma}+\frac{1}{\gamma}\sqrt{\mu^{2}-\frac{1}{27\gamma}}}
+\frac{w}{4}\sqrt[3]{-\frac{\mu}{\gamma}-\frac{1}{\gamma}\sqrt{\mu^{2}-\frac{1}{27\gamma}}},\\
&\lambda_{2}=\frac{w}{4}\sqrt[3]{-\frac{\mu}{\gamma}+\frac{1}{\gamma}\sqrt{\mu^{2}-\frac{1}{27\gamma}}}
+\frac{w^{2}}{4}\sqrt[3]{-\frac{\mu}{\gamma}-\frac{1}{\gamma}\sqrt{\mu^{2}-\frac{1}{27\gamma}}},\\
&\lambda_{3}=\frac{1}{4}\sqrt[3]{-\frac{\mu}{\gamma}+\frac{1}{\gamma}\sqrt{\mu^{2}-\frac{1}{27\gamma}}}
+\frac{1}{4}\sqrt[3]{-\frac{\mu}{\gamma}-\frac{1}{\gamma}\sqrt{\mu^{2}-\frac{1}{27\gamma}}},
\end{split}
\end{align}
with $w=\frac{-1+\sqrt{3}i}{2}$. When $\mu\in D_{1}=\left(\epsilon,\sqrt{\frac{1}{27\gamma}}-\epsilon\right)$ for any positive constant $\epsilon$, we have the relations $\lambda_{1}>0$, $\lambda_{2}>0$ and $\lambda_{3}<0$. The signature table for the distribution of $\re\varphi(\xi)$ in the complex $\xi$-plane is shown in Figure 2. in this case. We will focus this case in the following analysis. For case $\mu\in D_{2}=\left(-\sqrt{\frac{1}{27\gamma}}+\epsilon,-\epsilon\right)$, the similar discussion can be conduced, the only different is that the positive or negative sign of these three real roots are the exact opposite, and it will lead to
the singularity point zero located in the right of $\lambda_{1}$ point. This difference will be reflected in the discussion of singularity conditions at zero point in the after transformation. The result to the long-time asymptotics of the solution $q(x,t)$ will keep the same. When $\mu=0$, we have $\lambda_{3}=0$. Similar to the above discussion, we will also obtain the same asymptotic behavior, but here we fix $\mu=0$. \\

\subsection{Factorization of the jump matrix}\quad

Firstly, there are two types of triangular factorizations of jump matrix $J(x,t,\xi)$
\begin{align*}
J(x,t,\xi)=\left\{ \begin{aligned}
         &\left(\begin{array}{cc}
             1 & -r_{2}(\xi)e^{-2it\theta} \\
             0 & 1 \\
           \end{array}\right)\left(
                               \begin{array}{cc}
                                 1 & 0 \\
                                 -r_{1}(\xi)e^{2it\theta} & 1 \\
                               \end{array}
                             \right),\\
         &\left(\begin{array}{cc}
                    1 & 0 \\
                    -\frac{r_{1}(\xi)e^{2it\theta}}{1+r_{1}(\xi)r_{2}(\xi)} & 1 \\
                  \end{array}\right)\left(
                                      \begin{array}{cc}
                                        1+r_{1}(\xi)r_{2}(\xi) & 0 \\
                                        0 & \frac{1}{1+r_{1}(\xi)r_{2}(\xi)} \\
                                      \end{array}
                  \right)\left(
                  \begin{array}{cc}
                    1 & -\frac{r_{2}(\xi)e^{-2it\theta}}{1+r_{1}(\xi)r_{2}(\xi)} \\
                      0 & 1 \\
                       \end{array}
                        \right).
                          \end{aligned} \right.
\end{align*}
To get ride of the intermediate matrix of the second factorization, it is necessary to introduce a function $\delta(\xi,\mu)$ as the solution of the scalar RH problem:
\begin{enumerate}[(i)]
\item $\delta(\xi,\mu)$ is holomorphic for $\xi\in\mathbb{C}\backslash ((-\infty,\lambda_{3} ]\cup[\lambda_{2},\lambda_{1}])$,
\item $\delta_{+}(\xi,\mu)=\delta_{-}(\xi,\mu)(1+r_{1}(\xi)r_{2}(\xi))$,\quad
$\xi\in(-\infty,\lambda_{3})\cup(\lambda_{2},\lambda_{1})$,
\item $\delta(\xi,\mu)\rightarrow 1$,\quad $\xi\rightarrow\infty$.
\end{enumerate}
Using the plemelj formula, its solution can be written in the form of Cauchy-type integral
\begin{align}\label{4-4}
\delta(\xi,\mu)=\exp\left\{\frac{1}{2\pi i}\left(\int_{-\infty}^{\lambda_{3}}+\int_{\lambda_{2}}^{\lambda_{1}}\right)\frac{\ln(1+r_{1}(\zeta)r_{2}(\zeta))}{\zeta-\xi}d\zeta \right\}.
\end{align}
From the symmetry relations of $r_{1}(\xi)$ and $r_{2}(\xi)$, we have $\delta(\xi,\mu)=\overline{\delta(-\bar{\xi},\mu)}$.

Moreover,$\delta(\xi,\mu)$ can be written as
\begin{align}\label{4-5}
\begin{split}
\delta(\xi,\mu)&=(\xi-\lambda_{1})^{iv(\lambda_{1})}\left(\frac{\xi-\lambda_{2}}{\xi-\lambda_{3}} \right)^{-iv(\lambda_{1})}e^{\chi_{1}(\xi)},\\
&=(\xi-\lambda_{3})^{iv(\lambda_{3})}\left(\frac{\xi-\lambda_{2}}{\xi-\lambda_{1}} \right)^{-iv(\lambda_{2})}e^{\chi_{2}(\xi)},\\
&=(\xi-\lambda_{3})^{iv(\lambda_{3})}\left(\frac{\xi-\lambda_{2}}{\xi-\lambda_{1}} \right)^{-iv(\lambda_{3})}e^{\chi_{3}(\xi)},
\end{split}
\end{align}
with
\begin{align}\label{4-6}
\begin{split}
\chi_{1}(\xi)&=\frac{1}{2\pi i}\left[\int_{\lambda_{2}}^{\lambda_{3}}\ln\left(\frac{1+r_{1}(\zeta)r_{2}(\zeta)}{1+r_{1}(\lambda_{1})r_{2}(\lambda_{1})} \right)\frac{d\zeta}{\zeta-\xi}-\int_{-\infty}^{\lambda_{1}}\ln(\xi-\zeta)d\ln(1+r_{1}(\zeta)r_{2}(\zeta)) \right],\\
\chi_{2}(\xi)&=\frac{1}{2\pi i}\left[\int_{\lambda_{2}}^{\lambda_{1}}\ln\left(\frac{1+r_{1}(\zeta)r_{2}(\zeta)}{1+r_{1}(\lambda_{2})r_{2}(\lambda_{2})} \right)\frac{d\zeta}{\zeta-\xi}-\int_{-\infty}^{\lambda_{3}}\ln(\xi-\zeta)d\ln(1+r_{1}(\zeta)r_{2}(\zeta)) \right],\\
\chi_{3}(\xi)&=\frac{1}{2\pi i}\left[\int_{\lambda_{2}}^{\lambda_{1}}\ln\left(\frac{1+r_{1}(\zeta)r_{2}(\zeta)}{1+r_{1}(\lambda_{3})r_{2}(\lambda_{3})} \right)\frac{d\zeta}{\zeta-\xi}-\int_{-\infty}^{\lambda_{3}}\ln(\xi-\zeta)d\ln(1+r_{1}(\zeta)r_{2}(\zeta)) \right],\\
\end{split}
\end{align}
there $v(\lambda_{l})$ $(l=1,2,3)$ can be expressed as
\begin{align}\label{4-7}
v(\lambda_{l})
=-\frac{1}{2\pi}\ln|1+r_{1}(\lambda_{l})r_{2}(\lambda_{l})|-\frac{i}{2\pi}\Delta(\lambda_{l}),
\quad l=1,2,3,
\end{align}
\begin{align*}
\Delta(\lambda_{l})=\int_{-\infty}^{\lambda_{l}}d \arg(1+r_{1}(\zeta)r_{2}(\zeta)).
\end{align*}
By assuming that $\Delta(\xi)\in(-\pi,\pi)$ for $\xi\in\mathbb{R}$, we have $|\im v(\xi)|<\frac{1}{2}$. In this assumption, $\ln(1+r_{1}(\xi)r_{2}(\xi))$ is single-valued, and the singularities $\xi=\lambda_{l}$ $(l=1,2,3)$ of $\delta(\xi,\mu)$ are square integrable.

With $\delta(\xi,\mu)$ constructed in \eqref{4-4}, we define the new function
\begin{align}\label{4-8}
\widetilde{M}(x,t,\xi)=M(x,t,\xi)\delta^{-\sigma_{3}}(\xi,\mu),
\end{align}
therefore we can define the RH problem of the function $\widetilde{M}(x,t,\xi)$.
\begin{RHP}\label{rhp2}
Find a matrix function $\widetilde{M}(x,t,\xi)$ admits the following relations
\begin{enumerate}[(i)]
\item $\widetilde{M}(x,t,\xi)$ is meromorphic for $\xi\in\mathbb{C}\backslash\mathbb{R}$ and has a simple pole located at $\xi=i\xi_{1}$, $\xi_{1}>0$.
\item Jump conditions:
The non-tangential limits $\widetilde{M}(x,t,\xi)=\widetilde{M}(x,t,\xi\pm i0)$ exist a.e. for $\xi\in\mathbb{R}$ such that $\widetilde{M}(x,t,\cdot)-I\in L^{2}(\mathbb{R}\backslash \left[-\varepsilon,\varepsilon\right])$ for any $\varepsilon>0$ and $\widetilde{M}_{\pm}(x,t,\xi)$ satisfy the condition
\begin{equation}\label{4-9}
\widetilde{M}_{+}(x,t,\xi)=\widetilde{M}_{-}(x,t,\xi)\widetilde{J}(x,t,\xi),\quad \xi\in\mathbb{R}\backslash{\left\{0\right\}},
\end{equation}
where
\begin{align}\label{4-10}
\widetilde{J}=\left\{ \begin{aligned}
         &\left(\begin{array}{cc}
             1 & 0 \\
        -\frac{r_{1}(\xi)\delta_{-}^{-2}(\xi,\mu)}{1+r_{1}(\xi)r_{2}(\xi)}e^{2it\theta} & 1 \\
           \end{array}\right)\left(
                               \begin{array}{cc}
           1 & -\frac{r_{2}(\xi)\delta_{+}^{2}(\xi,\mu)}{1+r_{1}(\xi)r_{2}(\xi)}e^{-2it\theta} \\
        0 & 1 \\
                               \end{array}
                             \right),\quad
        \xi<\lambda_{3},\lambda_{2}<\xi<\lambda_{1},\\
         &\left(\begin{array}{cc}
                    1 & -r_{2}(\xi)\delta^{2}(\xi,\mu)e^{-2it\theta} \\
                    0 & 1 \\
                  \end{array}\right)\left(
                  \begin{array}{cc}
                    1 & 0 \\
                    -r_{1}(\xi)\delta^{-2}(\xi,\mu)e^{2it\theta} & 1 \\
                       \end{array}
                        \right),\quad \xi>\lambda_{1},\lambda_{3}<\xi<\lambda_{2}.
                          \end{aligned} \right.
\end{align}
\item Normalization condition at $\xi=\infty$:
\begin{align}\label{4-11}
\widetilde{M}(x,t,\xi)=I+O(\frac{1}{\xi}),\quad \xi\rightarrow\infty.
\end{align}
\item Residue condition:
\begin{equation}\label{4-12}
\res_{\xi=i\xi_{1}}\left[\widetilde{M}(x,t,\xi)\right]_{1}=\frac{r_{1}}{\dot{a}_{1}(i\xi_{1})\delta^{2}(i\xi_{1})}
e^{-2\xi_{1}x+2i\xi_{1}^{2}t+16i\xi_{1}^{4}\gamma t}\left[\widetilde{M}(x,t,i\xi_{1})\right]_{2}.
\end{equation}
\item Singularity conditions at $\xi=0$:\\
In case 1,
\begin{subequations}
\begin{align}
\widetilde{M}_{+}&=
     \left(
     \begin{matrix}
     \frac{4 f_{1}(x,t)}{A^{2}a_{2}(0)\delta(0,\mu)} & -\delta(0,\mu)\overline{f_{2}}(-x,t)\\
     \frac{4 f_{2}(x,t)}{A^{2}a_{2}(0)\delta(0,\mu)} & -\delta(0,\mu)\overline{f_{1}}(-x,t)
     \end{matrix}
     \right)
     (I+O(\xi))
     \left(
     \begin{matrix}
     \xi & 0\\
     0 & \frac{1}{\xi}
     \end{matrix}
     \right),\quad \xi\rightarrow+i0,\label{4-13a}\\
\widetilde{M}_{-}&=\frac{2i}{A}
     \left(
     \begin{matrix}
     -\frac{\overline{f_{2}}(-x,t)}{\delta(0,\mu)} & \frac{\delta(0,\mu)f_{1}(x,t)}{a_{2}(0)}\\
     -\frac{\overline{f_{1}}(-x,t)}{\delta(0,\mu)} & \frac{\delta(0,\mu)f_{2}(x,t)}{a_{2}(0)}
     \end{matrix}
     \right)+O(\xi),\quad \xi\rightarrow-i0.\label{4-13b}
\end{align}
\end{subequations}
In case 2,
\begin{subequations}\label{4-14}
\begin{align}
\widetilde{M}_{+}&=
     \left(
     \begin{matrix}
     \frac{f_{1}(x,t)}{a_{11}\delta(0,\mu)} & -\overline{f_{2}}(-x,t)\delta(0,\mu)\\
     \frac{f_{2}(x,t)}{a_{11}\delta(0,\mu)} & -\overline{f_{1}}(-x,t)\delta(0,\mu)
     \end{matrix}
     \right)
     (I+O(\xi))
     \left(
     \begin{matrix}
     1 & 0\\
     0 & \frac{1}{\xi}
     \end{matrix}
     \right),\quad \xi\rightarrow+i0,\label{4-14a}\\
\widetilde{M}_{-}&=\frac{2i}{A}
     \left(
     \begin{matrix}
     -\frac{\overline{f_{2}}(-x,t)}{\delta(0,\mu)} & \frac{\delta(0,\mu)f_{1}(x,t)}{\dot{a}_{2}(0)}\\
     -\frac{\overline{f_{1}}(-x,t)}{\delta(0,\mu)} & \frac{\delta(0,\mu)f_{2}(x,t)}{\dot{a}_{2}(0)}
     \end{matrix}
     \right)
     (I+O(\xi))
     \left(
     \begin{matrix}
     1 & 0\\
     0 & \frac{1}{\xi}
     \end{matrix}
     \right),\quad \xi\rightarrow-i0.\label{4-14b}
\end{align}
\end{subequations}
\end{enumerate}
\end{RHP}

\subsection{RH problem transformation}\quad

There it is necessary to carry on the second transformation to transform the contour, which can make the jump matrices decline to identity $I$ for the large-$t$. Generally speaking, the new RH transformation $\widetilde{M}(x,t,\xi)$ rely on the reflection coefficients $r_{j}(\xi)$, $\frac{r_{j}(\xi)}{1+r_{1}(\xi)r_{2}(\xi)}$, $j=1,2$. Then by the classical Deift-Zhou method, they can be approximated by some rational functions with good error control.

In the following analysis, for clarity, we will suppose the initial data $q_{0}(x)$ allows a compact perturbation of pure-step initial data $q_{0A}(x)$ shown in \eqref{q31}, which makes certain that the eigenfuntions $\psi_{\pm}^{(s)}(x,0,\xi)$, $s=1,2$ and thus $r_{j}(\xi)$, $j=1,2$ are meromorphic in $\mathbb{C}$. Then we define the function $\widehat{M}(x,t,\xi)$ as follows (see Figure 3.)\\

\begin{align}\label{4-15}
\widehat{M}(x,t,\xi)=\left\{ \begin{aligned}
         &\widetilde{M}(x,t,\xi)\left(
                          \begin{array}{cc}
                            1 & \frac{r_{2}(\xi)\delta^{2}(\xi,\mu)}{1+r_{1}(\xi)r_{2}(\xi)}e^{-2it\theta} \\
                            0 & 1 \\
                          \end{array}
                        \right)
         ,&\quad &\xi\in \Omega_{2},\\
         &\widetilde{M}(x,t,\xi)\left(
                          \begin{array}{cc}
                            1 & 0 \\
                            r_{1}(\xi)\delta^{-2}(\xi,\mu)e^{2it\theta} & 1 \\
                          \end{array}
                        \right),&\quad &\xi\in \Omega_{1},\\
         &\widetilde{M}(x,t,\xi)\left(
                          \begin{array}{cc}
                            1 & -r_{2}(\xi)\delta^{2}(\xi,\mu)e^{-2it\theta} \\
                            0 & 1 \\
                          \end{array}
                        \right),& \quad &\xi\in \Omega^{*}_{1},\\
         &\widetilde{M}(x,t,\xi)\left(
                          \begin{array}{cc}
                            1 & 0 \\
                            -\frac{r_{1}(\xi)\delta^{-2}(\xi,\mu)}{1+r_{1}(\xi)r_{2}(\xi)}e^{2it\theta} & 1 \\
                          \end{array}
                        \right),& \quad &\xi\in \Omega^{*}_{2},\\
                        &\widetilde{M}(x,t,\xi),& \quad &\xi\in \Omega_{0}\cup \Omega^{*}_{0}.\\
                          \end{aligned} \right.
\end{align}

Here, we choose the appropriate angle between the real axis and rays $\Upsilon_{j}, \Upsilon^{*}_{j}$ to ensure the discrete spectrum $i\xi_{1}$ is located in the sector $\Omega_{0}$. Then the following RH problem $\widehat{M}(x,t,\xi)$ on the contour $\Upsilon$ is obtained.

\begin{RHP}\label{rhp3}
Find a matrix function $\widehat{M}(x,t,\xi)$ admits the following relations
\begin{enumerate}[(i)]
\item $\widehat{M}(x,t,\xi)$ is meromorphic for $\xi\in\mathbb{C}\backslash\Upsilon$ and has a simple pole located at $\xi=i\xi_{1}$, $\xi_{1}>0$.
\item Jump conditions:
The non-tangential limits $\widehat{M}(x,t,\xi)=\widehat{M}(x,t,\xi\pm i0)$ exist a.e. for $\xi\in\mathbb{R}$ such that $\widehat{M}(x,t,\cdot)-I\in L^{2}(\mathbb{R}\backslash \left[-\varepsilon,\varepsilon\right])$ for any $\varepsilon>0$ and $\widehat{M}_{\pm}(x,t,\xi)$ satisfy the condition
\begin{equation}\label{4-16}
\widehat{M}_{+}(x,t,\xi)=\widehat{M}_{-}(x,t,\xi)\widehat{J}(x,t,\xi),\quad \xi\in \Upsilon,
\end{equation}
where
\begin{align}\label{4-17}
\widehat{J}(x,t,\xi)=\left\{ \begin{aligned}
         &\left(
         \begin{array}{cc}
         1 & -\frac{r_{2}(\xi)\delta^{2}(\xi,\mu)}{1+r_{1}(\xi)r_{2}(\xi)}e^{-2it\theta}\\
         0 & 1 \\
         \end{array}
         \right)
         ,&\quad &\xi\in \Upsilon_{2},\\
         &\left(
         \begin{array}{cc}
         1 & 0 \\
         -r_{1}(\xi)\delta^{-2}(\xi,\mu)e^{2it\theta} & 1 \\
         \end{array}
         \right),&\quad &\xi\in \Upsilon_{1},\\
         &\left(
         \begin{array}{cc}
         1 & r_{2}(\xi)\delta^{2}(\xi,\mu)e^{-2it\theta} \\
         0 & 1 \\
         \end{array}
         \right),&\quad &\xi\in \Upsilon^{*}_{1},\\
         &\left(
         \begin{array}{cc}
         1 & 0 \\
         \frac{r_{1}(\xi)\delta^{-2}(\xi,\mu)}{1+r_{1}(\xi)r_{2}(\xi)}e^{2it\theta} & 1 \\
         \end{array}
         \right),& \quad &\xi\in \Upsilon^{*}_{2}.\\
                          \end{aligned} \right.
\end{align}
\item Normalization condition at $\xi=\infty$:
\begin{align}\label{4-18}
\widehat{M}(x,t,\xi)=I+O(\frac{1}{\xi}),\quad \xi\rightarrow\infty.
\end{align}
\item Residue condition:
\begin{equation}\label{4-19}
\res_{\xi=i\xi_{1}}\left[\widehat{M}(x,t,\xi)\right]_{1}=c_{1}(x,t)\left[\widehat{M}(x,t,i\xi_{1})\right]_{2},
\end{equation}
where $c_{1}(x,t)=\frac{r_{1}}{\dot{a}_{1}(i\xi_{1})\delta^{2}(i\xi_{1})}
e^{-2\xi_{1}x+2i\xi_{1}^{2}t+16i\xi_{1}^{4}\gamma t}$.
\item Singularity conditions at $\xi=0$, $\widehat{M}(x,t,\xi)$ satisfies the following relations in both case1 and case2:\\
\begin{subequations}\label{4-20}
\begin{align}
\widehat{M}_{+}&=
     \left(
     \begin{matrix}
     -\frac{2i \overline{f}_{2}(-x,t)}{A \delta(0,\mu)}+O(\xi) & -\frac{1}{\xi}\overline{f}_{2}(-x,t)\delta(0,\mu)+O(1)\\
     -\frac{2i \overline{f}_{1}(-x,t)}{A \delta(0,\mu)}+O(\xi) & -\frac{1}{\xi}\overline{f}_{1}(-x,t)\delta(0,\mu)+O(1)
     \end{matrix}
     \right),\quad \xi\rightarrow+i0,\label{4-20a}\\
\widehat{M}_{-}&=\frac{2i}{A}
     \left(
     \begin{matrix}
     -\frac{\overline{f}_{2}(-x,t)}{\delta(0,\mu)}+O(\xi) & -\frac{A}{2i\xi}\delta(0,\mu)\overline{f}_{2}(-x,t)+O(\xi)\\
     -\frac{\overline{f}_{1}(-x,t)}{\delta(0,\mu)}+O(\xi) & -\frac{A}{2i\xi}\delta(0,\mu)\overline{f}_{1}(-x,t)+O(\xi)
     \end{matrix}
     \right),\quad \xi\rightarrow-i0.\label{4-20b}
\end{align}
\end{subequations}
\end{enumerate}
\end{RHP}

Moreover, it can be noted that the singularity conditions at $\xi=0$ in both cases can be reduced to the same residue condition
\begin{align}\label{4-21}
\res_{\xi=0}\left[\widehat{M}(x,t,\xi)\right]_{2}=c_{0}(\mu)\left[\widehat{M}(x,t,0)\right]_{1},
\end{align}
with $c_{0}(\mu)=\frac{A\delta^{2}(0,\mu)}{2i}$.

\begin{figure}[H]
      \centering
\begin{tikzpicture}[scale=1.5]
\draw[->,dashed](3,0)--(4,0)node[right]{ \textcolor{black}{$\re \xi$}};
\draw[->,dashed](0.5,-2)--(0.5,2)node[right]{\textcolor{black}{$\im \xi$}};
\draw[->][dashed](-3,0)--(-2,0);
\draw[-][dashed](-2,0)--(-1,0);
\draw[->][dashed](-1,0)--(0,0);
\draw[-][dashed](0,0)--(1,0);
\draw[->][thick](0,1)--(0.5,0.5);
\draw[-][thick](0.5,0.5)--(1,0);
\draw[->][thick](0,-1)--(0.5,-0.5);
\draw[-][thick](0.5,-0.5)--(1,0);
\draw[-][thick](1.5,-0.5)--(1,0);
\draw[-][thick](1.5,0.5)--(1,0);
\draw[-][thick](2,0)--(1.5,0.5);
\draw[-][thick](0,1)--(-0.5,0.5);
\draw[->][thick](-1,0)--(-0.5,0.5);
\draw[-][thick](0,-1)--(-0.5,-0.5);
\draw[->][thick](-1,0)--(-0.5,-0.5);
\draw[->][thick](-2,1)--(-1.5,0.5);
\draw[-][thick](-1,0)--(-1.5,0.5);
\draw[-][thick](-1.5,0.5)--(-2,1);
\draw[->][thick](-2,-1)--(-1.5,-0.5);
\draw[->][thick](1,0)--(1.25,0.25);
\draw[->][thick](1,0)--(1.25,-0.25);
\draw[->][thick](1.5,0.5)--(1.75,0.25);
\draw[->][thick](1.5,-0.5)--(1.75,-0.25);
\draw[-][thick](-1,0)--(-1.5,-0.5);
\draw[-][thick](-1.5,-0.5)--(-2,-1);
\draw [-,blue, dashed] (2,0) to [out=90,in=-160] (4,1.5);
\draw [-,blue, dashed] (2,0) to [out=-90,in=160] (4,-1.5);
\draw [-,blue, dashed] (-1,0) to [out=90,in=-20] (-3,1.5);
\draw [-,blue, dashed] (-1,0) to [out=-90,in=20] (-3,-1.5);
\draw [-,blue, dashed] (1,0) to [out=100,in=-80] (0.6,1);
\draw [-,blue, dashed] (1,0) to [out=-100,in=90] (0.6,-1);
\draw[-][blue, dashed](0.6,1)to [out=100,in=-90] (0.5,2);
\draw[-][blue, dashed](0.6,-1)to [out=90,in=-90] (0.5,-2);
\draw[->][dashed](1,0)--(2,0);
\draw[fill] (1,0) circle [radius=0.035];
\draw[fill] (0.5,0) circle [radius=0.035]node[below]{$0$};
\draw[fill,red] (2,0) circle [radius=0.035]node[below]{$\lambda_{1}$};
\draw[fill,red] (1,0) circle [radius=0.035]node[below]{$\lambda_{2}$};
\draw[fill,red] (-1,0) circle [radius=0.035]node[below]{$\lambda_{3}$};
\draw[->][dashed](2,0)--(3,0);
\draw[fill] (2,1.5) node{$\Omega_{0}$};
\draw[fill] (2,-1.5) node{$\Omega^{*}_{0}$};
\draw[fill] (0,0.5) node{$\Omega_{1}$};
\draw[fill] (0,-0.5) node{$\Omega^{*}_{1}$};
\draw[fill] (1.5,-0.25) node{$\Omega^{*}_{2}$};
\draw[fill] (1.5,0.25) node{$\Omega_{2}$};
\draw[fill] (-2,-0.5) node{$\Omega^{*}_{2}$};
\draw[fill] (-2,0.5) node{$\Omega_{2}$};
\draw[fill] (3,0.5) node{$\Omega_{1}$};
\draw[fill] (3,-0.5) node{$\Omega^{*}_{1}$};
\draw[-][thick](1.5,-0.5)--(2,0);
\draw[->][thick](2,0)--(2.5,0.5);
\draw[-][thick](2.5,0.5)--(3,1);
\draw[->][thick](2,0)--(2.5,-0.5);
\draw[-][thick](2.5,-0.5)--(3,-1);
\draw[fill] (-1.4,0.7) node{$\Upsilon_{2}$};
\draw[fill] (-1.4,-0.7) node{$\Upsilon^{*}_{2}$};
\draw[fill] (-0.6,0.7) node{$\Upsilon_{1}$};
\draw[fill] (-0.6,-0.7) node{$\Upsilon^{*}_{1}$};
\draw[fill] (0.6,0.7) node{$\Upsilon_{1}$};
\draw[fill] (0.6,-0.7) node{$\Upsilon^{*}_{1}$};
\draw[fill] (1.2,0.5) node{$\Upsilon_{2}$};
\draw[fill] (1.2,-0.5) node{$\Upsilon^{*}_{2}$};
\draw[fill] (1.9,0.5) node{$\Upsilon_{2}$};
\draw[fill] (1.9,-0.5) node{$\Upsilon^{*}_{2}$};
\draw[fill] (2.4,0.7) node{$\Upsilon_{1}$};
\draw[fill] (2.4,-0.7) node{$\Upsilon^{*}_{1}$};
\end{tikzpicture}
          \caption{ \footnotesize The regions $\Omega_{j}$, $\Omega^{*}_{j}$, $j=0,1,2$ and the contours
$\Upsilon=\Upsilon_{j}\cup\Upsilon^{*}_{j}$, $j=1,2$.}
\end{figure}
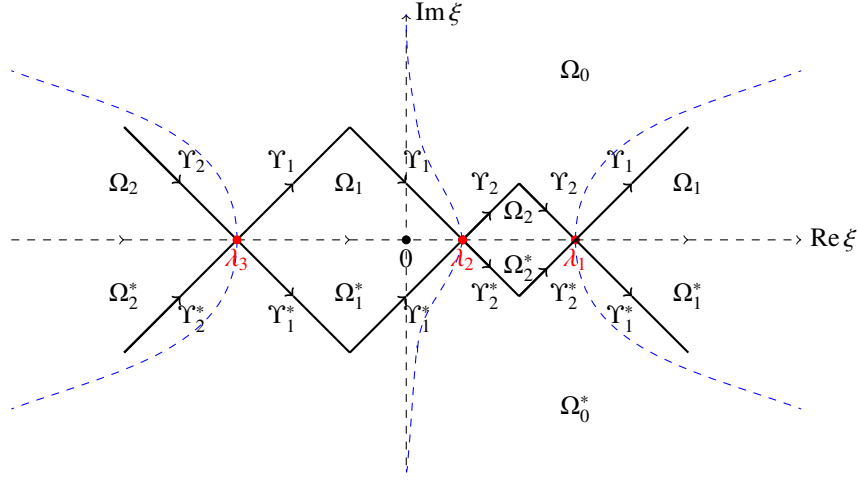

\subsection{Regular RH problem}\quad

In the discussion of this subsection, we will use the BP factor to transform the RH problem \ref{rhp3} with the two residue conditions \eqref{4-19} and \eqref{4-21} into a regular RH problem without residue conditions.

After making the transformation
\begin{align}\label{4-22}
\widehat{M}(x,t,\xi)=B(x,t,\xi)\widehat{M}^{r}(x,t,\xi)\left(
         \begin{array}{cc}
         1 & 0\\
         0 & \frac{\xi-i\xi_{1}}{\xi} \\
         \end{array}
         \right),\quad \xi\in\mathbb{C},
\end{align}
then solving the solution $\widehat{M}(x,t,\xi)$ of RH problem \ref{rhp3} can be converted into solving the solution $\widehat{M}^{r}(x,t,\xi)$ of the regular RH problem. There $B(x,t,\xi)$ has the form $B(x,t,\xi)=I+\frac{i\xi_{1}}{\xi-i\xi_{1}}P(x,t)$, and $B(x,t,\xi)$, $P(x,t)$ are the BP factors. Now we will introduce the regular RH problem as follows
\begin{RHP}\label{rhp5}
Find a matrix function $\widehat{M}^{r}(x,t,\xi)$ admits the following relations
\begin{enumerate}[(i)]
\item $\widehat{M}^{r}(x,t,\xi)$ is analytic for $\xi\in\mathbb{C}\backslash\Upsilon$.
\item Jump conditions:
\begin{equation}\label{4-23}
\widehat{M}^{r}_{+}(x,t,\xi)=\widehat{M}^{r}_{-}(x,t,\xi)\widehat{J}^{r}(x,t,\xi),
\end{equation}
where
\begin{align}\label{4-24}
\widehat{J}^{r}(x,t,\xi)=\left(
         \begin{array}{cc}
         1 & 0\\
         0 & \frac{\xi-i\xi_{1}}{\xi} \\
         \end{array}
         \right)\widehat{J}(x,t,\xi)\left(
         \begin{array}{cc}
         1 & 0\\
         0 & \frac{\xi}{\xi-i\xi_{1}} \\
         \end{array}
         \right),
\end{align}
\begin{align}\label{4-25}
\widehat{J}^{r}(x,t,\xi)=\left\{ \begin{aligned}
         &\left(
         \begin{array}{cc}
         1 & -\frac{r_{2}^{r}(\xi)\delta^{2}(\xi,\mu)}{1+r_{1}^{r}(\xi)r_{2}^{r}(\xi)}e^{-2it\theta}\\
         0 & 1 \\
         \end{array}
         \right)
         ,& \quad &\xi\in \Upsilon_{2},\\
         &\left(
         \begin{array}{cc}
         1 & 0 \\
         -r_{1}^{r}(\xi)\delta^{-2}(\xi,\mu)e^{2it\theta} & 1 \\
         \end{array}
         \right),& \quad &\xi\in \Upsilon_{1},\\
         &\left(
         \begin{array}{cc}
         1 & r_{2}^{r}(\xi)\delta^{2}(\xi,\mu)e^{-2it\theta} \\
         0 & 1 \\
         \end{array}
         \right),& \quad &\xi\in \Upsilon^{*}_{1},\\
         &\left(
         \begin{array}{cc}
         1 & 0 \\
         \frac{r_{1}^{r}(\xi)\delta^{-2}(\xi,\mu)}{1+r_{1}^{r}(\xi)r_{2}^{r}(\xi)}e^{2it\theta} & 1 \\
         \end{array}
         \right),&\quad &\xi\in \Upsilon^{*}_{2},\\
                          \end{aligned} \right.
\end{align}
\begin{align}\label{4-26}
r_{1}^{r}(\xi)=\frac{\xi-i\xi_{1}}{\xi}r_{1}(\xi),\quad
r_{2}^{r}(\xi)=\frac{\xi}{\xi-i\xi_{1}}r_{2}(\xi).
\end{align}
\item Normalization condition at $\xi=\infty$:
\begin{align}\label{4-27}
\widehat{M}^{r}(x,t,\xi)=I+O(\frac{1}{\xi}),\quad \xi\rightarrow\infty.
\end{align}
\item The elements of matrix-value factor $P(x,t)$ are determined by $\widehat{M}^{r}(x,t,\xi)$
\begin{subequations}\label{4-28}
\begin{align}
P_{12}(x,t)&=\frac{u_{1}(x,t)v_{1}(x,t)}{u_{1}(x,t)v_{2}(x,t)-u_{2}(x,t)u_{1}(x,t)},\label{4-28a}\\
P_{21}(x,t)&=-\frac{u_{2}(x,t)v_{2}(x,t)}{u_{1}(x,t)v_{2}(x,t)-u_{2}(x,t)u_{1}(x,t)},\label{4-28b}
\end{align}
\end{subequations}
there the elements $P_{ij}(x,t)$ $(i,j=1,2)$ represent the position of row $i$ and column $j$ in matrix $P_{ij}(x,t)$, and $u(x,t)=(u_{1}(x,t),u_{2}(x,t))^{T}$, $v(x,t)=(v_{1}(x,t),v_{2}(x,t))^{T}$ are given by
\begin{subequations}\label{4-29}
\begin{align}
u(x,t)&=i\xi_{1}\left[\widehat{M}^{r}(x,t,i\xi_{1})\right]_{1}-c_{1}(x,t)\left[\widehat{M}^{r}(x,t,i\xi_{1})\right]_{2},\label{4-29a}\\
v(x,t)&=i\xi_{1}\left[\widehat{M}^{r}(x,t,0)\right]_{2}+c_{0}(\xi)\left[\widehat{M}^{r}(x,t,0)\right]_{1}.\label{4-29b}
\end{align}
\end{subequations}
\end{enumerate}
\end{RHP}

\begin{proof}
The item (i)-(iii) are easy to check. As for the item (iv), by writing the terms of the transformation \eqref{4-22} in matrices form, and pay attention to the residue conditions \eqref{4-19} and \eqref{4-21}, the relation \eqref{4-28} can be obtained by direct calculation.
\end{proof}

\begin{rem}
It should be noticed that from the relations $r_{1}(\xi)=\overline{r_{1}(-\bar{\xi})}$ and $r_{2}(\xi)=\overline{r_{2}(-\bar{\xi})}$ we can obtain
\begin{align}\label{4-30}
r_{1}^{r}(\xi)=\overline{r_{1}^{r}(-\bar{\xi})},\quad
r_{2}^{r}(\xi)=\overline{r_{2}^{r}(-\bar{\xi})}.
\end{align}
Then by the symmetry $\delta(\xi,\mu)=\overline{\delta(-\bar{\xi},\mu)}$ and \eqref{4-23}, we have the regular RH problem admits the symmetry
\begin{align}\label{4-31}
\widehat{M}^{r}(x,t,\xi)=\overline{\widehat{M}^{r}(x,t,-\bar{\xi})}.
\end{align}
\end{rem}
Then, the relation between the solution $q(x,t)$ and $\widehat{M}^{r}(x,t,\xi)$ can be acquired in the following proposition. From the rough approximation $\widehat{M}^{r}(x,t,\xi)\approx I$ as $t\rightarrow\infty$, the rough error estimate of large-$t$ asymptotics of $q(x,t)$ can be obtained.

\begin{prop}\label{prop2}
The solution $q(x,t)$ of the Cauchy problem \eqref{e1} and \eqref{e4} can be indicated in terms of
\begin{subequations}\label{4-32}
\begin{align}
q(x,t)&=-2\xi_{1}P_{12}(x,t)+2i\lim_{\xi\to\infty}\xi\widehat{M}^{r}_{12}(x,t,\xi),\quad x>0,\label{4-32a}\\
q(x,t)&=-2\xi_{1}\overline{P_{21}(-x,t)}-2i\lim_{\xi\to\infty}\xi\overline{\widehat{M}^{r}_{21}(-x,t,\xi)},\quad x<0.\label{4-32b}
\end{align}
\end{subequations}
Furthermore, when $t\rightarrow\infty$, we have a rough estimate about $q(x,t)$ as follows
\begin{subequations}\label{4-33}
\begin{align}
q(x,t)&=A\delta^{2}(0,\mu)+o(1),\quad x>0,\\
q(x,t)&=o(1),\quad x<0,
\end{align}
\end{subequations}
along any ray $\mu=\frac{x}{t}=const\in\left(\epsilon,\sqrt{\frac{1}{27\gamma}}+\epsilon\right)$ or $\mu\in\left(-\sqrt{\frac{1}{27\gamma}}+\epsilon,\epsilon\right)$.
\end{prop}

\begin{proof}
Taking into account the transformation \eqref{4-22} and $B(x,t,\xi)=I+\frac{i\xi_{1}}{\xi-i\xi_{1}}P(x,t)$, we gain the asymptotic expansion of $\widehat{M}(x,t,\xi)$ at $\xi\rightarrow\infty$
\begin{align}\label{4-34}
\widehat{M}(x,t,\xi)=\left(
         \begin{array}{cc}
         1 & 0 \\
         0 & 1-\frac{i\xi_{1}}{\xi} \\
         \end{array}
         \right)+\frac{i\xi_{1}}{\xi-i\xi_{1}}P(x,t)
         +\frac{\widehat{M}^{r}_{1}(x,t)}{\xi}+O(\xi^{-2}),\quad \xi\rightarrow\infty,
\end{align}
where $\widehat{M}^{r}(x,t,\xi)=I+\frac{\widehat{M}^{r}_{1}(x,t)}{\xi}+O(\xi^{-2})$, $\xi\rightarrow\infty$. Then combine relation \eqref{u35}, \eqref{u36} and recall the transformations \eqref{4-8}, \eqref{4-15} to obtain \eqref{4-32}.

The rough approximation $\widehat{M}^{r}(x,t,\xi)\approx I$ as $t\rightarrow\infty$ indicate that, for $x>0$,
\begin{align}\label{4-35}
\left(
         \begin{array}{cc}
         u_{1}(x,t) \\
         u_{2}(x,t) \\
         \end{array}
         \right)\approx
         \left(
         \begin{array}{cc}
         i\xi_{1} \\
         0 \\
         \end{array}
         \right),\quad
\left(
         \begin{array}{cc}
         v_{1}(x,t) \\
         v_{2}(x,t) \\
         \end{array}
         \right)\approx
\left(
         \begin{array}{cc}
         c_{0}(\mu) \\
         i\xi_{1} \\
         \end{array}
         \right),
\end{align}
hence,
\begin{align}\label{4-36}
q(x,t)\approx
-2\xi_{1}\frac{i\xi_{1}c_{0}(\mu)}{-\xi^{2}_{1}+c_{0}(\mu)c_{1}(x,t)}\approx
2ic_{0}(\mu)=A\delta^{2}(0,\mu).
\end{align}
For $x<0$,
\begin{align}\label{4-37}
\left(
         \begin{array}{cc}
         \overline{u_{1}(-x,t)} \\
         \overline{u_{2}(-x,t)} \\
         \end{array}
         \right)\approx
         \left(
         \begin{array}{cc}
         i\xi_{1} \\
         0 \\
         \end{array}
         \right),\quad
\left(
         \begin{array}{cc}
         \overline{v_{1}(-x,t)} \\
         \overline{v_{2}(-x,t)} \\
         \end{array}
         \right)\approx
\left(
         \begin{array}{cc}
         c_{0}(\mu) \\
         i\xi_{1} \\
         \end{array}
         \right),
\end{align}
therefore,
\begin{align}\label{4-38}
q(x,t)\approx
2\xi_{1}\frac{-\overline{c_{1}(-x,t)}(-i\xi_{1})}{-\xi^{2}_{1}+\overline{c_{0}(-\mu)}
\overline{c_{1}(-x,t)}}\approx 0.
\end{align}
\end{proof}

\subsection{Local models near the saddle points}\quad

It is obviously that we can prove the jump matrix $\widehat{J}^{r}(x,t,\xi)$ approach to identity matrix as $t\rightarrow\infty$, but the neighborhood of saddle points $\lambda_{1}$, $\lambda_{2}$ and $\lambda_{3}$ need to be additionally analysed. The goal of this subsection is to gain a good approximation of function $\widehat{M}^{r}(x,t,\xi)$ near these saddle points by the parabolic cylinder functions to obtain the long-time asymptotics of $\widehat{M}^{r}(x,t,\xi)$.

In the following discussion, we take $r_{j}(\xi)\neq0$ $(j=1,2)$. If there is as least one of the $r_{j}(\xi)$, $j=1,2$ equal to zero, then $v(\lambda_{j})=0$. It is enough to estimate the large-time asymptotic solution of $\widehat{J}^{r}(x,t,\xi)$ at $\xi=0$, $\xi=i\xi_{1}$ and $\xi=\infty$ by considering \eqref{4-32}. Furthermore, we find that this RH problem is similar to the case under zero boundary condition, hence we will refer to the idea of \cite{Wang-Liu-2022} to conduct the following analysis.

The scaling transformation are defined as follows
\begin{subequations}\label{4-39}
\begin{align}
\xi&\rightarrow\lambda_{1}+\frac{\tau}{\sqrt{4t(48\gamma\lambda_{1}^{2}-1)}},\\
\xi&\rightarrow\lambda_{2}+\frac{\tau}{\sqrt{4t(1-48\gamma\lambda_{2}^{2})}},\\
\xi&\rightarrow\lambda_{3}+\frac{\tau}{\sqrt{4t(48\gamma\lambda_{3}^{2}-1)}}.
\end{align}
\end{subequations}
Besides, the function $\delta$ can be written as
\begin{align}\label{4-40}
\begin{split}
\delta(\xi(\tau),\mu)&=\left(\frac{\tau}{\sqrt{4t(48\gamma\lambda_{1}^{2}-1)}}\right)^{iv(\lambda_{1})}
\left(\sqrt{\frac{48\gamma\lambda_{3}^{2}-1}{1-48\gamma\lambda_{2}^{2}}} \right)^{-iv(\lambda_{1})}e^{\chi_{1}(\xi)},\\
&=\left(\frac{\tau}{\sqrt{4t(48\gamma\lambda_{3}^{2}-1)}}\right)^{iv(\lambda_{3})}
\left(\sqrt{\frac{48\gamma\lambda_{1}^{2}-1}{1-48\gamma\lambda_{2}^{2}}} \right)^{-iv(\lambda_{2})}e^{\chi_{2}(\xi)},\\
&=\left(\frac{\tau}{\sqrt{4t(48\gamma\lambda_{3}^{2}-1)}}\right)^{iv(\lambda_{3})}
\left(\sqrt{\frac{48\gamma\lambda_{1}^{2}-1}{1-48\gamma\lambda_{2}^{2}}} \right)^{-iv(\lambda_{3})}e^{\chi_{3}(\xi)},
\end{split}
\end{align}
there $v(\lambda_{j})$ and $\chi_{j}(\xi)$, $(j=1,2,3)$ are shown in \eqref{4-6} and \eqref{4-7}. We define $D_{\varepsilon}(\lambda_{j})$, $j=1,2,3$ to express the open disk with radius $\varepsilon$ and centered at $\lambda_{j}$. Also we introduce the contours
\begin{align}\label{4-41}
\Upsilon_{\lambda_{j}}=\Upsilon\cap D_{\varepsilon}(\lambda_{j})
=\Upsilon_{1,\varepsilon_{j}}\cup\Upsilon_{2,\varepsilon_{j}}
\cup\Upsilon^{*}_{1,\varepsilon_{j}}\cup\Upsilon^{*}_{2,\varepsilon_{j}},
\end{align}
for $j=1,2,3$, there $\Upsilon=\Upsilon_{1}\cup\Upsilon_{2}\cup\Upsilon^{*}_{1}\cup\Upsilon^{*}_{2}$.

The next step, we introduce the local parametrix $\widehat{M}^{r}_{\lambda_{j}}(x,t,\xi)$, $j=1,2,3$
\begin{align}\label{4-42}
\widehat{M}^{r}_{\lambda_{1}}(x,t,\xi)=\Lambda_{1}(\mu,t)
\widehat{m}^{pc}_{\lambda_{1}}(\lambda_{1},\tau(\xi))\Lambda^{-1}_{1}(\mu,t),
\end{align}
with
\begin{align}\label{l1}
\Lambda_{1}=e^{[\chi_{1}(\xi)+\phi_{1}(\mu,\tau(\xi))]\sigma_{3}}
\left(\frac{1-48\gamma\lambda_{2}^{2}}{4t(48\gamma\lambda_{1}^{2}-1)(48\gamma\lambda_{3}^{2}-1)}\right)
^{\frac{i}{2}v(\lambda_{1})\sigma_{3}},
\end{align}
\begin{align}\label{phi-1}
\begin{split}
\phi_{1}(\mu,\tau(\xi))=&-\frac{i\gamma\tau^{4}}{2t(48\gamma\lambda_{1}^{2}-1)^{2}}
-\frac{4i\gamma\lambda_{1}\tau^{3}}{\sqrt{t(48\gamma\lambda_{1}^{2}-1)^{3}}}\\
&+\frac{i\tau^{2}}{4t(48\gamma\lambda_{1}^{2}-1)}
-\frac{i(16\gamma\lambda^{2}_{1}-t)\lambda_{1}\tau}{\sqrt{t(48\gamma\lambda_{1}^{2}-1)}}
-4\gamma\lambda_{1}^{4}+\frac{1}{2}\lambda_{1}^{2}.
\end{split}
\end{align}
\begin{align}\label{4-43}
\widehat{M}^{r}_{\lambda_{2}}(x,t,\xi)=\Lambda_{2}(\mu,t)
\widehat{m}^{pc}_{\lambda_{2}}(\lambda_{2},\tau(\xi))\Lambda^{-1}_{2}(\mu,t),
\end{align}
with
\begin{align}\label{l2}
\Lambda_{2}=e^{[\chi_{2}(\xi)+\phi_{2}(\mu,\tau(\xi))]\sigma_{3}}
\left(\frac{1}{4t(48\gamma\lambda_{3}^{2}-1)}\right)
^{-\frac{i}{2}v(\lambda_{3})\sigma_{3}}
\left(\frac{1-48\gamma\lambda_{2}^{2}}{48\gamma\lambda_{1}^{2}-1}\right)
^{-\frac{i}{2}v(\lambda_{2})\sigma_{3}},
\end{align}
\begin{align}\label{phi-2}
\begin{split}
\phi_{2}(\mu,\tau(\xi))=&-\frac{i\gamma\tau^{4}}{2t(1-48\gamma\lambda_{2}^{2})^{2}}
-\frac{4i\gamma\lambda_{2}\tau^{3}}{\sqrt{t(1-48\gamma\lambda_{2}^{2})^{3}}}\\
&+\frac{i\tau^{2}}{4t(1-48\gamma\lambda_{2}^{2})}
-\frac{i(16\gamma\lambda^{2}_{2}-t)\lambda_{2}\tau}{\sqrt{t(1-48\gamma\lambda_{2}^{2})}}
-4\gamma\lambda_{2}^{4}+\frac{1}{2}\lambda_{2}^{2}.
\end{split}
\end{align}
\begin{align}\label{4-44}
\widehat{M}^{r}_{\lambda_{3}}(x,t,\xi)=\Lambda_{3}(\mu,t)
\widehat{m}^{pc}_{\lambda_{3}}(\lambda_{3},\tau(\xi))\Lambda^{-1}_{3}(\mu,t),
\end{align}
with
\begin{align}\label{l3}
\Lambda_{3}=e^{[\chi_{3}(\xi)+\phi_{3}(\mu,\tau(\xi))]\sigma_{3}}
\left(\frac{1-48\gamma\lambda_{2}^{2}}{4t(48\gamma\lambda_{3}^{2}-1)(48\gamma\lambda_{1}^{2}-1)}\right)
^{\frac{i}{2}v(\lambda_{3})\sigma_{3}},
\end{align}
\begin{align}\label{phi-3}
\begin{split}
\phi_{3}(\mu,\tau(\xi))=&-\frac{i\gamma\tau^{4}}{2t(48\gamma\lambda_{3}^{2}-1)^{2}}
-\frac{4i\gamma\lambda_{3}\tau^{3}}{\sqrt{t(48\gamma\lambda_{3}^{2}-1)^{3}}}\\
&+\frac{i\tau^{2}}{4t(48\gamma\lambda_{3}^{2}-1)}
-\frac{i(16\gamma\lambda^{2}_{3}-t)\lambda_{3}\tau}{\sqrt{t(48\gamma\lambda_{3}^{2}-1)}}
-4\gamma\lambda_{3}^{4}+\frac{1}{2}\lambda_{3}^{2}.
\end{split}
\end{align}

There the parameterized RH problems $\widehat{m}^{pc}_{\lambda_{j}}(\lambda_{j},\tau(\xi))$, $j=1,2,3$ can be obtained by
\begin{align}\label{4-45}
\left\{ \begin{aligned}
         &\widehat{m}^{pc}_{\lambda_{j},+}(\lambda_{j},\tau)=
         \widehat{m}^{pc}_{\lambda_{j},-}(\lambda_{j},\tau)J^{pc}_{\lambda_{j}}(\tau),
         \quad \tau\in\Sigma^{pc}_{\lambda_{j}},\\
         &\widehat{m}^{pc}_{\lambda_{j}}(\lambda_{j},\tau)\rightarrow I,\quad \tau\rightarrow\infty.
         \end{aligned} \right.
\end{align}
There the RH problem $\widehat{m}^{pc}_{\lambda_{j}}(\lambda_{j},\tau(\xi))$ can be solved explicitly in Appendix A. The jump contours $\Upsilon_{\lambda_{s}}$ $(s=1,3)$ and $\Upsilon_{\lambda_{2}}$ of local parametrix $\widehat{M}^{r}_{\lambda_{j}}(x,t,\xi)$, $j=1,2,3$ are defined in Figure 4. and Figure 5.
\begin{align}\label{4-46}
J^{pc}_{\lambda_{s}}(\tau)=\left\{ \begin{aligned}
         &\left(
         \begin{array}{cc}
         1 & -\frac{r_{2}^{r}(\lambda_{s})}{1+r_{1}^{r}(\lambda_{s})r_{2}^{r}(\lambda_{s})}e^{-\frac{i}{2}\tau^{2}}\tau^{2iv(\lambda_{s})}\\
         0 & 1 \\
         \end{array}
         \right)
         ,&\quad &\tau\in \Sigma_{2,\varepsilon_{s}},\\
         &\left(
         \begin{array}{cc}
         1 & 0 \\
         -r_{1}^{r}(\lambda_{s})e^{\frac{i}{2}\tau^{2}}\tau^{-2iv(\lambda_{s})} & 1 \\
         \end{array}
         \right),& \quad &\tau\in \Sigma_{1,\varepsilon_{s}},\\
         &\left(
         \begin{array}{cc}
         1 & r_{2}^{r}(\lambda_{s})e^{-\frac{i}{2}\tau^{2}}\tau^{2iv(\lambda_{s})} \\
         0 & 1 \\
         \end{array}
         \right),& \quad &\tau\in \Sigma^{*}_{1,\varepsilon_{s}},\\
         &\left(
         \begin{array}{cc}
         1 & 0 \\
         \frac{r_{1}^{r}(\lambda_{s})}{1+r_{1}^{r}(\lambda_{s})r_{2}^{r}(\lambda_{s})}
         e^{\frac{i}{2}\tau^{2}}\tau^{-2iv(\lambda_{s})} & 1 \\
         \end{array}
         \right),& \quad &\tau\in \Sigma^{*}_{2,\varepsilon_{s}}.\\
                          \end{aligned} \right.
\end{align}
\begin{align}\label{4-47}
J^{pc}_{\lambda_{2}}(\tau)=\left\{ \begin{aligned}
         &\left(
         \begin{array}{cc}
         1 & 0 \\
         -r_{1}^{r}(\lambda_{2})e^{-\frac{i}{2}\tau^{2}}\tau^{2iv(\lambda_{2})} & 1 \\
         \end{array}
         \right)
         ,& \quad  &\tau\in \Sigma_{2,\varepsilon_{2}},\\
         &\left(
         \begin{array}{cc}
         1 & -\frac{r_{2}^{r}(\lambda_{2})}{1+r_{1}^{r}(\lambda_{2})r_{2}^{r}(\lambda_{2})}e^{\frac{i}{2}\tau^{2}}\tau^{-2iv(\lambda_{2})} \\
         0 & 1 \\
         \end{array}
         \right),& \quad &\tau\in \Sigma_{1,\varepsilon_{2}},\\
         &\left(
         \begin{array}{cc}
         1 & 0 \\
         \frac{r_{1}^{r}(\lambda_{2})}{1+r_{1}^{r}(\lambda_{2})r_{2}^{r}(\lambda_{2})}
         e^{-\frac{i}{2}\tau^{2}}\tau^{2iv(\lambda_{2})} & 1 \\
         \end{array}
         \right),& \quad &\tau\in \Sigma^{*}_{1,\varepsilon_{2}},\\
         &\left(
         \begin{array}{cc}
         1 & r_{2}^{r}(\lambda_{2})e^{\frac{i}{2}\tau^{2}}\tau^{-2iv(\lambda_{2})} \\
         0 & 1 \\
         \end{array}
         \right),& \quad &\tau\in \Sigma^{*}_{2,\varepsilon_{2}}.\\
                          \end{aligned} \right.
\end{align}

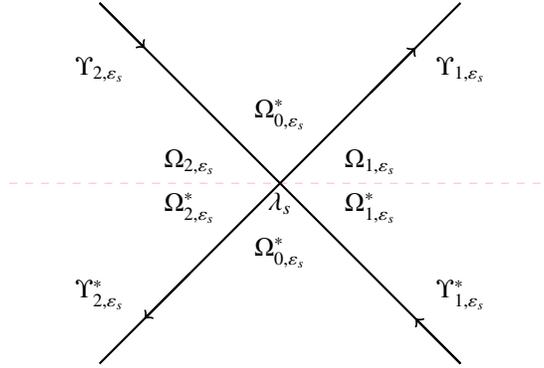
\begin{figure}[htp]
      \centering
\begin{tikzpicture}[scale=0.6]
\draw[-][pink,dashed](-6,0)--(6,0);
\draw[-][thick](-4,-4)--(4,4);
\draw[-][thick](-4,4)--(4,-4);
\draw[->][thick](2,2)--(3,3);
\draw[->][thick](-4,4)--(-3,3);
\draw[->][thick](-2,-2)--(-3,-3);
\draw[->][thick](4,-4)--(3,-3);
\draw[fill] (4,3)node[below]{$\Upsilon_{1,\varepsilon_{s}}$};
\draw[fill] (4,-3)node[above]{$\Upsilon^{*}_{1,\varepsilon_{s}}$};
\draw[fill] (-4,3)node[below]{$\Upsilon_{2,\varepsilon_{s}}$};
\draw[fill] (-4,-3)node[above]{$\Upsilon^{*}_{2,\varepsilon_{s}}$};
\draw[fill] (0,0)node[below]{$\lambda_{s}$};
\draw[fill] (2,0)node[below]{$\Omega^{*}_{1,\varepsilon_{s}}$};
\draw[fill] (2,0)node[above]{$\Omega_{1,\varepsilon_{s}}$};
\draw[fill] (0,-1)node[below]{$\Omega^{*}_{0,\varepsilon_{s}}$};
\draw[fill] (0,1)node[above]{$\Omega^{*}_{0,\varepsilon_{s}}$};
\draw[fill] (-2,0)node[below]{$\Omega^{*}_{2,\varepsilon_{s}}$};
\draw[fill] (-2,0)node[above]{$\Omega_{2,\varepsilon_{s}}$};
\end{tikzpicture}
          \caption{ \footnotesize The jump contours and domains of the local parametrix $\widehat{M}^{r}_{\lambda_{s}}(x,t,\xi)$ $(s=1,3)$.}
\end{figure}

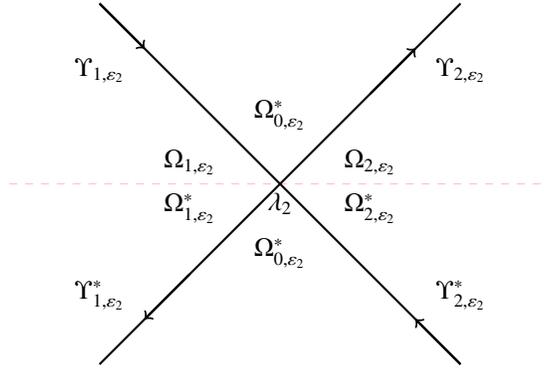
\begin{figure}[H]
      \centering
\begin{tikzpicture}[scale=0.6]
\draw[-][pink,dashed](-6,0)--(6,0);
%\draw[-][dashed](-3,0)--(-2,0);
%\draw[-][dashed](-2,0)--(-1,0);
%\draw[-][dashed](-1,0)--(0,0);
%\draw[-][dashed](0,0)--(1,0);
%\draw[-][dashed](1,0)--(2,0);
%\draw[-][dashed](2,0)--(3,0);
%\draw[-][dashed](3,0)--(4,0);
%\draw[-][thick](0,4)--(0,3);
%\draw[-][thick](0,3)--(0,2);
%\draw[-][thick](0,2)--(0,1);
%\draw[-][thick](0,1)--(0,0);
%\draw[-][thick](0,-4)--(0,-3);
%\draw[-][thick](0,-3)--(0,-2);
%\draw[-][thick](0,-2)--(0,-1);
%\draw[-][thick](0,-1)--(0,-0);
\draw[-][thick](-4,-4)--(4,4);
\draw[-][thick](-4,4)--(4,-4);
\draw[->][thick](2,2)--(3,3);
\draw[->][thick](-4,4)--(-3,3);
\draw[->][thick](-2,-2)--(-3,-3);
\draw[->][thick](4,-4)--(3,-3);
\draw[fill] (4,3)node[below]{$\Upsilon_{2,\varepsilon_{2}}$};
\draw[fill] (4,-3)node[above]{$\Upsilon^{*}_{2,\varepsilon_{2}}$};
\draw[fill] (-4,3)node[below]{$\Upsilon_{1,\varepsilon_{2}}$};
\draw[fill] (-4,-3)node[above]{$\Upsilon^{*}_{1,\varepsilon_{2}}$};
\draw[fill] (0,0)node[below]{$\lambda_{2}$};
\draw[fill] (2,0)node[below]{$\Omega^{*}_{2,\varepsilon_{2}}$};
\draw[fill] (2,0)node[above]{$\Omega_{2,\varepsilon_{2}}$};
\draw[fill] (0,-1)node[below]{$\Omega^{*}_{0,\varepsilon_{2}}$};
\draw[fill] (0,1)node[above]{$\Omega^{*}_{0,\varepsilon_{2}}$};
\draw[fill] (-2,0)node[below]{$\Omega^{*}_{1,\varepsilon_{2}}$};
\draw[fill] (-2,0)node[above]{$\Omega_{1,\varepsilon_{2}}$};
\end{tikzpicture}
     \caption{ \footnotesize The jump contours and domains of the local parametrix $\widehat{M}^{r}_{\lambda_{2}}(x,t,\xi)$.}
\end{figure}

$\widehat{m}^{pc}_{\lambda_{j}}(\lambda_{j},\tau(\xi))$, $j=1,2,3$ which can be directly solved by the parabolic cylindrical function are defined by
\begin{align}\label{4-48}
\begin{split}
\left\{ \begin{aligned}
\widehat{m}^{pc}_{\lambda_{j}}(\lambda_{j},\tau)&=
m_{\lambda_{j}}(\lambda_{j},\tau)(G^{l}_{\lambda_{j}})^{-1}(\lambda_{j},\tau),
\quad \tau\in\Omega_{l,\varepsilon_{j}},l=0,1,2,\\
\widehat{m}^{pc}_{\lambda_{j}}(\lambda_{j},\tau)&=
m_{\lambda_{j}}(\lambda_{j},\tau)(G^{l*}_{\lambda_{j}})^{-1}(\lambda_{j},\tau),
\quad \tau\in\Omega^{*}_{l,\varepsilon_{j}},l=0,1,2.
\end{aligned} \right.
\end{split}
\end{align}
We assume that $G^{0}_{\lambda_{s}}=G^{0*}_{\lambda_{s}}=e^{-\frac{i}{4}\tau^{2}\sigma_{3}}\tau^{iv(\lambda_{s})\sigma_{3}}$,
$s=1,3$ and $G^{0}_{\lambda_{2}}=G^{0*}_{\lambda_{2}}=e^{\frac{i}{4}\tau^{2}\sigma_{3}}\tau^{-iv(\lambda_{2})\sigma_{3}}$,
\begin{align*}
\begin{split}
&G^{1}_{\lambda_{s}}=G^{0}_{\lambda_{s}}
\left(
         \begin{array}{cc}
         1 & 0 \\
         -r^{r}_{1}(\lambda_{s}) & 1 \\
         \end{array}
         \right),\quad
G^{1*}_{\lambda_{s}}=G^{0*}_{\lambda_{s}}
\left(
         \begin{array}{cc}
         1 & r^{r}_{2}(\lambda_{s}) \\
         0 & 1 \\
         \end{array}
         \right),\\
&G^{2}_{\lambda_{s}}=G^{0}_{\lambda_{s}}
\left(
         \begin{array}{cc}
         1 & -\frac{r^{r}_{2}(\lambda_{s})}{1+r^{r}_{1}(\lambda_{s})r^{r}_{2}(\lambda_{s})} \\
         0 & 1 \\
         \end{array}
         \right),\quad
G^{2*}_{\lambda_{s}}=G^{0*}_{\lambda_{s}}
\left(
         \begin{array}{cc}
         1 & 0 \\
         \frac{r^{r}_{1}(\lambda_{s})}{1+r^{r}_{1}(\lambda_{s})r^{r}_{2}(\lambda_{s})} & 1 \\
         \end{array}
         \right),\\
\end{split}
\end{align*}
\begin{align*}
\begin{split}
&G^{1}_{\lambda_{2}}=G^{0}_{\lambda_{2}}
\left(
         \begin{array}{cc}
         1 & -\frac{r^{r}_{2}(\lambda_{2})}{1+r^{r}_{1}(\lambda_{2})r^{r}_{2}(\lambda_{2})} \\
         0 & 1 \\
         \end{array}
         \right),\quad
G^{1*}_{\lambda_{s}}=G^{0*}_{\lambda_{s}}
\left(
         \begin{array}{cc}
         1 & 0 \\
         \frac{r^{r}_{1}(\lambda_{2})}{1+r^{r}_{1}(\lambda_{2})r^{r}_{2}(\lambda_{2})} & 1 \\
         \end{array}
         \right),\\
&G^{2}_{\lambda_{2}}=G^{0}_{\lambda_{2}}
\left(
         \begin{array}{cc}
         1 & 0 \\
         -r^{r}_{1}(\lambda_{2}) & 1 \\
         \end{array}
         \right),\quad
G^{2*}_{\lambda_{2}}=G^{0*}_{\lambda_{2}}
\left(
         \begin{array}{cc}
         1 & r^{r}_{2}(\lambda_{2}) \\
         0 & 1 \\
         \end{array}
         \right).
\end{split}
\end{align*}
The functions $m_{\lambda_{j}}(\lambda_{j},\tau)$, $j=1,2,3$ admit the following RH problem
\begin{align}\label{4-49}
\begin{split}
\left\{ \begin{aligned}
m_{\lambda_{j},+}(\lambda_{j},\tau)&=m_{\lambda_{j},-}(\lambda_{j},\tau)J_{j}(\lambda_{j}),
\quad \tau\in\mathbb{R},\\
m_{\lambda_{j}}(\lambda_{j},\tau)&=(I+O(\tau^{-1}))
e^{(-1)^{j}\frac{i}{4}\tau^{2}\sigma_{3}}\tau^{(-1)^{j+1}iv(\lambda_{j})\sigma_{3}},\quad\tau\rightarrow\infty,
\end{aligned} \right.
\end{split}
\end{align}
with
\begin{align}\label{4-50}
J_{j}(\lambda_{j})=\left(
         \begin{array}{cc}
         1+r^{r}_{1}(\lambda_{j})r^{r}_{2}(\lambda_{j}) & -r^{r}_{2}(\lambda_{j}) \\
         -r^{r}_{1}(\lambda_{j}) & 1 \\
         \end{array}
         \right).
\end{align}
In addition, the RH problem $\widehat{m}^{pc}_{\lambda_{j}}(\lambda_{j},\tau)$, $j=1,2,3$ admit the asymptotic behavior as $\tau\rightarrow\infty$
\begin{align}\label{4-51}
\widehat{m}^{pc}_{\lambda_{j}}(\lambda_{j},\tau)=I+\frac{i}{\tau}
\left(
         \begin{array}{cc}
         0 & \beta^{r}_{j}(\lambda_{j}) \\
         -\gamma^{r}_{j}(\lambda_{j}) & 0 \\
         \end{array}
         \right)+O(\tau^{-2}),\quad \tau\rightarrow\infty,
\end{align}
with
\begin{subequations}\label{4-52}
\begin{align}
&\beta^{r}_{s}(\lambda_{s})=-\frac{\sqrt{2\pi}e^{-\frac{\pi}{2}v(\lambda_{s})}e^{\frac{i\pi}{4}}}
{r^{r}_{1}(\lambda_{s})\Gamma(-iv(\lambda_{s}))},\quad s=1,3,\\
&\gamma_{s}(\lambda_{s})=-\frac{\sqrt{2\pi}e^{-\frac{\pi}{2}v(\lambda_{s})}e^{-\frac{i\pi}{4}}}
{r^{r}_{2}(\lambda_{s})\Gamma(iv(\lambda_{s}))},\quad s=1,3.
\end{align}
\end{subequations}
From the symmetry \eqref{4-31}, we have
\begin{align}\label{4-53}
\widehat{m}^{pc}_{\lambda_{2}}(\lambda_{2},\tau)
=\overline{\widehat{m}^{pc}_{\lambda_{1}}(\lambda_{1},-\bar{\tau})},
\end{align}
then
\begin{align}\label{4-54}
\beta^{r}_{2}(\lambda_{2})=\overline{\beta^{r}_{1}(\lambda_{2})},\quad
\gamma^{r}_{2}(\lambda_{2})=\overline{\gamma^{r}_{1}(\lambda_{2})}.
\end{align}

\subsection{The long-time asymptotic behavior}\quad

The purpose of this subsection is to establish the explicit long-time asymptotic expression of the nonlocal LPD equation. After acquiring the local parametrix $\widehat{M}^{r}_{\lambda_{j}}(x,t,\xi)$, $j=1,2,3$, we introduce $\breve{M}^{r}(x,t,\xi)$ as follows
\begin{align}\label{4-55}
\breve{M}^{r}(x,t,\xi)=\left\{ \begin{aligned}
         &\widehat{M}^{r}(x,t,\xi)(\widehat{M}^{r}_{\lambda_{1}})^{-1}(x,t,\xi),
         &\quad &|\xi-\lambda_{1}|<\varepsilon,\\
         &\widehat{M}^{r}(x,t,\xi)(\widehat{M}^{r}_{\lambda_{2}})^{-1}(x,t,\xi),
         &\quad &|\xi-\lambda_{2}|<\varepsilon,\\
         &\widehat{M}^{r}(x,t,\xi)(\widehat{M}^{r}_{\lambda_{3}})^{-1}(x,t,\xi),
         &\quad &|\xi-\lambda_{3}|<\varepsilon,\\
         &\widehat{M}^{r}(x,t,\xi),
         &\quad &elsewhere,
         \end{aligned} \right.
\end{align}
here $\varepsilon$ is small enough to make $|\lambda_{j}|>\varepsilon$ and $|i\xi_{1}-\lambda_{j}|>\varepsilon$, we define the jump contour $\breve{\Upsilon}=\Upsilon\cup\partial D_{\varepsilon}(\lambda_{1})\cup\partial D_{\varepsilon}(\lambda_{2})\cup\partial D_{\varepsilon}(\lambda_{3})$ of $\breve{M}(x,t,\xi)$ shown in Figure 6. There we also define $\Upsilon_{\varepsilon}=[\Upsilon\cap D_{\varepsilon}(\lambda_{1})]\cup[\Upsilon\cap D_{\varepsilon}(\lambda_{2})]\cup[\Upsilon\cap D_{\varepsilon}(\lambda_{3})]$ and the function $\breve{M}^{r}(x,t,\xi)$ admits the following RH problem
\begin{RHP}\label{rhp6}
Find a matrix function $\breve{M}^{r}(x,t,\xi)$ satisfy the following relations
\begin{enumerate}[(i)]
\item $\breve{M}^{r}(x,t,\xi)$ is analytic for $\xi\in\mathbb{C}\backslash\breve{\Upsilon}$.
\item Jump conditions:
\begin{equation}\label{4-56}
\breve{M}^{r}_{+}(x,t,\xi)=\breve{M}^{r}_{-}(x,t,\xi)\breve{J}(x,t,\xi),
\end{equation}
where
\begin{align}\label{4-57}
\breve{J}(x,t,\xi)=\left\{ \begin{aligned}
         &\widehat{M}^{r}_{\lambda_{1}}(x,t,\xi)
         \widehat{J}^{r}(x,t,\xi)(\widehat{M}^{r}_{\lambda_{1}})^{-1}(x,t,\xi),
         & \quad &\xi\in \Upsilon\cap D_{\varepsilon}(\lambda_{1}),\\
         &\widehat{M}^{r}_{\lambda_{2}}(x,t,\xi)
         \widehat{J}^{r}(x,t,\xi)(\widehat{M}^{r}_{\lambda_{2}})^{-1}(x,t,\xi),
         & \quad &\xi\in \Upsilon\cap D_{\varepsilon}(\lambda_{2}),\\
         &\widehat{M}^{r}_{\lambda_{3}}(x,t,\xi)
         \widehat{J}^{r}(x,t,\xi)(\widehat{M}^{r}_{\lambda_{3}})^{-1}(x,t,\xi),
         & \quad &\xi\in \Upsilon\cap D_{\varepsilon}(\lambda_{3}),\\
         &(\widehat{M}^{r}_{\lambda_{1}})^{-1}(x,t,\xi),& \quad &\xi\in \partial D_{\varepsilon}(\lambda_{1}),\\
         &(\widehat{M}^{r}_{\lambda_{2}})^{-1}(x,t,\xi),& \quad &\xi\in \partial D_{\varepsilon}(\lambda_{2}),\\
         &(\widehat{M}^{r}_{\lambda_{3}})^{-1}(x,t,\xi),& \quad &\xi\in \partial D_{\varepsilon}(\lambda_{3}),\\
         &\widehat{J}^{r}(x,t,\xi), &\quad &\xi\in \Upsilon\backslash \Upsilon_{\varepsilon}.\\
                          \end{aligned} \right.
\end{align}
\item Normalization condition at $\xi=\infty$:
\begin{align}\label{4-58}
\breve{M}^{r}(x,t,\xi)=I+O(\frac{1}{\xi}),\quad \xi\rightarrow\infty.
\end{align}
\end{enumerate}
\end{RHP}

\begin{figure}[H]
      \centering
\begin{tikzpicture}[scale=1.5]
\draw[-][thick](-1.0,0)node[below]{$\lambda_{3}$};
\draw[->][thick](0,1)--(0.5,0.5);
\draw[-][thick](0.5,0.5)--(1,0);
\draw[->][thick](0,-1)--(0.5,-0.5);
\draw[-][thick](0.5,-0.5)--(1,0);
\draw[-][thick](1.5,-0.5)--(1,0);
\draw[-][thick](1.5,0.5)--(1,0);
\draw[-][thick](2,0)--(1.5,0.5);
\draw[-][thick](0,1)--(-0.5,0.5);
\draw[->][thick](-1,0)--(-0.5,0.5);
\draw[-][thick](0,-1)--(-0.5,-0.5);
\draw[->][thick](-1,0)--(-0.5,-0.5);
\draw[->][thick](-2,1)--(-1.5,0.5);
\draw[-][thick](-1,0)--(-1.5,0.5);
\draw[-][thick](-1.5,0.5)--(-2,1);
\draw[->][thick](-2,-1)--(-1.5,-0.5);
\draw[->][thick](1,0)--(1.25,0.25);
\draw[->][thick](1,0)--(1.25,-0.25);
\draw[->][thick](1.5,0.5)--(1.75,0.25);
\draw[->][thick](1.5,-0.5)--(1.75,-0.25);
\draw[-][thick](-1,0)--(-1.5,-0.5);
\draw[-][thick](-1.5,-0.5)--(-2,-1);
\draw [pink, dashed](-3,0)--(4,0);
\draw[-][thick](1,0)node[below]{$\lambda_{2}$};
\draw[fill] (1,0) circle [radius=0.035];
\draw[fill] (0.5,0) circle [radius=0.035]node[below]{$0$};
\draw[fill] (2,0) circle [radius=0.035]node[below]{$\lambda_{1}$};
\draw[fill] (-1,0) circle [radius=0.035];
\draw[-][thick](1.5,-0.5)--(2,0);
\draw[->][thick](2,0)--(2.5,0.5);
\draw[-][thick](2.5,0.5)--(3,1);
\draw[->][thick](2,0)--(2.5,-0.5);
\draw[-][thick](2.5,-0.5)--(3,-1);
\draw(1,0) [blue, line width=1] circle(0.3);
\draw(2,0) [blue, line width=1] circle(0.3);
\draw(-1,0) [blue, line width=1] circle(0.3);
\draw[fill] (-1.4,0.7) node{$\Upsilon_{2}$};
\draw[fill] (-1.4,-0.7) node{$\Upsilon^{*}_{2}$};
\draw[fill] (-0.6,0.7) node{$\Upsilon_{1}$};
\draw[fill] (-0.6,-0.7) node{$\Upsilon^{*}_{1}$};
\draw[fill] (0.6,0.7) node{$\Upsilon_{1}$};
\draw[fill] (0.6,-0.7) node{$\Upsilon^{*}_{1}$};
\draw[fill] (1.2,0.5) node{$\Upsilon_{2}$};
\draw[fill] (1.2,-0.5) node{$\Upsilon^{*}_{2}$};
\draw[fill] (1.9,0.5) node{$\Upsilon_{2}$};
\draw[fill] (1.9,-0.5) node{$\Upsilon^{*}_{2}$};
\draw[fill] (2.4,0.7) node{$\Upsilon_{1}$};
\draw[fill] (2.4,-0.7) node{$\Upsilon^{*}_{1}$};
\end{tikzpicture}
          \caption{ \footnotesize The jump contours $\breve{\Upsilon}$ of $\breve{M}^{r}(x,t,\xi)$.}
\end{figure}
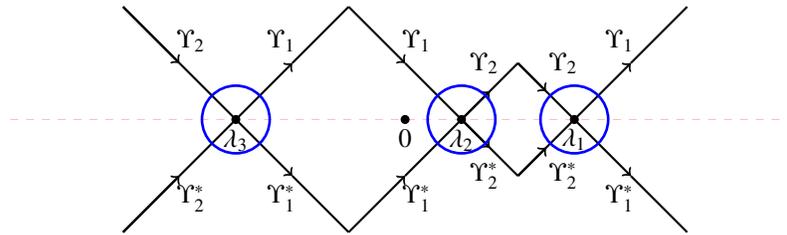

The next step is to obtain the large-$t$ valuation of $\breve{M}^{r}(x,t,\xi)$. Now we define $w(x,t,\xi)=\breve{J}(x,t,\xi)-I$. There from the symmetry relations $\widehat{J}^{r}(x,t,\xi)=\overline{\widehat{J}^{r}(x,t,-\bar{\xi})}$ and $\widehat{M}^{r}(x,t,\xi)=\overline{\widehat{M}^{r}(x,t,-\bar{\xi})}$, we have
\begin{align}\label{4-59}
w(x,t,\xi)=\overline{w(x,t,-\bar{\xi})}.
\end{align}
$w(x,t,\xi)$ admits the following estimates (see \cite{Wang-Liu-2022}), there $1\leq n\leq \infty$,
\begin{align}\label{4-60}
\begin{split}
&\parallel w(x,t,\xi) \parallel_{(L^{1}\cap L^{2}\cap L^{\infty})(\Upsilon\backslash \Upsilon_{\varepsilon})}=O(e^{-ct}),\\
&\parallel w(x,t,\xi) \parallel_{L^{n}(\Upsilon_{\varepsilon})}
=O(t^{-\frac{1}{2}-\frac{1}{2n}+\max\left\{|\im v(\lambda_{1})|,|\im v(\lambda_{2})|,|\im v(\lambda_{3})|\right\}}\ln t).
\end{split}
\end{align}
From the definition in \eqref{l1}, \eqref{l2} and \eqref{l3}, we have the estimates at $t\rightarrow\infty$
\begin{align}\label{4-61}
\begin{split}
&\Lambda_{1}=O(t^{\frac{1}{2}\im v(\lambda_{1})},t^{-\frac{1}{2}\im v(\lambda_{1})}),\\
&\Lambda_{2}=O(t^{-\frac{1}{2}\im v(\lambda_{2})},t^{\frac{1}{2}\im v(\lambda_{2})}),\\
&\Lambda_{3}=O(t^{\frac{1}{2}\im v(\lambda_{3})},t^{-\frac{1}{2}\im v(\lambda_{3})}).
\end{split}
\end{align}
For $\xi\in\partial D_{\varepsilon}(\lambda_{1})$,
\begin{align}\label{4-62}
\begin{split}
w&=(\widehat{M}^{r}_{\lambda_{1}})^{-1}(x,t,\xi)-I\\
&=\Lambda_{1}(\mu,t)\left((\widehat{m}^{\Upsilon}_{\lambda_{1}})^{-1}(\lambda_{1},\tau(\xi))-I\right)
\Lambda^{-1}_{1}(\mu,t)\\
&=\Lambda_{1}(\mu,t)\left(-\frac{i}{\tau}\left(
         \begin{array}{cc}
         0 & \beta^{r}_{1}(\lambda_{1}) \\
         -\gamma^{r}_{1}(\lambda_{1}) & 0 \\
         \end{array}
         \right)+O(\tau^{-2})\right)
\Lambda^{-1}_{1}(\mu,t)\\
&=\frac{\Xi_{1}(\mu,t)}{\sqrt{t}(\xi-\lambda_{1})}+\hat{R}^{1}_{1}(\lambda_{1},t),
\end{split}
\end{align}
with
\begin{align}\label{4-63}
\Xi_{1}=-\frac{i}{2\sqrt{48\gamma\lambda^{2}_{1}-1}}
\left(
         \begin{array}{cc}
         0 & \beta^{r}_{1}(\lambda_{1})e^{2[\chi_{1}+\phi_{1}]}
         F_{1}^{iv(\lambda_{1})} \\
         -\gamma^{r}_{1}(\lambda_{1})e^{-2[\chi_{1}+\phi_{1}]}
         F_{1}^{-iv(\lambda_{1})}  & 0 \\
         \end{array}
         \right),
\end{align}
and
\begin{align}\label{R1-1}
\hat{R}^{1}_{1}(\lambda_{1},t)=\left(O(t^{-1-\im v(\lambda_{1})}),O(t^{-1+\im v(\lambda_{1})})\right),
\end{align}
\begin{align*} F_{1}=\frac{1-48\gamma\lambda^{2}_{2}}{4t(48\gamma\lambda^{2}_{1}-1)(48\gamma\lambda^{2}_{3}-1)}.
\end{align*}
Using the same method, it can be obtained that for $\xi\in\partial D_{\varepsilon}(\lambda_{2})$,
\begin{align}\label{4-67}
w=\frac{\Xi_{2}(\mu,t)}{\sqrt{t}(\xi-\lambda_{2})}+\hat{R}^{2}_{1}(\lambda_{2},t),
\end{align}
with
\begin{align}\label{4-68}
\Xi_{2}=-\frac{i}{2\sqrt{1-48\gamma\lambda^{2}_{2}}}
\left(
         \begin{array}{cc}
         0 & \beta^{r}_{2}(\lambda_{2})e^{2[\chi_{2}+\phi_{2}]}
         \frac{F_{2}^{-iv(\lambda_{2})}}{\tilde{F}_{2}} \\
         -\gamma^{r}_{2}(\lambda_{2})e^{-2[\chi_{2}+\phi_{2}]}
         \frac{F_{2}^{iv(\lambda_{2})}}{\tilde{F}_{2}^{-1}}  & 0 \\
         \end{array}
         \right),
\end{align}
and
\begin{align}\label{R1-2}
\hat{R}^{2}_{1}(\lambda_{2},t)=\left(O(t^{-1+\im v(\lambda_{2})}),O(t^{-1-\im v(\lambda_{2})})\right),
\end{align}
\begin{align*}
F_{2}=\frac{1}{4t(1-48\gamma\lambda^{2}_{2})},\quad
\tilde{F}_{2}=\left(4t(48\gamma\lambda^{2}_{1}-1)\right)^{iv(\lambda_{2})}
\left(\frac{1}{4t(48\gamma\lambda^{2}_{3}-1)}\right)^{-iv(\lambda_{3})}.
\end{align*}
For $\xi\in\partial D_{\varepsilon}(\lambda_{3})$,
\begin{align}\label{4-69}
w=\frac{\Xi_{3}(\mu,t)}{\sqrt{t}(\xi-\lambda_{3})}+\hat{R}^{3}_{1}(\lambda_{3},t),
\end{align}
with
\begin{align}\label{4-70}
\Xi_{3}=-\frac{i}{2\sqrt{48\gamma\lambda^{2}_{3}-1}}
\left(
         \begin{array}{cc}
         0 & \beta^{r}_{3}(\lambda_{3})e^{2[\chi_{3}+\phi_{3}]}
         F_{3}^{iv(\lambda_{3})} \\
         -\gamma^{r}_{3}(\lambda_{3})e^{-2[\chi_{3}+\phi_{3}]}
         F_{3}^{-iv(\lambda_{3})}  & 0 \\
         \end{array}
         \right),
\end{align}
and
\begin{align}\label{R1-3}
\hat{R}^{3}_{1}(\lambda_{3},t)=\left(O(t^{-1-\im v(\lambda_{3})}),O(t^{-1+\im v(\lambda_{3})})\right),
\end{align}
\begin{align*}
F_{3}=\frac{1-48\gamma\lambda^{2}_{2}}{4t(48\gamma\lambda^{2}_{3}-1)(48\gamma\lambda^{2}_{1}-1)}.
\end{align*}
There we also have the estimates at $t\rightarrow\infty$ for $w(x,t,\xi)$ as follows when $\xi\in\breve{\Upsilon}$, the $w^{(l)}$ denotes the column $l$ of $w$,
\begin{align}\label{4-71}
\begin{split}
&\parallel w(x,t,\xi) \parallel_{(L^{1}\cap L^{2})(\breve{\Upsilon})}=O(t^{-\frac{1}{2}+\max\left\{|\im v(\lambda_{1})|,|\im v(\lambda_{2})|,|\im v(\lambda_{3})|\right\}}),\\
&\parallel w(x,t,\xi) \parallel_{L^{\infty}(\breve{\Upsilon})}=O(t^{-\frac{1}{2}+\max\left\{|\im v(\lambda_{1})|,|\im v(\lambda_{2})|,|\im v(\lambda_{3})|\right\}}\ln t),\\
&\parallel w^{(l)}(x,t,\xi) \parallel_{(L^{1}\cap L^{2})(\breve{\Upsilon})}=O(t^{-\frac{1}{2}+(-1)^{l+j}\max\left\{\im v(\lambda_{1}),\im v(\lambda_{2}),\im v(\lambda_{3})\right\}}),\\
&\parallel w^{(l)}(x,t,\xi) \parallel_{L^{\infty}(\breve{\Upsilon})}=O(t^{-\frac{1}{2}+(-1)^{l+j}\max\left\{\im v(\lambda_{1}),\im v(\lambda_{2}),\im v(\lambda_{3})\right\}}\ln t).
\end{split}
\end{align}

Here we define the Cauchy operator $(Cf)(\xi^{\prime})=\frac{1}{2\pi i}\int_{\breve{\Upsilon}}\frac{f(s)}{s-\xi^{\prime}}ds$, $\xi^{\prime}\in\mathbb{C}\setminus\breve{\Upsilon}$, and the integral operator $C_{w}:L^{2}(\breve{\Upsilon})+L^{\infty}(\breve{\Upsilon})\rightarrow L^{2}(\breve{\Upsilon})$ by $C_{w}(f)=C_{-}(fw)$, then, we have
\begin{align}\label{4-72}
\parallel C_{w} \parallel\leq T\parallel w \parallel_{L^{\infty}(\breve{\Upsilon})}
=O(t^{-\frac{1}{2}+\max\left\{|\im v(\lambda_{1})|,|\im v(\lambda_{2})|,|\im v(\lambda_{3})|\right\}}\ln t),\quad t\rightarrow\infty,
\end{align}
with $T$ is a constant, and there $C_{-}$ is a operator $L^{2}(\breve{\Upsilon})\rightarrow L^{2}(\breve{\Upsilon})$. It can be seen $\parallel C_{w} \parallel$ degenerates to zero as $t\rightarrow \infty$, which means $I-C_{w}$ is reversible for the large-time. Therefore, we introduce $\breve{v}(x,t,\xi)-I\in L^{2}(\breve{\Upsilon})$, there $\breve{v}(x,t,\xi)$ is the solution of the Fredholm integral function
\begin{align}\label{4-73}
\breve{v}(x,t,\xi)-I=(I-C_{w})^{-1}C_{w}I,
\end{align}
moreover,
\begin{align}\label{4-74}
\parallel\breve{v}(x,t,\xi)-I\parallel_{L^{2}(\breve{\Upsilon})}\leq T\parallel w(x,t,\xi) \parallel_{(L^{1}\cap L^{2})(\breve{\Upsilon})}.
\end{align}
then we have
\begin{align}\label{4-75}
\parallel\breve{v}(x,t,\xi)-I\parallel_{L^{2}(\breve{\Upsilon})}=O(t^{-\frac{1}{2}+\max\left\{|\im v(\lambda_{1})|,|\im v(\lambda_{2})|,|\im v(\lambda_{3})|\right\}}),\quad t\rightarrow\infty.
\end{align}

From the Beals-Coifman theory, the function $\breve{M}^{r}(x,t,\xi)$ can be denoted by the solution of a singular integral equation, it rely on $w$ and normalization condition \eqref{4-58} and has the form
\begin{align}\label{4-76}
\breve{M}^{r}(x,t,\xi)=I+C(\breve{v}w)=I+\frac{1}{2\pi i}\int_{\breve{\Upsilon}}
\breve{v}(x,t,s)w(x,t,s)\frac{ds}{s-\xi},
\end{align}
then, it can be used to derive the following relation
\begin{align}\label{4-77}
\lim_{\xi\rightarrow\infty}\xi(\breve{M}^{r}(x,t,\xi)-I)=-\frac{1}{2\pi i}\int_{\breve{\Upsilon}}\breve{v}(x,t,s)w(x,t,s)ds.
\end{align}
After taking \eqref{4-62}-\eqref{4-70} and \eqref{4-75} into account, for $j=1,2,3$, one has
\begin{align}\label{4-78}
\begin{split}
&\oint_{|s-\lambda_{j}|=\varepsilon}\breve{v}(x,t,s)w(x,t,s)ds\\
&=\oint_{|s-\lambda_{j}|=\varepsilon}w(x,t,s)ds+\oint_{|s-\lambda_{j}|=\varepsilon}(\breve{v}(x,t,s)-I)w(x,t,s)ds\\
&=\frac{\Xi_{j}(\mu,t)}{\sqrt{t}}\oint_{|s-\lambda_{j}|=\varepsilon}\frac{1}{s-\lambda_{j}}ds
+\hat{R}^{j}_{1}(\lambda_{j},t)+\hat{R}^{j}_{2}(\lambda_{j},t)\\
&=-2\pi i\Xi_{j}^{r}(\mu,t)+\hat{R}^{j}_{1}(\lambda_{j},t)+\hat{R}^{j}_{2}(\lambda_{j},t),
\end{split}
\end{align}
there we assume $\Xi_{j}^{r}(\mu,t)=-\frac{\Xi_{j}(\mu,t)}{\sqrt{t}}$, besides, $\Xi_{j}(\mu,t)$ $(j=1,2,3)$ are given by \eqref{4-63}, \eqref{4-68} and \eqref{4-70}, $\hat{R}^{j}_{1}(\lambda_{j},t)$ $(j=1,2,3)$ are given by \eqref{R1-1}, \eqref{R1-2} and \eqref{R1-3}. Moreover,
\begin{align}\label{4-79}
\begin{split}
\hat{R}^{1}_{2}(\lambda_{1},t)&=\parallel \breve{v}(x,t,s)-I \parallel_{L^{2}(\partial D_{\varepsilon}(\lambda_{1}))}O(\Xi_{1}^{r}(\mu,t))\\
&=(O(t^{-1+|\im v(\lambda_{1})|-\im v(\lambda_{1})}),O(t^{-1+|\im v(\lambda_{1})|+\im v(\lambda_{1})})),
\end{split}
\end{align}
\begin{align}\label{4-80}
\begin{split}
\hat{R}^{2}_{2}(\lambda_{2},t)&=\parallel \breve{v}(x,t,s)-I \parallel_{L^{2}(\partial D_{\varepsilon}(\lambda_{2}))}O(\Xi_{2}^{r}(\mu,t))\\
&=(O(t^{-1+|\im v(\lambda_{2})|+\im v(\lambda_{2})}),O(t^{-1+|\im v(\lambda_{2})|-\im v(\lambda_{2})})),
\end{split}
\end{align}
\begin{align}\label{4-81}
\begin{split}
\hat{R}^{3}_{2}(\lambda_{3},t)&=\parallel \breve{v}(x,t,s)-I \parallel_{L^{2}(\partial D_{\varepsilon}(\lambda_{3}))}O(\Xi_{3}^{r}(\mu,t))\\
&=(O(t^{-1+|\im v(\lambda_{3})|-\im v(\lambda_{3})}),O(t^{-1+|\im v(\lambda_{3})|+\im v(\lambda_{3})})).
\end{split}
\end{align}
After supposing $R(\mu,t)=\sum_{j=1}^{3}\hat{R}_{j}(\mu,t)$, there $\hat{R}_{1}(\mu,t)=\sum_{q=1}^{3}\hat{R}^{q}_{1}(\lambda_{q},t)$, $\hat{R}_{2}(\mu,t)=\sum_{p=1}^{3}\hat{R}^{p}_{2}(\lambda_{p},t)$ and $\hat{R}_{3}(\mu,t)=\hat{R}_{1}(\mu,t)+\hat{R}_{2}(\mu,t)$ possesses the form
\begin{align}\label{4-82}
\hat{R}_{3}(\mu,t)=(O(t^{-1+m^{(1)}-m^{(2)}}),O(t^{-1+m^{(1)}+m^{(2)}})),
\end{align}
with
\begin{align*}
&m^{(1)}=\max\left\{|\im v(\lambda_{1})|,|\im v(\lambda_{2})|,|\im v(\lambda_{3})|\right\},\\
&m^{(2)}=\max\left\{\im v(\lambda_{1}),\im v(\lambda_{2}),\im v(\lambda_{3})\right\}.
\end{align*}
Let $R(\mu,t)=\left(
\begin{array}{cc}
R_{1}(\mu,t) & R_{2}(\mu,t) \\
R_{1}(\mu,t) & R_{2}(\mu,t) \\
\end{array}
\right)$, the estimates of $R_{1}(\mu,t)$ and $R_{2}(\mu,t)$ are shown as follows
\begin{align}\label{1-12}
R_{1}=\left\{ \begin{aligned}
         &O(t^{-1}),
         \quad\quad\quad\quad\quad\quad\quad\quad\quad\quad\quad\quad\quad
          (-1)^{j}\im v(\lambda_{j})>0,\\
         &O(t^{-1}\ln t),
         \quad\quad\quad\quad\quad\quad\quad\quad
          \im v(\lambda_{j})=0,(-1)^{l}\im v(\lambda_{l})\leq0,l\neq j,\\
         &O(t^{-1+2|\im v(\lambda_{1})|}),
         \quad\quad\quad\quad\quad\quad
         \im v(\lambda_{1})>0,\im v(\lambda_{2})\geq0,\im v(\lambda_{3})\leq0,\\
         &O(t^{-1+2|\im v(\lambda_{2})|}),
         \quad\quad\quad\quad\quad\quad
          \im v(\lambda_{1})\leq0,\im v(\lambda_{2})<0,\im v(\lambda_{3})\leq0,\\
         &O(t^{-1+2|\im v(\lambda_{3})|}),
         \quad\quad\quad\quad\quad\quad
          \im v(\lambda_{1})\leq0,\im v(\lambda_{2})\geq0,\im v(\lambda_{3})>0,\\
         &O(t^{-1+2\max\left\{|\im v(\lambda_{1})|,|\im v(\lambda_{2})|\right\}}),
         \quad
         \im v(\lambda_{1})>0,\im v(\lambda_{2})<0,\im v(\lambda_{3})\leq0,\\
         &O(t^{-1+2\max\left\{|\im v(\lambda_{2})|,|\im v(\lambda_{3})|\right\}}),
         \quad
         \im v(\lambda_{1})\leq0,\im v(\lambda_{2})<0,\im v(\lambda_{3})>0,\\
         &O(t^{-1+2\max\left\{|\im v(\lambda_{1})|,|\im v(\lambda_{3})|\right\}}),
         \quad
         \im v(\lambda_{1})>0,\im v(\lambda_{2})\geq0,\im v(\lambda_{3})>0,\\
         &O(t^{-1+2\max\left\{|\im v(\lambda_{1})|,|\im v(\lambda_{2})|,|\im v(\lambda_{3})|\right\}}),
         \quad\quad (-1)^{j}\im v(\lambda_{j})<0,
                          \end{aligned} \right.
\end{align}
\begin{align}\label{1-13}
R_{2}=\left\{ \begin{aligned}
         &O(t^{-1+2\max\left\{|\im v(\lambda_{1})|,|\im v(\lambda_{2})|,|\im v(\lambda_{3})|\right\}}),
         \quad (-1)^{j}\im v(\lambda_{j})>0,\\
         &O(t^{-1+2\max\left\{|\im v(\lambda_{1})|,|\im v(\lambda_{2})|\right\}}),
         \quad\quad
         \im v(\lambda_{1})<0,\im v(\lambda_{2})>0,\im v(\lambda_{3})\geq0,\\
         &O(t^{-1+2\max\left\{|\im v(\lambda_{2})|,|\im v(\lambda_{3})|\right\}}),
         \quad\quad
         \im v(\lambda_{1})\geq0,\im v(\lambda_{2})>0,\im v(\lambda_{3})<0,\\
         &O(t^{-1+2\max\left\{|\im v(\lambda_{1})|,|\im v(\lambda_{3})|\right\}}),
         \quad\quad
         \im v(\lambda_{1})<0,\im v(\lambda_{2})\leq0,\im v(\lambda_{3})<0,\\
         &O(t^{-1+2|\im v(\lambda_{1})|}),
         \quad\quad\quad\quad\quad\quad\quad
         \im v(\lambda_{1})<0,\im v(\lambda_{2})\leq0,\im v(\lambda_{3})\geq0,\\
         &O(t^{-1+2|\im v(\lambda_{2})|}),
         \quad\quad\quad\quad\quad\quad\quad
          \im v(\lambda_{1})\geq0,\im v(\lambda_{2})>0,\im v(\lambda_{3})\geq0,\\
         &O(t^{-1+2|\im v(\lambda_{3})|}),
         \quad\quad\quad\quad\quad\quad\quad
          \im v(\lambda_{1})\geq0,\im v(\lambda_{2})\leq0,\im v(\lambda_{3})<0,\\
         &O(t^{-1}\ln t),
         \quad\quad\quad\quad\quad\quad\quad\quad\quad
          \im v(\lambda_{j})=0,(-1)^{l}\im v(\lambda_{l})\geq0,l\neq j,\\
          &O(t^{-1}),
         \quad\quad\quad\quad\quad\quad\quad\quad\quad\quad\quad\quad\quad
          (-1)^{j}\im v(\lambda_{j})<0.
                          \end{aligned} \right.
\end{align}

Then, combine \eqref{4-78}, equation \eqref{4-76} can be written as
\begin{align}\label{4-88}
\begin{split}
\breve{M}^{r}=&I-\frac{1}{2\pi i}\oint_{|s-\lambda_{1}|=\varepsilon}\frac{\Xi_{1}^{r}(\mu,t)}{(s-\lambda_{1})(s-\xi)}ds-
\frac{1}{2\pi i}\oint_{|s-\lambda_{2}|=\varepsilon}\frac{\Xi_{2}^{r}(\mu,t)}{(s-\lambda_{2})(s-\xi)}ds\\
&-\frac{1}{2\pi i}\oint_{|s-\lambda_{3}|=\varepsilon}\frac{\Xi_{3}^{r}(\mu,t)}{(s-\lambda_{3})(s-\xi)}ds
+R(\mu,t),\quad |\xi-\lambda_{j}|>\varepsilon,j=1,2,3.
\end{split}
\end{align}
By \eqref{4-55}, it can be further obtained that $\widehat{M}^{r}(x,t,\xi)=\breve{M}^{r}(x,t,\xi)$, then
\begin{align}\label{4-89}
\begin{split}
\lim_{\xi\rightarrow\infty}\xi(\widehat{M}^{r}(x,t,\xi)-I)
=\Xi_{1}^{r}(\mu,t)+\Xi_{2}^{r}(\mu,t)+\Xi_{3}^{r}(\mu,t)+R(\mu,t),
\end{split}
\end{align}
and
\begin{subequations}\label{4-90}
\begin{align}
\widehat{M}^{r}(x,t,0)&=I-\frac{\Xi_{1}^{r}(\mu,t)}{\lambda_{1}}-
\frac{\Xi_{2}^{r}(\mu,t)}{\lambda_{2}}-
\frac{\Xi_{3}^{r}(\mu,t)}{\lambda_{3}}+R(\mu,t),\\
\widehat{M}^{r}(x,t,i\xi_{1})&=I-\frac{\Xi_{1}^{r}(\mu,t)}{\lambda_{1}-i\xi_{1}}-
\frac{\Xi_{2}^{r}(\mu,t)}{\lambda_{2}-i\xi_{1}}-
\frac{\Xi_{3}^{r}(\mu,t)}{\lambda_{3}-i\xi_{1}}+R(\mu,t).
\end{align}
\end{subequations}
Next, we estimate the elements $P_{12}(x,t)$ and $P_{21}(x,t)$ of matrix-value factor $P(x,t)$ defined in \eqref{4-28}. According to the definitions of $u(x,t)$ and $v(x,t)$ in \eqref{4-29}, we have
\begin{align}\label{4-91}
\left\{ \begin{aligned}
         &u_{1}(x,t)=i\xi_{1}+R_{1}(\mu,t),\\
         &u_{2}(x,t)=-i\xi_{1}\left(\frac{(\Xi_{1}^{r})_{21}(\mu,t)}{\lambda_{1}-i\xi_{1}}
         +\frac{(\Xi_{2}^{r})_{21}(\mu,t)}{\lambda_{2}-i\xi_{1}}+
         \frac{(\Xi_{3}^{r})_{21}(\mu,t)}{\lambda_{3}-i\xi_{1}}\right)+R_{1}(\mu,t),
         \end{aligned} \right.
\end{align}
\begin{align}\label{4-92}
\left\{ \begin{aligned}
         &v_{1}(x,t)=c_{0}(\mu)-i\xi_{1}\left(\frac{(\Xi_{1}^{r})_{12}(\mu,t)}{\lambda_{1}}
         +\frac{(\Xi_{2}^{r})_{12}(\mu,t)}{\lambda_{2}}+
         \frac{(\Xi_{3}^{r})_{12}(\mu,t)}{\lambda_{3}}\right)+R_{3}(\mu,t),\\
         &v_{2}(x,t)=i\xi_{1}-c_{0}(\mu)\left(\frac{(\Xi_{1}^{r})_{21}(\mu,t)}{\lambda_{1}}
         +\frac{(\Xi_{2}^{r})_{21}(\mu,t)}{\lambda_{2}}+
         \frac{(\Xi_{3}^{r})_{21}(\mu,t)}{\lambda_{3}}\right)+R_{3}(\mu,t),
         \end{aligned} \right.
\end{align}
with $R_{3}(\mu,t)=R_{1}(\mu,t)+R_{2}(\mu,t)$. Furthermore, we gain the estimate
\begin{align}\label{4-93}
\begin{split}
\left\{ \begin{aligned}
u_{1}v_{1}&=i\xi_{1}c_{0}(\mu)+\xi^{2}_{1}\left(\frac{(\Xi_{1}^{r})_{12}(\mu,t)}{\lambda_{1}}
         +\frac{(\Xi_{2}^{r})_{12}(\mu,t)}{\lambda_{2}}+
         \frac{(\Xi_{3}^{r})_{12}(\mu,t)}{\lambda_{3}}\right)+R_{3}(\mu,t),\\
u_{1}v_{2}&=-\xi^{2}_{1}-i\xi_{1}c_{0}(\mu)\left(\frac{(\Xi_{1}^{r})_{21}(\mu,t)}{\lambda_{1}}
         +\frac{(\Xi_{2}^{r})_{21}(\mu,t)}{\lambda_{2}}+
         \frac{(\Xi_{3}^{r})_{21}(\mu,t)}{\lambda_{3}}\right)+R_{3}(\mu,t),\\
u_{2}v_{1}&=-i\xi_{1}c_{0}(\mu)\left(\frac{(\Xi_{1}^{r})_{21}(\mu,t)}{\lambda_{1}-i\xi_{1}}
         +\frac{(\Xi_{2}^{r})_{21}(\mu,t)}{\lambda_{2}-i\xi_{1}}+
         \frac{(\Xi_{3}^{r})_{21}(\mu,t)}{\lambda_{3}-i\xi_{1}}\right)+R_{1}(\mu,t),\\
u_{2}v_{2}&=\xi^{2}_{1}\left(\frac{(\Xi_{1}^{r})_{21}(\mu,t)}{\lambda_{1}-i\xi_{1}}
         +\frac{(\Xi_{2}^{r})_{21}(\mu,t)}{\lambda_{2}-i\xi_{1}}+
         \frac{(\Xi_{3}^{r})_{21}(\mu,t)}{\lambda_{3}-i\xi_{1}}\right)+R_{1}(\mu,t).
         \end{aligned} \right.
\end{split}
\end{align}
After bringing the above equation into formula \eqref{4-28} and performing the direct calculations, we have
\begin{align}\label{4-94}
\begin{split}
P_{12}(x,t)=&-\frac{ic_{0}}{\xi_{1}}-\left(\frac{(\Xi_{1}^{r})_{12}(\mu,t)}{\lambda_{1}}
         +\frac{(\Xi_{2}^{r})_{12}(\mu,t)}{\lambda_{2}}+
         \frac{(\Xi_{3}^{r})_{12}(\mu,t)}{\lambda_{3}}\right)\\
         &+\frac{ic^{2}_{0}}{\xi_{1}}\left(\frac{(\Xi_{1}^{r})_{21}(\mu,t)}{\lambda_{1}(\lambda_{1}-i\xi_{1})}
         +\frac{(\Xi_{2}^{r})_{21}(\mu,t)}{\lambda_{2}(\lambda_{2}-i\xi_{1})}+
         \frac{(\Xi_{3}^{r})_{21}(\mu,t)}{\lambda_{3}(\lambda_{3}-i\xi_{1})}\right)+R_{3}(\mu,t),
\end{split}
\end{align}
\begin{align}\label{4-95}
P_{21}(x,t)=\left(\frac{(\Xi_{1}^{r})_{21}(\mu,t)}{\lambda_{1}-i\xi_{1}}
         +\frac{(\Xi_{2}^{r})_{21}(\mu,t)}{\lambda_{2}-i\xi_{1}}+
         \frac{(\Xi_{3}^{r})_{21}(\mu,t)}{\lambda_{3}-i\xi_{1}}\right)+R_{1}(\mu,t).
\end{align}
It can be seen the expressions of $P_{12}(x,t)$ and $P_{21}(x,t)$ explicitly contain parameter $\xi_{1}$. Then we define
\begin{align}\label{4-96}
(\Xi_{j}^{r})_{12}=\frac{\lambda_{j}}{\lambda_{j}-i\xi_{1}}(\tilde{\Xi}_{j})_{12},\quad
(\Xi_{j}^{r})_{21}=\frac{\lambda_{j}-i\xi_{1}}{\lambda_{j}}(\tilde{\Xi}_{j})_{21},
\end{align}
there we use $r_{l}(\lambda_{j})$ to replace $r^{r}_{l}(\lambda_{j})$, $l=1,2$, $j=1,2,3$ in $\Xi_{j}^{r}$ to obtain $\tilde{\Xi}_{j}$. Therefore, we get that the terms which have the explicit expressions about $\xi_{1}$ now have decayed in the main asymptotic terms. Next, we substitute \eqref{4-94}-\eqref{4-96} and \eqref{4-89} into \eqref{4-32}, Then the long-term asymptotic behaviors of the solutions of the LPD equation at cases $x>0$ and $x<0$ are established. The main results are shown in the following Theorem.

\begin{thm}\label{thm1}
Taking into account the Cauchy problem \eqref{e1} and \eqref{e4}, where the initial data $q_{0}(x)$ is a compact perturbation of the pure step initial data \eqref{q31}: $q_{0}(x)-q_{0A}(x)=0$ for $|x|>\varepsilon$ with some $\varepsilon>0$. Here we assume that the scattering coefficients $a_{1}(\xi)$, $a_{2}(\xi)$ and $b(\xi)$ which are associated to the initial data $q_{0}(x)$ satisfy the following conditions
\begin{enumerate}[(I)]
\item $a_{1}(\xi)$ has a single, simple zero point in $\overline{\mathbb{C}^{+}}$ at $\xi=i\xi_{1}$ and $a_{2}(\xi)$ either has no zero points or has a single, simple zero point in $\overline{\mathbb{C}^{-}}$ at $\xi=0$;
\item $\im v(\lambda_{j})\in(-\frac{1}{2},\frac{1}{2})$, $j=1,2,3$ for
$\im v(\lambda_{j})=\frac{1}{2\pi}\int_{-\infty}^{\lambda_{j}}d arg(1+r_{1}(s)r_{2}(s))$ with $r_{1}(\xi)=\frac{\overline{b(-\bar{\xi})}}{a_{1}(\xi)}$ and $r_{2}(\xi)=\frac{b(\xi)}{a_{2}(\xi)}$.
\end{enumerate}
Under the assumption that the solution $q(x,t)$ satisfying the Cauchy problem \eqref{e1} and \eqref{e4} exists, the long-time asymptotics of $q(x,t)$ along any line $\mu=\frac{x}{t}=const\in\left(-\sqrt{\frac{1}{27\gamma}}+\varepsilon,
\sqrt{\frac{1}{27\gamma}}-\varepsilon\right)$ can be obtained as follows
\begin{enumerate}[(i)]
\item for $x<0$, the long-time asymptotics of $q(x,t)$ reads
\begin{align}\label{1-5}
\begin{split}
q(x,t)=&-\sum_{s=1}^{3}t^{-\frac{1}{2}+(-1)^{s}\im v(-\lambda_{s})}\exp\left\{-2[\overline{\chi_{s}}+\overline{\phi_{s}}(-\mu,\tau)]+(-1)^{s}i\re v(-\lambda_{s})\ln t\right\}H_{s}\\
&+R_{1}(-\mu,t),
\end{split}
\end{align}
with
\begin{subequations}\label{1-6}
\begin{align}
H_{1}&=
\frac{\sqrt{2\pi}e^{-\frac{\pi}{2}\overline{v(-\lambda_{1})}}e^{\frac{i\pi}{4}}}
{\sqrt{48\gamma\lambda^{2}_{1}-1}\overline{r_{2}}(-\lambda_{1})\Gamma(-i\overline{v(-\lambda_{1})})}
\left(\frac{1-48\gamma\lambda^{2}_{2}}{4(48\gamma\lambda^{2}_{1}-1)(48\gamma\lambda^{2}_{3}-1)}\right)^{i\overline{v(-\lambda_{1})}},\\
H_{2}&=
\frac{\sqrt{2\pi}e^{-\frac{\pi}{2}v(-\lambda_{2})}e^{-\frac{i\pi}{4}}}
{\sqrt{1-48\gamma\lambda^{2}_{2}}r_{2}(-\lambda_{2})\Gamma(iv(-\lambda_{2}))}
\left(\frac{1}{4(48\gamma\lambda^{2}_{3}-1)}\right)^{-i\overline{v(-\lambda_{3})}}
\left(\frac{1-48\gamma\lambda^{2}_{2}}{48\gamma\lambda^{2}_{1}-1}\right)^{-i\overline{v(-\lambda_{2})}},\\
H_{3}&=
\frac{\sqrt{2\pi}e^{-\frac{\pi}{2}\overline{v(-\lambda_{3})}}e^{\frac{i\pi}{4}}}
{\sqrt{48\gamma\lambda^{2}_{3}-1}\overline{r_{2}}(-\lambda_{3})\Gamma(-i\overline{v(-\lambda_{3})})}
\left(\frac{1-48\gamma\lambda^{2}_{2}}{4(48\gamma\lambda^{2}_{3}-1)
(48\gamma\lambda^{2}_{1}-1)}\right)^{i\overline{v(-\lambda_{3})}}.
\end{align}
\end{subequations}
\item for $x>0$, based on the value of $\im v(\lambda_{j})$ $(j=1,2,3)$ (here we assume that for all $j=1,2,3$, $\im v(\lambda_{j})$ in the same interval), three possible types asymptotics of $q(x,t)$ are as follows,
\begin{enumerate}[(a)]
\item  $\im v(\lambda_{j})\in I_{1}=(-\frac{1}{2},-\frac{1}{6}]$, $j=1,2,3$,
\begin{align}\label{1-7}
\begin{split}
q(x,t)=&\sum_{s=1}^{3}t^{-\frac{1}{2}+(-1)^{s}\im v(\lambda_{s})}\exp\left\{-2[\chi_{s}+\phi_{s}(\mu,\tau)]-(-1)^{s}i\re v(\lambda_{s})\ln t\right\}N_{s}\\
&+A\delta^{2}(\mu,0)+R_{1}(\mu,t),
\end{split}
\end{align}
\item  $\im v(\lambda_{j})\in I_{2}=(-\frac{1}{6},\frac{1}{6})$, $j=1,2,3$,
\begin{align}\label{1-8}
\begin{split}
q(x,t)=&-\sum_{s=1}^{3}t^{-\frac{1}{2}-(-1)^{s}\im v(\lambda_{s})}\exp\left\{2[\chi_{s}+\phi_{s}(\mu,\tau)]+(-1)^{s}i\re v(\lambda_{s})\ln t\right\}L_{s}\\
&+\sum_{s=1}^{3}t^{-\frac{1}{2}+(-1)^{s}\im v(\lambda_{s})}\exp\left\{-2[\chi_{s}+\phi_{s}(\mu,\tau)]-(-1)^{s}i\re v(\lambda_{s})\ln t\right\}N_{s}\\
&+A\delta^{2}(\mu,0)+R_{3}(\mu,t),
\end{split}
\end{align}
\item  $\im v(\lambda_{j})\in I_{3}=[\frac{1}{6},\frac{1}{2})$, $j=1,2,3$,
\begin{align}\label{1-9}
\begin{split}
q(x,t)=&-\sum_{s=1}^{3}t^{-\frac{1}{2}-(-1)^{s}\im v(\lambda_{s})}\exp\left\{2[\chi_{s}+\phi_{s}(\mu,\tau)]+(-1)^{s}i\re v(\lambda_{s})\ln t\right\}L_{s}\\
&+A\delta^{2}(\mu,0)+R_{2}(\mu,t),
\end{split}
\end{align}
\end{enumerate}
with
\begin{subequations}\label{1-10}
\begin{align}
L_{1}&=
\frac{\sqrt{2\pi}e^{-\frac{\pi}{2}v(\lambda_{1})}e^{\frac{i\pi}{4}}}
{\sqrt{48\gamma\lambda^{2}_{1}-1}r_{1}(\lambda_{1})\Gamma(-iv(\lambda_{1}))}
\left(\frac{1-48\gamma\lambda^{2}_{2}}{4(48\gamma\lambda^{2}_{1}-1)(48\gamma\lambda^{2}_{3}-1)}\right)^{iv(\lambda_{1})},\\
L_{2}&=
\frac{\sqrt{2\pi}e^{-\frac{\pi}{2}\overline{v(\lambda_{2})}}e^{-\frac{i\pi}{4}}}
{\sqrt{1-48\gamma\lambda^{2}_{2}}\overline{r_{1}}(\lambda_{2})\Gamma(i\overline{v(\lambda_{2})})}
\left(\frac{1}{4(48\gamma\lambda^{2}_{3}-1)}\right)^{-iv(\lambda_{3})}
\left(\frac{1-48\gamma\lambda^{2}_{2}}{48\gamma\lambda^{2}_{1}-1}\right)^{-iv(\lambda_{2})},\\
L_{3}&=
\frac{\sqrt{2\pi}e^{-\frac{\pi}{2}v(\lambda_{3})}e^{\frac{i\pi}{4}}}
{\sqrt{48\gamma\lambda^{2}_{3}-1}r_{1}(\lambda_{3})\Gamma(-iv(\lambda_{3}))}
\left(\frac{1-48\gamma\lambda^{2}_{2}}{4(48\gamma\lambda^{2}_{3}-1)
(48\gamma\lambda^{2}_{1}-1)}\right)^{iv(\lambda_{3})},
\end{align}
\end{subequations}
and
\begin{subequations}\label{1-11}
\begin{align}
N_{1}&=
\frac{c^{2}_{0}\sqrt{2\pi}e^{-\frac{\pi}{2}v(\lambda_{1})}e^{-\frac{i\pi}{4}}}
{\sqrt{48\gamma\lambda^{2}_{1}-1}r_{2}(\lambda_{1})\Gamma(iv(\lambda_{1}))\lambda^{2}_{1}}
\left(\frac{1-48\gamma\lambda^{2}_{2}}{4(48\gamma\lambda^{2}_{1}-1)(48\gamma\lambda^{2}_{3}-1)}\right)^{-iv(\lambda_{1})},\\
N_{2}&=
\frac{c^{2}_{0}\sqrt{2\pi}e^{-\frac{\pi}{2}\overline{v(\lambda_{2})}}e^{\frac{i\pi}{4}}}
{\sqrt{1-48\gamma\lambda^{2}_{2}}\overline{r_{2}}(\lambda_{2})\Gamma(-i\overline{v(\lambda_{2})})\lambda^{2}_{2}}
\left(\frac{1}{4(48\gamma\lambda^{2}_{3}-1)}\right)^{-iv(\lambda_{1})}
\left(\frac{1-48\gamma\lambda^{2}_{2}}{48\gamma\lambda^{2}_{1}-1}\right)^{iv(\lambda_{2})},\\
N_{3}&=
\frac{c^{2}_{0}\sqrt{2\pi}e^{-\frac{\pi}{2}v(\lambda_{3})}e^{-\frac{i\pi}{4}}}
{\sqrt{48\gamma\lambda^{2}_{3}-1}r_{2}(\lambda_{3})\Gamma(iv(\lambda_{3}))\lambda^{2}_{3}}
\left(\frac{1-48\gamma\lambda^{2}_{2}}{4(48\gamma\lambda^{2}_{3}-1)
(48\gamma\lambda^{2}_{1}-1)}\right)^{-iv(\lambda_{3})}.
\end{align}
\end{subequations}

There we have the following relations
\begin{align*}
\delta(\mu,0)=\exp\left\{\frac{1}{2\pi i}\left(\int_{-\infty}^{\lambda_{3}}+\int_{\lambda_{2}}^{\lambda_{1}}\right)\frac{\ln(1+r_{1}(s)r_{2}(s))}{s}ds \right\},
\end{align*}
$v(\lambda_{j})$, $j=1,2,3$ can be seen in \eqref{4-7}. Moreover, $\chi_{j}(\xi)$, $j=1,2,3$ have the expressions in \eqref{4-6}. $\phi_{j}(\mu,\tau(\xi))$, $j=1,2,3$ are shown in \eqref{phi-1}, \eqref{phi-2} and \eqref{phi-3}. $\Gamma(\cdot)$ is the Gamma function and the estimates $R_{3}(\mu,t)=R_{1}(\mu,t)+R_{2}(\mu,t)$. The error estimation $R_{1}(\mu,t)$ and $R_{2}(\mu,t)$ are shown in \eqref{1-12} and \eqref{1-13}.
\end{enumerate}
\end{thm}

\begin{rem}\label{rem1}
For case $x>0$, we only consider the conditions all $\im v(\lambda_{j})$ for $j=1,2,3$ in the same interval, but if they are not in the same interval, for instance, $\im v(\lambda_{j})\in I_{1}(I_{3})$, then the solution $q(x,t)$ possesses the items containing $N_{j}(L_{j})$. If $\im v(\lambda_{j})\in I_{2}$, the solution $q(x,t)$ possesses the items containing $N_{j}$ and the items containing $L_{j}$.
\end{rem}

\begin{rem}\label{rem2}
It can be seen that $A\delta^{2}(\mu,0)\rightarrow A$ as $\lambda_{j}\rightarrow\infty$ for $j=1,2,3$, and thus the asymptotic solutions in \eqref{1-7}-\eqref{1-9} about the boundary condition \eqref{e4b} still hold.
\end{rem}

\begin{rem}\label{rem3}
For the case of pure-step initial data in \eqref{q31}, the two assumptions (I) and (II) of theorem \ref{thm1} can be all satisfied. Furthermore, in this case we have $1+r_{1}(\xi)r_{2}(\xi)=\frac{4\xi^{2}}{4\xi^{2}+A^{2}}$, which means $\im v(\cdot)=0$.
\end{rem}

\section{Conclusion}
In this work, the nonlinear steepest descent method of Deift and Zhou is developed to study the long-time asymptotic behavior of nonlocal Lakshmanan-Porsezian-Daniel equation with step-like initial data: $q_0(x)=o(1)$ as $x\rightarrow-\infty$ and $q_0(x)=A+o(1)$ as $x\rightarrow+\infty$, where $A$ is an arbitrary positive constant. Comparing to the existing results in literature \cite{Peng-Chen-2023}, we upgrade the decaying initial value condition that located in the Schwartz space to the non-decaying initial value condition that has a step-like structure. Besides, the nonlocal condition with symmetries $x\rightarrow-x$ and $t\rightarrow t$ is also considered, the difference is that we give the long-time asymptotic behaviors of the LPD equation as $t\rightarrow+\infty$ and $t\rightarrow-\infty$, separately. At the same time, under the step-like initial data condition, there exist singularities in the original RH problem. To transform it into a regular RH problem, we introduce the BP factor, and this will increase the complexity of our work in the subsequent analyses about the long-time asymptotic behaviors.

Firstly, we make the spectral analysis to the Lax pair of LPD equation and acquire the Volterra integral forms of the eigenfunctions. In the direct scattering part, the analytic, symmetric and asymptotic properties of the eigenfunctions and scattering data are given. It is worth noting that, for the asymptotic properties at singularity point zero, we make the asymptotic expansions to the eigenfunctions by assuming some undetermined functions. Then, the relationships between these undetermined functions are acquired by using Volterra integral equations and the symmetry relations of eigenfunctions, and they are used to represent the expansions of scattering data at zero point. In addition to this, we mention the special case of the scattering matrix under pure-step initial data condition. By observing the scattering matrix in this case, we propose two assumptions about the zero point of the scattering data as case1 and case2. After that, we calculate the expressions for zero point $\xi_{1}$ in two cases. In the inverse scattering part, we construct the RH problem, and obtain the solution of LPD equation by using the solution of RH problem. Secondly, we decompose the jump matrix $J(x,t,\xi)$ into the matrices consist by the upper triangle and lower triangle. To get ride of the intermediate matrix, we introduce the $\delta$ function. Then, we perform the second RH deformation to transform the contour and make the jump matrices decline to identity $I$ for the large-$t$. After that, the BP factor is introduced to transform the RH problem into a regular RH problem. Next, by using the BP matrix and the solution of regular RH problem, we construct the solution of LPD equation and the rough estimate of it as $t\rightarrow\infty$. Then, the regular RH problem can be solved by the parabolic cylinder functions. Through the Beals-coifman theory, we gain the error analysis of regular RH problem. Finally, the long-time asymptotics of the solutions of LPD equation at cases $x>0$ and $x<0$ are attained, respectively.

\section*{Acknowledgements}
This work was supported by the National Natural Science Foundation of China under Grant No. 11975306, the Natural Science Foundation of Jiangsu Province under Grant No. BK20181351, the Six Talent Peaks Project in Jiangsu Province under Grant No. JY-059, the 333 Project in Jiangsu Province, the Fundamental Research Fund for the Central Universities under the Grant No. 2019ZDPY07, and funded  by the Graduate  Innovation  Program of China University of Mining and Technology under Grant No. 2023WLJCRCZL142.

\section{Appendix A: The parabolic cylinder model problem}
This Appendix is to solve the local RH problem at the three saddle points $\lambda_{j}$, $j=1,2,3$. Taking $\lambda_{1}$ as example, we consider the following parabolic cylinder model RH problem.

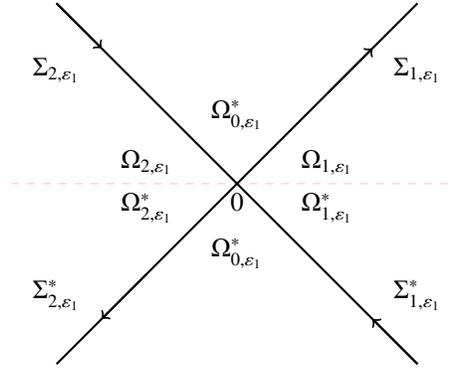
\begin{figure}[H]
      \centering
\begin{tikzpicture}[scale=0.6]
\draw[-][pink,dashed](-5,0)--(5,0);
%\draw[-][dashed](-3,0)--(-2,0);
%\draw[-][dashed](-2,0)--(-1,0);
%\draw[-][dashed](-1,0)--(0,0);
%\draw[-][dashed](0,0)--(1,0);
%\draw[-][dashed](1,0)--(2,0);
%\draw[-][dashed](2,0)--(3,0);
%\draw[-][dashed](3,0)--(4,0);
%\draw[-][thick](0,4)--(0,3);
%\draw[-][thick](0,3)--(0,2);
%\draw[-][thick](0,2)--(0,1);
%\draw[-][thick](0,1)--(0,0);
%\draw[-][thick](0,-4)--(0,-3);
%\draw[-][thick](0,-3)--(0,-2);
%\draw[-][thick](0,-2)--(0,-1);
%\draw[-][thick](0,-1)--(0,-0);
\draw[-][thick](-4,-4)--(4,4);
\draw[-][thick](-4,4)--(4,-4);
\draw[->][thick](2,2)--(3,3);
\draw[->][thick](-4,4)--(-3,3);
\draw[->][thick](-2,-2)--(-3,-3);
\draw[->][thick](4,-4)--(3,-3);
\draw[fill] (4,3)node[below]{$\Sigma_{1,\varepsilon_{1}}$};
\draw[fill] (4,-3)node[above]{$\Sigma^{*}_{1,\varepsilon_{1}}$};
\draw[fill] (-4,3)node[below]{$\Sigma_{2,\varepsilon_{1}}$};
\draw[fill] (-4,-3)node[above]{$\Sigma^{*}_{2,\varepsilon_{1}}$};
\draw[fill] (0,0)node[below]{$0$};
\draw[fill] (2,0)node[below]{$\Omega^{*}_{1,\varepsilon_{1}}$};
\draw[fill] (2,0)node[above]{$\Omega_{1,\varepsilon_{1}}$};
\draw[fill] (0,-1)node[below]{$\Omega^{*}_{0,\varepsilon_{1}}$};
\draw[fill] (0,1)node[above]{$\Omega^{*}_{0,\varepsilon_{1}}$};
\draw[fill] (-2,0)node[below]{$\Omega^{*}_{2,\varepsilon_{1}}$};
\draw[fill] (-2,0)node[above]{$\Omega_{2,\varepsilon_{1}}$};
\end{tikzpicture}
     \caption{ \footnotesize The contours and domains of jump matrix $J^{pc}_{\lambda_{1}}(\tau)$.}
\end{figure}

\begin{RHP}\label{PC-model}
Find a matrix-value function $\widehat{m}^{pc}_{\lambda_{1}}(\tau):=\widehat{m}^{pc}_{\lambda_{1}}(\lambda_{1},\tau)$
such that
\begin{enumerate}[(i)]
\item $\widehat{m}^{pc}_{\lambda_{1}}(\tau)$ is analytic in $\mathbb{C}\backslash\Sigma^{pc}_{\lambda_{1}}$.
\item Jump conditions:
\begin{align}\label{A1}
\widehat{m}^{pc}_{\lambda_{1},+}(\tau)=
         \widehat{m}^{pc}_{\lambda_{1},-}(\tau)J^{pc}_{\lambda_{1}}(\tau),
         \quad \tau\in\Sigma^{pc}_{\lambda_{1}},
\end{align}
with
\begin{align}\label{A2}
J^{pc}_{\lambda_{1}}(\tau)=\left\{ \begin{aligned}
         &\left(
         \begin{array}{cc}
         1 & -\frac{r_{2}^{r}(\lambda_{1})}{1+r_{1}^{r}(\lambda_{1})r_{2}^{r}(\lambda_{1})}e^{-\frac{i}{2}\tau^{2}}\tau^{2iv(\lambda_{1})}\\
         0 & 1 \\
         \end{array}
         \right)
         ,& \quad &\tau\in \Sigma_{2,\varepsilon_{1}},\\
         &\left(
         \begin{array}{cc}
         1 & 0 \\
         -r_{1}^{r}(\lambda_{1})e^{\frac{i}{2}\tau^{2}}\tau^{-2iv(\lambda_{1})} & 1 \\
         \end{array}
         \right),& \quad &\tau\in \Sigma_{1,\varepsilon_{1}},\\
         &\left(
         \begin{array}{cc}
         1 & r_{2}^{r}(\lambda_{1})e^{-\frac{i}{2}\tau^{2}}\tau^{2iv(\lambda_{1})} \\
         0 & 1 \\
         \end{array}
         \right), & \quad &\tau\in \Sigma^{*}_{1,\varepsilon_{1}},\\
         &\left(
         \begin{array}{cc}
         1 & 0 \\
         \frac{r_{1}^{r}(\lambda_{1})}{1+r_{1}^{r}(\lambda_{1})r_{2}^{r}(\lambda_{1})}
         e^{\frac{i}{2}\tau^{2}}\tau^{-2iv(\lambda_{1})} & 1 \\
         \end{array}
         \right),& \quad &\tau\in \Sigma^{*}_{2,\varepsilon_{1}}.\\
                          \end{aligned} \right.
\end{align}
\item Normalization condition at $\tau=\infty$:
\begin{align}\label{A3}
\widehat{m}^{pc}_{\lambda_{1}}(\tau)=I+
\frac{(\widehat{m}^{pc}_{\lambda_{1}})_{1}}{\tau}+
O(\frac{1}{\tau^{2}}), \quad \tau\rightarrow\infty.
\end{align}
\end{enumerate}
\end{RHP}

The jump contours and domains of jump matrix $J^{pc}_{\lambda_{1}}(\tau)$ are shown in Figure 7. As we know, the parabolic cylinder model RH problem $\widehat{m}^{pc}_{\lambda_{1}}(\tau)$ has the form of Webber equation
\begin{align*}
\left(\frac{\partial^{2}}{\partial \tau^{2}}+(\frac{1}{2}-\frac{\tau^{2}}{2}+a)\right)D_{a}(\tau)=0.
\end{align*}

Then we make the transformation to obtain the explicit solution of $\widehat{m}^{pc}_{\lambda_{1}}(\tau)$
\begin{align}\label{A4}
\widehat{m}^{pc}_{\lambda_{1}}(\tau)=m_{\lambda_{1}}(\tau)\mathcal {P}\tau^{-iv(\lambda_{1})\sigma_3}e^{\frac{i}{4}\tau^2\sigma_3},
\end{align}
where
\begin{align*}
\mathcal{P}=\left\{\begin{aligned}
&\left(
                    \begin{array}{cc}
                      1 & 0 \\
         -r^{r}_{1}(\lambda_{1}) & 1 \\
                    \end{array}
                  \right),& \quad &\tau\in\Omega_{1},\\
&\left(
                    \begin{array}{cc}
                      1 & -\frac{r^{r}_{2}(\lambda_{1})}{1+r^{r}_{1}(\lambda_{1})r^{r}_{2}(\lambda_{1})} \\
                    0 & 1 \\
                    \end{array}
                  \right),& \quad &\tau\in\Omega_{2},\\
&\left(
                    \begin{array}{cc}
                      1 & 0 \\
         \frac{r^{r}_{1}(\lambda_{1})}{1+r^{r}_{1}(\lambda_{1})r^{r}_{2}(\lambda_{1})} & 1 \\
                    \end{array}
                  \right),& \quad &\tau\in\Omega^{*}_{2},\\
&\left(
                    \begin{array}{cc}
                     1 & r^{r}_{2}(\lambda_{1}) \\
                     0 & 1 \\
                    \end{array}
                  \right),&\quad &\tau\in\Omega^{*}_{1},\\
&~~~\mathbb{I},& \quad &\tau\in\Omega_{0}\cup\Omega^{*}_{0}.
\end{aligned}\right.
\end{align*}
The $2\times2$ matrix-valued function $m_{\lambda_{1}}(\tau)$ admits the following RH problem
\begin{RHP}\label{RHPA2}
Find a matrix-value function $m_{\lambda_{1}}(\tau)$
such that
\begin{enumerate}[(i)]
\item $m_{\lambda_{1}}(\tau)$ is analytic in $\mathbb{C}\backslash \mathbb{R}$.
\item Jump condition:
\begin{align}\label{A5}
m_{\lambda_{1},+}(\tau)=
         m_{\lambda_{1},-}(\tau)J_{1}(\lambda_{1}),
         \quad \tau\in\mathbb{R},
\end{align}
with
\begin{align}\label{A6}
J_{1}(\lambda_{1})=\left(
         \begin{array}{cc}
         1+r^{r}_{1}(\lambda_{1})r^{r}_{2}(\lambda_{1}) & -r^{r}_{2}(\lambda_{1}) \\
         -r^{r}_{1}(\lambda_{1}) & 1 \\
         \end{array}
         \right).
\end{align}
\item Asymptotic behavior:
\begin{align}\label{A7}
m_{\lambda_{1}}(\tau)&=\left(I+
\frac{(\widehat{m}^{pc}_{\lambda_{1}})_{1}}{\tau}+
O(\frac{1}{\tau^{2}})\right)
e^{-\frac{i}{4}\tau^{2}\sigma_{3}}\tau^{iv(\lambda_{1})\sigma_{3}},\quad\tau\rightarrow\infty.
\end{align}
\end{enumerate}
\end{RHP}
Using the fact that $\frac{i}{2}\tau\sigma_3 m_{\lambda_{1},+}=\frac{i}{2}\tau\sigma_3 m_{\lambda_{1},-}J_{1}(\lambda_{1})$, differentiating \eqref{A5} with respect to $\tau$ yields
\begin{align}\label{A8}
\left(\frac{d m_{\lambda_{1}}}{d\tau}+\frac{i}{2}\tau\sigma_3m_{\lambda_{1}}\right)_{+}=
\left(\frac{d m_{\lambda_{1}}}{d\tau}+\frac{i}{2}\tau\sigma_3m_{\lambda_{1}}\right)_{-}J_{1}(\lambda_{1}).
\end{align}
The condition $\det J_{1}(\lambda_{1})=1$ indicates that $\det m_{\lambda_{1},+}=\det m_{\lambda_{1},-}$. By using the painlev\'{e} extension theorem, it can be verified that $\left(\frac{d m_{\lambda_{1}}}{d\tau}+\frac{i}{2}\tau\sigma_3 m_{\lambda_{1}}\right)m_{\lambda_{1}}^{-1}$ is analytic in the whole plane. Taking \eqref{A4} into account, after direct calculation, we have
\begin{align}\label{A9}
\left(\frac{d m_{\lambda_{1}}}{d\tau}+\frac{i}{2}\tau\sigma_3 m_{\lambda_{1}}\right)m_{\lambda_{1}}^{-1}=\left(\frac{d \widehat{m}^{pc}_{\lambda_{1}}}{d\tau}+\widehat{m}^{pc}_{\lambda_{1}}\frac{iv(\lambda_{1})\sigma_3}{\tau}
\right)(\widehat{m}^{pc}_{\lambda_{1}})^{-1}+\frac{i}{2}\tau[\sigma_3,(\widehat{m}^{pc}_{\lambda_{1}})_{1}]
(\widehat{m}^{pc}_{\lambda_{1}})^{-1}.
\end{align}
By the Liouville theorem, we obtain that $\left(\frac{d m_{\lambda_{1}}}{d\tau}+\frac{i}{2}\tau\sigma_3 m_{\lambda_{1}}\right)m_{\lambda_{1}}^{-1}$ is a constant matrix, then there exists a constant matrix $B$ such that
\begin{align}\label{A10}
B=\frac{i}{2}\tau[\sigma_3,(\widehat{m}^{pc}_{\lambda_{1}})_{1}]=\left(
         \begin{array}{cc}
         0 & i(\widehat{m}^{pc}_{\lambda_{1}})^{12}_{1} \\
         -i(\widehat{m}^{pc}_{\lambda_{1}})^{21}_{1} & 0 \\
         \end{array}
         \right)=
\left(
         \begin{array}{cc}
         0 & \beta^{r}_{1}(\lambda_{1}) \\
         \gamma^{r}_{1}(\lambda_{1}) & 0 \\
         \end{array}
         \right).
\end{align}
Then we have
\begin{align}\label{A11}
\frac{d m_{\lambda_{1}}}{d\tau}+\frac{i}{2}\tau\sigma_3 m_{\lambda_{1}}=B m_{\lambda_{1}},
\end{align}
expending the above equation on the upper half-plane, after calculation, it can be concluded that
\begin{subequations}\label{A12}
\begin{align}
&(m_{\lambda_{1}})_{11}^{''}+\left(\frac{i}{2}+\frac{\tau^2}{4}-\gamma^{r}_{1}
\beta^{r}_{1}\right)(m_{\lambda_{1}})_{11}=0,~~~
(m_{\lambda_{1}})_{21}^{''}+\left(-\frac{i}{2}+\frac{\tau^2}{4}-\gamma^{r}_{1}
\beta^{r}_{1}\right)(m_{\lambda_{1}})_{21}=0,\label{A12a}\\
&(m_{\lambda_{1}})_{12}^{''}+\left(\frac{i}{2}+\frac{\tau^2}{4}-\gamma^{r}_{1}
\beta^{r}_{1}\right)(m_{\lambda_{1}})_{12}=0,~~~
(m_{\lambda_{1}})_{22}^{''}+\left(-\frac{i}{2}+\frac{\tau^2}{4}-\gamma^{r}_{1}
\beta^{r}_{1}\right)(m_{\lambda_{1}})_{22}=0.\label{A12b}
\end{align}
\end{subequations}
There we set $a=i\beta^{r}_{1}\gamma^{r}_{1}$, and introduce a new variable $\zeta=\tau e^{-\frac{3\pi i}{4}}$. Let $(m_{\lambda_{1}})_{11}(\tau)=g(\tau e^{-\frac{3\pi i}{4}})$, the equations \eqref{A12} can be written as parabolic cylinder equation reads
\begin{align}\label{A13}
g^{''}(\zeta)+\left(\frac{i}{2}-\frac{\zeta^2}{4}+a\right)g(\zeta)=0.
\end{align}

Thus, in the upper plane, for $0<\arg \tau<\pi$, we have $-\frac{3\pi}{4}<\arg \zeta<\frac{\pi}{4}$. The solution of $m_{\lambda_{1}}$ as $\im \tau>0$ can be expressed as follows
\begin{align}\label{A14}
m_{\lambda_{1}}(\tau)=\left(
              \begin{array}{cc}
                e^{-\frac{3\pi}{4}v(\lambda_1)}D_{iv(\lambda_1)}(\tau e^{-\frac{3\pi i}{4}}) &
                -\frac{iv(\lambda_1)}{\gamma^{r}_1(\lambda_1)}e^{\frac{\pi}{4}(v(\lambda_1)-i)}
                D_{-iv(\lambda_1)-1}(\tau e^{-\frac{\pi i}{4}}) \\
                \frac{iv(\lambda_1)}{\beta^{r}_1(\lambda_1)}e^{-\frac{3\pi}{4}(v(\lambda_1)+i)}
                D_{iv(\lambda_1)-1}(\tau e^{-\frac{3\pi i}{4}}) &
                 e^{\frac{\pi}{4}v(\lambda_1)}D_{-iv(\lambda_1)}(\tau e^{-\frac{\pi i}{4}}) \\
              \end{array}
            \right),
\end{align}
as $\im \tau>0$, we have
\begin{align}\label{A15}
m_{\lambda_{1}}(\tau)=\left(
              \begin{array}{cc}
                e^{\frac{\pi}{4}v(\lambda_1)}D_{iv(\lambda_1)}(\tau e^{\frac{\pi i}{4}}) &
                -\frac{iv(\lambda_1)}{\gamma^{r}_1(\lambda_1)} e^{-\frac{3\pi}{4}(v(\lambda_1)-i)}
                D_{-iv(\lambda_1)-1}(\tau e^{\frac{3\pi i}{4}}) \\
                \frac{iv(\lambda_1)}{\beta^{r}_1(\lambda_1)}e^{\frac{\pi}{4}(v(\lambda_1)+i)}
                D_{iv(\lambda_1)-1}(\tau e^{\frac{\pi i}{4}}) &
                 e^{-\frac{3\pi}{4}v(\lambda_1)}D_{-iv(\lambda_1)}(\tau e^{\frac{3\pi i}{4}}) \\
              \end{array}
            \right).
\end{align}
According to \eqref{A5}, we have
\begin{align}\label{A16}
\begin{split}
-r^{r}_{1}(\lambda_1)&=(m_{\lambda_{1}})^{-}_{11}(m_{\lambda_{1}})^{+}_{21}
-(m_{\lambda_{1}})^{+}_{11}(m_{\lambda_{1}})^{-}_{21}\\
&=e^{\frac{\pi}{4}v(\lambda_1)}D_{iv(\lambda_1)}(\tau e^{-\frac{3\pi i}{4}})\frac{e^{-\frac{3\pi v(\lambda_1)}{4}}}{\beta^{r}_1(\lambda_1)}\left[\partial_{\tau}\left(D_{iv(\lambda_1)}(\tau e^{-\frac{3\pi i}{4}})\right)+\frac{i\tau}{2}D_{iv(\lambda_1)}(\tau e^{-\frac{3\pi i}{4}})\right]  \\
&-e^{-\frac{3\pi v(\lambda_1)}{4}}D_{iv(\lambda_1)}(\tau e^{-\frac{3\pi i}{4}})
\frac{e^{\frac{\pi v(\lambda_1)}{4}}}{\beta^{r}_1(\lambda_1)}\left[\partial_{\tau} \left(D_{iv(\lambda_1)}(\tau e^{\frac{\pi i}{4}})\right)+\frac{i\tau}{2}D_{iv(\lambda_1)}(\tau e^{\frac{\pi i}{4}})\right] \\
&=\frac{e^{-\frac{\pi v(\lambda_1)}{2}}}{\beta^{r}_1(\lambda_1)}\frac{\sqrt{2\pi}e^{\frac{\pi i}{4}}
}{\Gamma(-iv(\lambda_1))},
\end{split}
\end{align}
\begin{align}\label{A17}
\begin{split}
-r^{r}_{2}(\lambda_1)&=(m_{\lambda_{1}})^{-}_{22}(m_{\lambda_{1}})^{+}_{12}
-(m_{\lambda_{1}})^{+}_{22}(m_{\lambda_{1}})^{-}_{12}\\
&=e^{-\frac{3\pi}{4}v(\lambda_1)}D_{-iv(\lambda_1)}(\tau e^{\frac{3\pi i}{4}})\frac{e^{\frac{\pi v(\lambda_1)}{4}}}{\gamma^{r}_1(\lambda_1)}\left[\partial_{\tau}\left(D_{-iv(\lambda_1)}(\tau e^{-\frac{\pi i}{4}})\right)-\frac{i\tau}{2}D_{-iv(\lambda_1)}(\tau e^{-\frac{\pi i}{4}})\right]  \\
&-e^{\frac{\pi v(\lambda_1)}{4}}D_{-iv(\lambda_1)}(\tau e^{-\frac{\pi i}{4}})
\frac{e^{-\frac{3\pi v(\lambda_1)}{4}}}{\gamma^{r}_1(\lambda_1)}\left[\partial_{\tau} \left(D_{-iv(\lambda_1)}(\tau e^{\frac{3\pi i}{4}})\right)-\frac{i\tau}{2}D_{-iv(\lambda_1)}(\tau e^{\frac{3\pi i}{4}})\right] \\
&=\frac{e^{-\frac{\pi v(\lambda_1)}{2}}}{\gamma^{r}_1(\lambda_1)}\frac{\sqrt{2\pi}e^{-\frac{\pi i}{4}}
}{\Gamma(iv(\lambda_1))}.
\end{split}
\end{align}
Then
\begin{align}\label{A18}
\beta^{r}_1(\lambda_1)=-\frac{\sqrt{2\pi}e^{-\frac{\pi v(\lambda_1)}{2}}e^{\frac{\pi i}{4}}}{r^{r}_{1}(\lambda_1)\Gamma(-iv(\lambda_1))},\quad
\gamma^{r}_1(\lambda_1)=-\frac{\sqrt{2\pi}e^{-\frac{\pi v(\lambda_1)}{2}}e^{-\frac{\pi i}{4}}}{r^{r}_{2}(\lambda_1)\Gamma(iv(\lambda_1))},
\end{align}
there $\Gamma(\cdot)$ denotes the gamma function, and
\begin{align}\label{A19}
\widehat{m}^{pc}_{\lambda_{1}}(\tau)=I+\frac{i}{\tau}
\left(
         \begin{array}{cc}
         0 & \beta^{r}_{1}(\lambda_{1}) \\
         -\gamma^{r}_{1}(\lambda_{1}) & 0 \\
         \end{array}
         \right)+O(\tau^{-2}), \quad \tau\rightarrow\infty.
\end{align}

\renewcommand{\baselinestretch}{1.2}


\begin{thebibliography}{00}\addtolength{\itemsep}{-1.5ex}

\bibitem{Biondini-2014}
G. Biondini, G. Kova\u{c}i\u{c}, Inverse scattering transform for the focusing nonlinear Schr\"{o}dinger equation with nonzero boundary conditions, J. Math. Phys. \textbf{55}(3) (2014), 031506.

\bibitem{Ji-Zhu-2017}
J. L. Ji, Z. N. Zhu, On a nonlocal modified Korteweg-de Vries equation: integrability, Darboux transformation and soliton solutions, Commun. Nonlinear Sci. Numer. Simul. \textbf{42} (2017), 699-708.

\bibitem{Geng-2018}
H. Liu, X. G. Geng, B. Xue, The Deift-Zhou steepest descent method to long-time asymptotics for the Sasa-Satsuma equation, J. Differ. Equ. \textbf{265}(11) (2018), 5984-6008.

\bibitem{Clarkson-1990}
P. A. Clarkson, J. A. Tuszynski, Exact solutions of the multidimensional derivative nonlinear Schr\"{o}dinger equation for many-body systems of criticality, J. Phys. A: Math. Gen. \textbf{23}(19) (1990), 4269-4288.


\bibitem{Brizhik-2003}
L. Brizhik, A. Eremko, B. Piette, W. J. Zakrzewski, Static solutions of a D-dimensional modified nonlinear Schr\"{o}dinger equation, Nonlinearity, \textbf{16} (2003), 1481-1497.


\bibitem{Tian-JDE-2017}
S. F. Tian, Initial-boundary value problems for the general coupled nonlinear Schr\"{o}dinger equation on the interval via the Fokas method, J. Differ. Equ. \textbf{262}(1) (2017), 506-558.

\bibitem{Tian-PAMS-2018}
S. T. Tian, T. T. Zhang, Long-time asymptotic behavior for the Gerdjikov-Ivanov type of derivative nonlinear Schr\"{o}dinger equation with time-periodic boundary condition, Proc. Am. Math. Soc. \textbf{146} (2018), 1713-1729.

\bibitem{LPD-1988}
M. Lakshmanan, K. Porsezian, M. Daniel, Effect of discreteness on the continuum limit of the Heisenberg spin chain, Phys. Lett. A, \textbf{133} (1988), 483-488.

\bibitem{Ablowitz-Musslimani-2013}
M. J. Ablowitz, Z. H. Musslimani, Integrable nonlocal nonlinear Schr\"{o}dinger equation, Phys. Rev. Lett. \textbf{110} (2013), 064105.


\bibitem{Bender-Boettcher-1998}
C. M. Bender, S. Boettcher, Real Spectra in Non-Hermitian Hamiltonians Having PT Symmetry, Phys. Rev. Lett. \textbf{80}(24) (1998), 5243-5246.

\bibitem{Fokas-2016}
A. S. Fokas, Integrable multidimensional versions of the nonlocal nonlinear Schr\"{o}dinger equation, Nonlinearity, \textbf{29} (2016), 319-324.

\bibitem{Liang-Zhou-2017}
J. L. Liang, Z. N. Zhu, Soliton solutions of an integrable nonlocal modified Korteweg-de Vries equation through inverse scattering transform, J. Math. Anal. Appl. \textbf{453}(2)  (2017), 973-984.

\bibitem{Ablowitz-Feng-2018}
M. J. Ablowitz, B. F. Feng, X. D. Luo, Reverse space-time nonlocal Sine-Gordon/Sinh-Gordon equations with nonzero boundary conditions, Stud. Appl. Math. \textbf{141}(3) (2018), 267-307.

\bibitem{Heredero-Reyes-2009}
R. Hern\'{a}ndez-Heredero, E. G. Reyes, Nonlocal symmetries and a Darboux transformation for the Camassa-Holm equation, J. Phys. A: Math. Theor. \textbf{42}(18) (2009), 182002.

\bibitem{Zhang-2009}
H. Q. Zhang, B. Tian, X. H. Meng, X. L\"{u}, W. J. Liu, Conservation laws, soliton solutions and modulational instability for the higher-order dispersive nonlinear Schr\"{o}dinger equation, Eur. Phys. J. B \textbf{72} (2009), 233-239.

\bibitem{Guo-2013}
R. Guo, H. Q. Hao, Breathers and multi-soliton solutions for the higher-order generalized nonlinear Schr\"{o}dinger equation, Commun. Nonlinear Sci. Numer. Simul. \textbf{18} (2013), 2426-2435.

\bibitem{Wang-2012}
X. L. Wang, W. G. Zhang, B. G. Zhai, H. Q. Zhai, Rogue waves of the higher-order dispersive nonlinear Schr\"{o}dinger equation, Commun. Theor. Phys. \textbf{58} (2012), 531-538.

\bibitem{Liu-2016}
W. Liu, D. Q. Qiu, Z. W. Wu, J. S. He, Dynamical behavior of solution in integrable nonlocal Lakshmanan-Porsezian-Daniel equation, Commun. Theor. Phys. \textbf{65}(6) (2016),  671-676.

\bibitem{Yang-2020}
Y. Q. Yang, T. Suzuki, X. P. Cheng, Darboux transformations and exact solutions for the integrable nonlocal Lakshmanan-Porsezian-Daniel equation, Appl. Math. Lett. \textbf{99} (2020), 105998.

\bibitem{Wang-Liu-2020}
Y. F. Wang, B. L. Guo, N. Liu, Riemann-Hilbert problem for a fourth-order dispersive nonlinear Schr\"{o}dinger equation on the half-line, J. Math. Anal. Appl. \textbf{488} (2020), 124078.

\bibitem{Xun-2020}
W. K. Xun, S. F. Tian, Inverse scattering transform for the integrable nonlocal Lakshmanan-Porsezian-Daniel equation, arXiv: 2005.04011.

\bibitem{Wang-Liu-2022}
Y. F. Wang, N. Liu, B. L. Guo, Long-time asymptotic behavior for a fourth-order dispersive nonlinear Schr\"{o}dinger equation, J. Math. Anal. Appl. \textbf{506} (2022), 125560.

\bibitem{Peng-Chen-2023}
W. Q. Peng, Y. Chen, Long-time asymptotics for the integrable nonlocal Lakshmanan-Porsezian-Daniel equation with decaying initial value problem, arXiv: 2305.05926.

\bibitem{Manakov-1974}
S. V. Manakov, Nonlinear Fraunhofer diffraction, Sov. Phys. JETP, \textbf{38} (1974), 693-696.

\bibitem{Zakharov-Manakov-1976}
V. E. Zakharov, S. V. Manakov, Asymptotic behavior of nonlinear wave systems integrated by the inverse scattering method, Sov. Phys. JETP, \textbf{44} (1976), 106-112.

\bibitem{Deift-Zhou-1993}
X. Zhou, P. Deift, A steepest descent method for oscillatory Riemann-Hilbert problems, Ann. Math. \textbf{137} (1993), 295-368.

\bibitem{Xu-Fan-2015}
J. Xu, E. G. Fan, Long-time asymptotics for the Fokas-Lenells equation with decaying
initial value problem: without solitons, J. Differ. Equ. \textbf{259}(3) (2015), 1098-1148.

\bibitem{Rybalko-2019}
Y. Rybalko, D. Shepelsky, Long-time asymptotics for the integrable nonlocal nonlinear Schr\"{o}dinger equation, J. Math. Phys. \textbf{60} (2019), 031504.



\bibitem{Chen-2019}
S. Y. Chen, Z. Y. Yan, The higher-order nonlinear Schr\"{o}dinger equation with non-zero boundary conditions: Robust inverse scattering transform, breathers, and rogons, Phys. Lett. A, \textbf{383}(29) (2019), 125906.

\bibitem{Rybalko-JDE-2021}
Y. Rybalko, D. Shepelsky, Long-time asymptotics for the nonlocal nonlinear Schr\"{o}dinger equation with step-like initial data, J. Differ. Equ. \textbf{270} (2021), 694-724.


\bibitem{Rybalko-JMP-2019}
Y. Rybalko, D. Shepelsky, Long-time asymptotics for the integrable nonlocal nonlinear Schr\"{o}dinger equation, J. Math. phys. \textbf{60} (2019), 031504.

\bibitem{Rybalko-CMP-2021}
Y. Rybalko, D. Shepelsky, Long-time asymptotics for the integrable nonlocal focusing nonlinear Schr\"{o}dinger equation for a family of step-like initial data, Commun. Math. Phys. \textbf{382} (2021), 87-121.

\bibitem{Rybalko-SAPM-2021}
Y. Rybalko, D. Shepelsky, Curved wedges in the long-time asymptotics for the integrable nonlocal nonlinear Schr\"{o}dinger equation, Stud. App. Math. \textbf{3} (2021), 147-178.

\bibitem{Rybalko-JMPAG-2021}
Y. Rybalko, D. Shepelsky, Defocusing nonlocal nonlinear Schr\"{o}dinger equation with step-like boundary conditions: long-time behavior for shifted initial data, J. Math. Phys. Anal. Geom. \textbf{16}(4) (2020), 418-453.

\bibitem{Xu-Fan-2023}
T. Y. Xu, E. G. Fan, Large-time asymptotics to the focusing nonlocal modified Kortweg-de Vries equation with step-like boundary conditions, Stud. App. Math. \textbf{150}(4) (2023), 1217-1273.

\bibitem{Minakov-2011}
A. Minakov, Long-time behavior of the solution to the mKdV equation with step-like initial data, J. Phys. A: Math. Theor. \textbf{44}(8) (2011), 276-281.

\bibitem{Grava-Minakov-2020}
T. Grava, A. Minakov, On the long-time asymptotic behavior of the modified Kortweg-de Vries equation with step-like initial data, SIAM J. Math. Anal. \textbf{52} (2020), 5892-5993.

\bibitem{Liu-2021}
L. Liu, W. Zhang, On a Riemann-Hilbert problem for the focusing nonlocal mKdV equation with step-like initial data, Appl. Math. Lett. \textbf{116} (2021), 107009.

\bibitem{Xu-Fan-2022}
T. Y. Xu, E. G. Fan, On the Cauchy problem of defocusing mKdV equation: long-time asymptotics under step-like initial data, arXiv: 2204.01299.

\bibitem{Minakov-2015}
A. Minakov, Riemann-Hilbert problem for Camassa-Holm equation with step-like initial data, J. Math. Anal. Appl. \textbf{429}(1) (2015), 81-104.

\bibitem{Yang-Fan-2022}
Y. L. Yang, G. Z. Li, E. G. Fan, On the long-time asymptotics of the modified Camassa-Holm equation under step-like background, arXiv: 2203.10573.

\bibitem{Minakov-2016}
A. Minakov, Asymptotics of step-like solutions for the Camassa-Holm equation, J. Differ. Equ. \textbf{261}(11) (2016), 6055-6098.

\bibitem{Xu-Fan-2013}
J. Xu, E. G. Fan, Y. Chen, Long-time asymptotic for the derivative nonlinear schr\"{o}dinger equation with step-like initial value, Math. Phys. Anal. Geom. \textbf{16}(3) (2013), 253-288.

\end{thebibliography}
\end{document}